\documentclass[11pt]{amsart}
\usepackage{a4wide}
\usepackage{comment}
\usepackage[dvipsnames]{xcolor}
\usepackage{mathrsfs}
\usepackage{amsmath}
\usepackage{amsthm}
\usepackage{amsfonts}
\usepackage{amssymb}
\usepackage{enumitem}
\usepackage{graphicx}
\usepackage{url}
\usepackage{thmtools, thm-restate}
\usepackage{tikz}
\usepackage{tikz-cd}
\usepackage{palatino}
\usepackage{overpic}
\usepackage[normalem]{ulem}
\usepackage[T1]{fontenc}
\usepackage{hyperref}
\usepackage[noabbrev,capitalize,nameinlink]{cleveref}
\newtheorem{theorem}{Theorem}[section]
\newtheorem{proposition}[theorem]{Proposition}
\newtheorem{lemma}[theorem]{Lemma}
\newtheorem{corollary}[theorem]{Corollary}

\newtheorem{problem}[theorem]{Problem}

\theoremstyle{definition}
\newtheorem{definition}[theorem]{Definition}
\newtheorem{example}[theorem]{Example}

\theoremstyle{remark}
\newtheorem{remark}[theorem]{Remark}
\numberwithin{equation}{section}

\usepackage{mleftright}

\newcommand{\R}{\mathbb{R}}

\newcommand{\D}{\mathbb{D}}

\newcommand{\ve}{\varepsilon}

\definecolor{darkgreen}{RGB}{0,153,0}
\definecolor{darkred}{RGB}{204,0,0}
\definecolor{darkblue}{RGB}{0, 51,204}
\definecolor{gold}{RGB}{255,205,0}
\definecolor{red}{RGB}{242,43,29}
\hypersetup{colorlinks,linkcolor={blue},citecolor={darkgreen},urlcolor={darkgreen}}

\usetikzlibrary{knots}
\usetikzlibrary{decorations.pathmorphing} 

\begin{document}

\title{Bypass moves in convex hypersurface theory}

\author{Joseph Breen}
\address{University of Alabama, Tuscaloosa, AL 35401}
\email{jjbreen@ua.edu} \urladdr{https://sites.google.com/view/joseph-breen}

\author{Austin Christian}
\address{California Polytechnic State University, San Luis Obispo, CA 93407}
\email{achris66@calpoly.edu} \urladdr{https://sites.google.com/view/austin-christian}

\thanks{JB was partially supported by NSF Grant DMS-2038103 and an AMS-Simons Travel Grant. AC was partially supported by NSF Grants DMS-1745583 and DMS-2532551 and an AMS-Simons Travel Grant.}

\begin{abstract}
We construct bypass attachments in higher dimensional contact manifolds that, when attached to a neighborhood of a Weinstein hypersurface, yield a neighborhood of a new Weinstein hypersurface, obtained via local modifications to the Weinstein handle decomposition of the first. For context, we give $3$-dimensional analogues of these bypass attachments and discuss their appearance in nature. We then show that our bypass attachments give a necessary and sufficient set of moves relating any two Weinstein domains which become almost symplectomorphic after one stabilization.  Finally, we use our construction to produce several examples of interesting convex hypersurfaces and recover an existence $h$-principle for Weinstein hypersurfaces.
\end{abstract}

\maketitle

\tableofcontents

\section{Introduction}\label{sec:introduction}

\subsection{Context} Convex hypersurface theory, originally due to Giroux \cite{giroux1991convexite}, is a powerful tool in contact topology. It has been especially useful in dimension $3$.  For example, Giroux \cite{giroux2000structures} used it to recover the classification results of Bennequin \cite{bennequin1983entrelacements} and Eliashberg \cite{eliashberg1992contact}, while he \cite{giroux2000structures} and others like Etnyre \cite{etnyre2000tight}, Honda \cite{honda2000classification,honda2000classification2}, and Kanda \cite{kanda1997classification} obtained further classifications of tight contact structures. Shortly after, Etnyre-Honda \cite{etnyre2001knots,etnyre2005cabling} used convex surface theory to produce classifications of Legendrian and transverse knots, introducing techniques that have since been a fixture of the literature (e.g. \cite{tosun2012legendrian,etnyre2018satellites,dalton2024legendrian}). Convex surface theory in dimension $3$ is fundamentally combinatorial, and discrete changes are measured by \textit{bypass attachments}. 

Although Giroux introduced convex hypersurface theory in all dimensions, its utility beyond dimension $3$ has not been clear until recently. The work of Honda-Huang in \cite{HH18, honda2019convex} (see also \cite{sackel2019handle, eliashberg2022hondahuangs}) established necessary technical foundations: a $C^0$-genericity result, a $1$-parametric genericity result, and the concept of a bypass attachment. Using this technology, they recovered various theoretical results, such as existence of supporting (Weinstein) open book decompositions. Since then, additional theoretical advancements have been made. For example, together with the first author in \cite{BHH23}, Honda-Huang established the Giroux correspondence between contact structures and supporting Weinstein open book decompositions, more or less extending the $3$-dimensional convex methods of Giroux \cite{Giroux2002GeometrieDC} to all dimensions. 

Despite being demonstrably theoretically powerful, the practical use of convex hypersurface theory in higher dimensional contact topology is in its infancy. The goal is to eventually use it to study classification problems as in dimension $3$, but there remain technical obstacles to address before this happens --- for instance, lack of a "Giroux criterion" for tightness. Additionally, we need a better understanding of the combinatorics of bypass attachments in order to use them effectively.  

The broad goal of this paper is to improve the utility of the bypass attachment as a practical tool in all dimensions. Specifically, we: 
\begin{enumerate}
    \item Describe a set of bypass attachments that perform local, combinatorial changes to embedded Weinstein hypersurfaces in dimension at least $5$.
    \item Describe the analogous bypass moves in dimension $3$, and identify numerous phenomena they witness.

    \item Use the bypass moves in dimension at least $5$ to characterize Weinstein domains up to almost symplectomorphism after one stabilization,\footnote{There are unfortunately at least three distinct meanings of the word \emph{stabilization} in our paper. Here, we mean \emph{dimensional} stabilization of a Weinstein domain, i.e., passage from $W$ to $W\times D^2$. Later, we invoke both \emph{Legendrian submanifold stabilization} and \emph{open book decomposition stabilization}.} and construct from this characterization a range of interesting convex decompositions of hypersurfaces in contact manifolds. We additionally recover an $h$-principle for embeddings of Weinstein hypersurfaces.
\end{enumerate}
With the necessary background in \cref{sec:preliminaries} and the $3$-dimensional context in \cref{sec:appearance_dim_3}, we hope that the reader familiar with $3$-dimensional contact topology may transition more easily to applying similar ideas in dimension $5$ and beyond.

\subsection{Main results}

In order to state our main results, we need some language from convex hypersurface theory. Here we are brief, and refer to \cref{sec:preliminaries} for a more precise review. We assume the reader is familiar with Weinstein topology, deferring to \cite{eliashberg2018weinsteinrevisited,cieliebak2012stein} for a survey and the standard reference, respectively.

A closed, embedded, oriented hypersurface $\Sigma^{2n} \subset (M^{2n+1}, \xi)$ in a contact manifold is \textit{convex} if it admits an everywhere-transverse contact vector field. For a suitably generic subclass of \textit{Weinstein convex hypersurfaces} there is a decomposition $\Sigma = R_+ \cup \Gamma \cup R_-$, where $\Gamma^{2n-1}$ is a codimension-$1$ separating hypersurface in $\Sigma$ and $\bar{R}_{\pm}$ is naturally a Weinstein domain. An example of a Weinstein convex hypersurface is the (rounded) boundary of a \textit{contact handlebody} $(H=W\times [-1,1],\, \xi = \ker(dt + \lambda))$, which is a compact contactization of a Weinstein domain $(W^{2n}, \lambda, \phi)$. We will often write $W\times [-1,1]$ or simply $H$ when referring to a contact handlebody when the contact structure is implicit. In this case, both $\bar{R}_+$ and $\bar{R}_-$ are deformation equivalent to $W$. Contact handlebodies are standard neighborhoods of Weinstein hypersurfaces,\footnote{A second note on potentially confusing language: \textit{convex hypersurfaces} are always closed, while \textit{Weinstein hypersurfaces} $W\hookrightarrow M$ have nonempty boundary and inherit from $\xi$ the structure of a Weinstein domain.} so they generalize ribbon neighborhoods of Legendrian graphs in dimension $3$. Importantly, a contact handlebody is built out of \textit{contact handles} of index $0, \dots, n$, which are contactizations of the underlying Weinstein handles in $(W^{2n}, \lambda, \phi)$.

By \cite{honda2019convex}, a generic $1$-parameter family of hypersurfaces is Weinstein convex except at a finite set of times. These exceptional times may alter the convex decomposition $R_+ \cup \Gamma \cup R_-$ of the hypersurface, and the change is described by a \textit{bypass attachment}. Abstractly, a bypass attachment is a certain smoothly canceling pair of contact handles of index $n$ and $n+1$. Though the handles cancel smoothly, they may not cancel at the level of contact structures. 

This paper is concerned with the effect of attaching an abstract bypass to a contact handlebody $H = W \times [-1,1]$. As a bypass involves a contact ($n+1$)-handle, the result of attaching a bypass to a contact handlebody is not necessarily contactomorphic to another contact handlebody; for instance, the contact structure on a contact handlebody is always tight, but there exist bypass attachments that render a contact structure overtwisted \cite[\S 10]{HH18}. 

\subsubsection{Bypass moves} 

Our main theorem describes bypass attachments to a contact handlebody $W\times [-1,1]$ that produce a different contact handlebody $W'\times [-1,1]$ via local changes to the handle structure of the underlying Weinstein domain. When $2n+1=5$, these moves are relative versions of the diagrammatic moves introduced by Ding-Geiges-van Koert \cite{Ding2012Diagrams}; see also \cref{prop:crossing_change}.

\begin{figure}[ht]
	\centering
	\begin{overpic}[scale=.55]{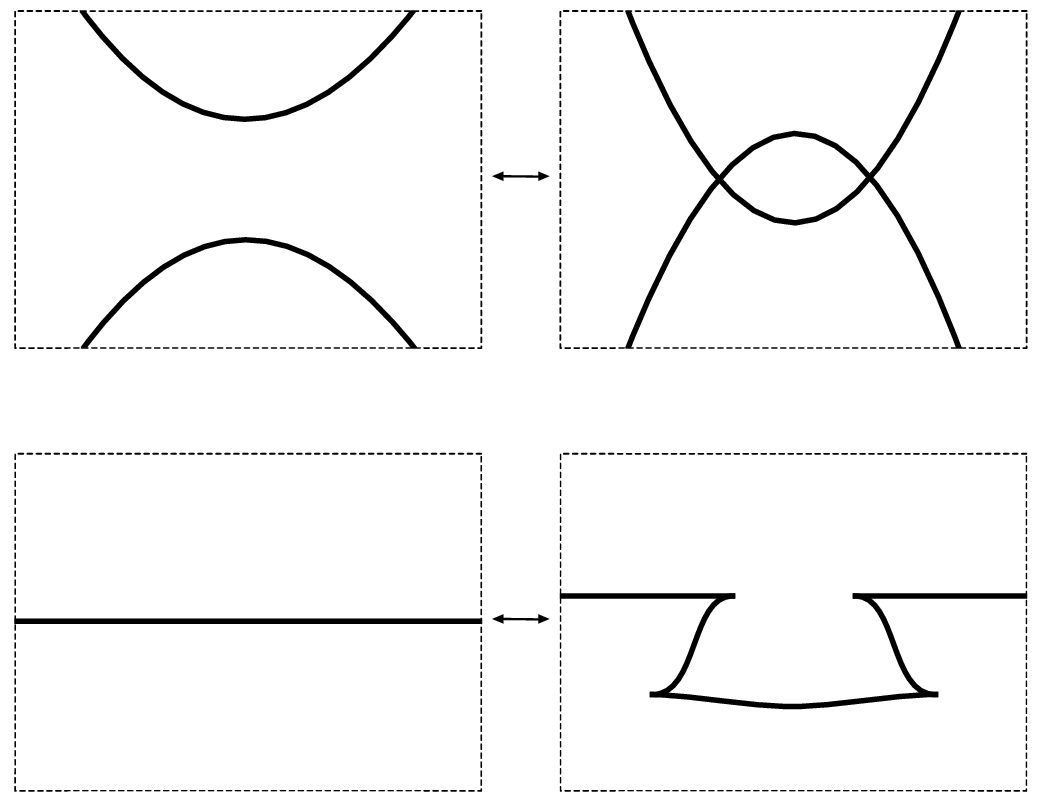}
    \put(37,21){\scriptsize $(-1)$}
    \put(89,23){\scriptsize $(-1)$}
    \put(37,70){\scriptsize $(-1)$}
    \put(37,50){\scriptsize $(-1)$}
    \put(90,68){\scriptsize $(-1)$}
    \put(90,52){\scriptsize $(-1)$}
	\end{overpic}
	\caption{The (un)clasping move in the first row and (de)stabilizing bypass move in the second row of \cref{thm:main_moves}. All diagrams are $S^{n-2}$-spun around the central vertical axis and represent front projections of local surgery diagrams for $2n$-dimensional Weinstein domains.}
	\label{fig:thm_moves}
\end{figure}
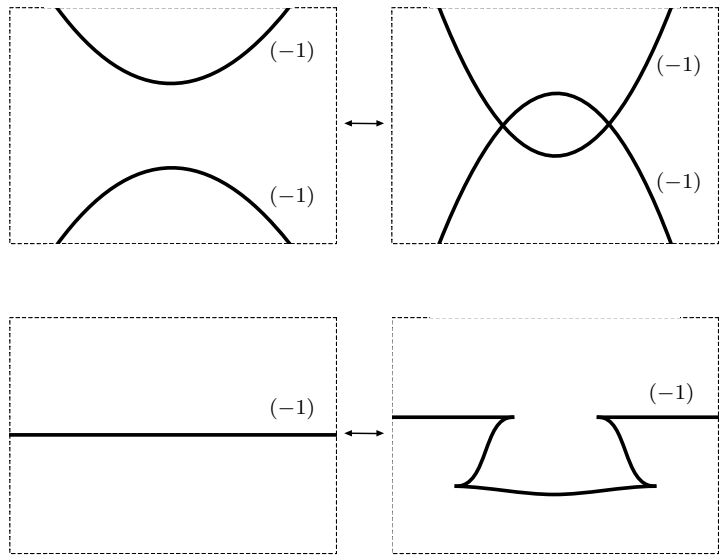

\begin{theorem}\label{thm:main_moves}
Let $(W,\lambda, \phi)$ be a Weinstein domain of dimension $2n\geq 4$ and let $H = W \times [-1,1]$ be the abstract contact handlebody based over $W$. Suppose that in a Legendrian Kirby diagram for $W$ there is a chart with front projection given by any of the diagrams in \cref{fig:thm_moves}. There is an abstract bypass attachment $B$ such that $H\cup B$ is contactomorphic relative to the boundary to a contact handlebody $H'=W' \times [-1,1]$, where the diagram for $W'$ differs from that of $W$ according to the adjacent diagram in \cref{fig:thm_moves}. Moreover, when $2n\geq 6$, the (de)stabilizing move preserves the smooth isotopy type of of the underlying hypersurface.
\end{theorem}

\begin{remark}
When $2n\geq 6$, the stabilizing move in \cref{thm:main_moves} "loosifies" the corresponding attaching sphere, and may consequently be used to "flexibilize" a Weinstein hypersurface. This has the additional effect of "loosifying" the embedding of the hypersurface itself. See \cref{subsec:flex} for additional discussion. 
 
\end{remark}

\cref{thm:main_moves} is a statement about bypasses attached to the exterior of an abstract contact handlebody; given a contact handlebody embedded inside a contact manifold, the theorem does not guarantee the existence of such exterior bypass attachments. However, by taking $H' = H \cup B$ to be the embedded handlebody, it does give the existence of \textit{interior} bypass attachments. For example, this produces interesting convex spheres in a Darboux ball. 

\begin{corollary}\label{cor:darboux_spheres}
Let $(D^{2n+1}, \xi_{\mathrm{std}})$ be a Darboux ball and $(W^{2n}, \lambda, \phi)$ a contractible Weinstein domain. There is a convex sphere $S^{2n}\subset D^{2n+1}$ such that $R_{\pm}(S^{2n})$ is deformation equivalent to $W$. 
\end{corollary}

\cref{cor:darboux_spheres} applies, for instance, when $W$ is a $4$-dimensional Weinstein Mazur manifold \cite{mazur1961note,akbulut2019knot,hayden2021exotic} or an exotic Weinstein structure on $D^{2n}$ for $2n\geq 6$ \cite{seidel2004ramanujam,mclean2008symplectic}. See \cref{cor:subcritically-fillable} below, which implies \cref{cor:darboux_spheres}.

\subsubsection{Appearance of bypass moves in dimension $3$} Suspecting that many readers are more familiar with contact topology in low dimensions than high, we establish evidence for the utility of the bypass moves in \cref{thm:main_moves} by identifying the analogous bypass attachments in dimension $3$. Interpreted correctly, the (un)clasping bypass move is ubiquitous.

\begin{theorem}\label{thm:main_dim3}
Let $W$ be an abstract oriented surface with nonempty boundary, viewed in disk-band form, and let $H= W\times [-1,1]$ be the contact handlebody based over $W$. Let $W'$ be the surface obtained by swapping the attaching regions of two adjacent feet as in \cref{fig:3dclasp}.  There is a bypass attachment $B$ such that $H \cup B$ is contactomorphic relative to the boundary to the contact handlebody $H'=W'\times [-1,1]$. Moreover, $B$ witnesses: 
\begin{enumerate}
    \item Modification of the slope of dividing curves on a standard convex torus via a basic slice. \label{part:slope}
    \item Destabilization of a Legendrian arc.\label{part:destab_arc}
    \item Permutation of the cyclic order of edges near a vertex in a Legendrian graph.\label{part:graph}
\end{enumerate}
\end{theorem}

\begin{figure}[ht]
	\centering
    \vspace{-1cm}
	\begin{overpic}[scale=.32]{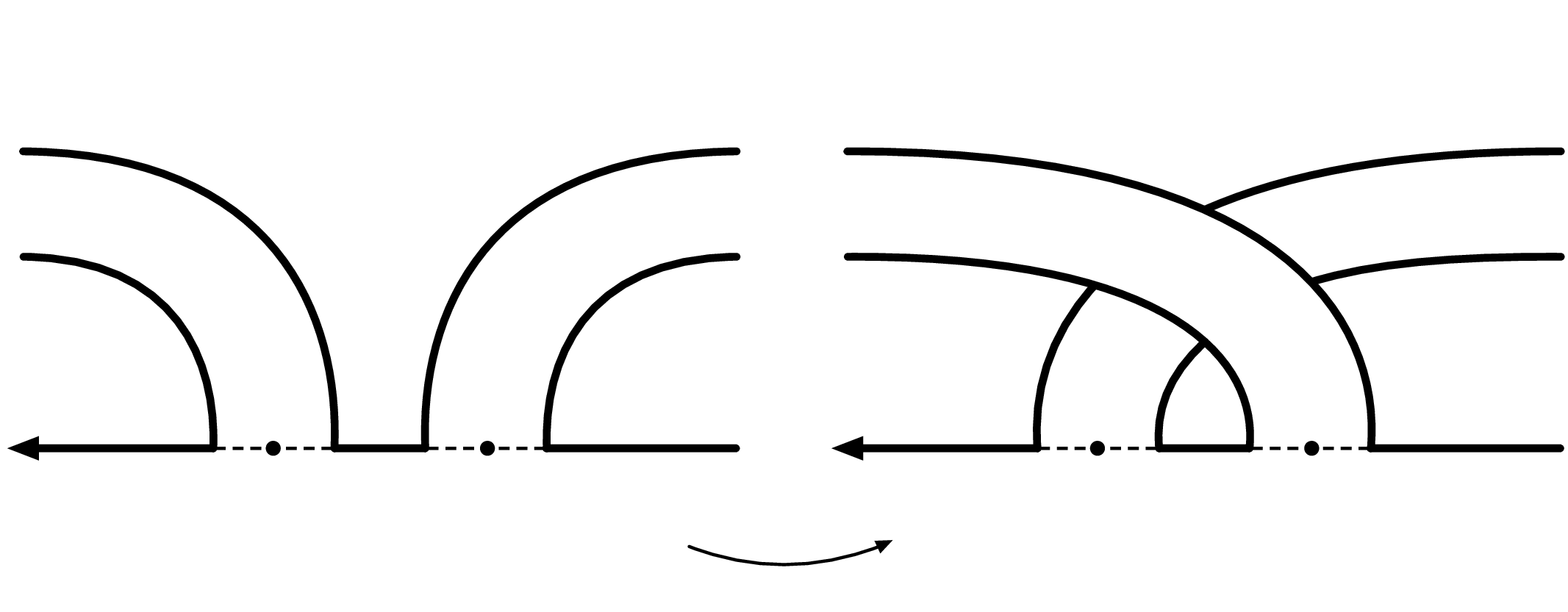}
    \put(22,5){\small $W$}
    \put(75,5){\small $W'$}
	\end{overpic}
    \vspace{-0.3cm}
	\caption{The $3$-dimensional clasp move.}
	\label{fig:3dclasp}
\end{figure}

The clasp move gives new perspectives on various contact geometric phenomena. For example, we can use it to control the convex decomposition of a Heegaard surface in a contact $3$-manifold. The following corollary describes two extreme cases, where we minimize and maximize the genus of positive and negative regions. 

\begin{restatable}{corollary}{heegaard}\label{cor:heegaard}
Let $(Y, \xi)$ be a closed contact $3$-manifold. Let $\Sigma\subset Y$ be a smooth Heegaard surface. 
\begin{enumerate}
    \item The surface $\Sigma$ is smoothly isotopic to a convex Heegaard surface $\Sigma'$ such that $R_+(\Sigma')$ and $R_-(\Sigma')$ are planar.\label{part:heegaard_planar}
    \item If $g(\Sigma)\neq 1$, then $\Sigma$ is also smoothly isotopic to a convex Heegaard surface $\Sigma''$ with connected dividing set $\Gamma_{\Sigma''}$.\label{part:heegaard_connected}
\end{enumerate}
\end{restatable}

\subsubsection{Appearance of bypass moves in higher dimensions}
While \cref{thm:main_moves} is a result about contact handlebodies, by considering neighborhoods of Weinstein pages of open book decompositions, we are naturally led to statements about closed contact manifolds.  We consider the simplest case, where the monodromy of the open book is the identity map.

\begin{restatable}{theorem}{movesForObds}\label{thm:moves_for_obds}
Let $(W,\lambda, \phi)$ be a Weinstein domain of dimension $2n\geq 4$ and let $(W',\lambda', \phi')$ be obtained from $W$ by any of the local moves described in \cref{thm:main_moves}.  The contact manifolds $\mathcal{OB}(W,\mathrm{id})$ and $\mathcal{OB}(W',\mathrm{id})$ are contactomorphic.
\end{restatable}

\begin{example}
In the same spirit as \cref{cor:darboux_spheres}, if $W^{2n}$ is a Weinstein Mazur manifold or an exotic Weinstein ball, \cref{thm:moves_for_obds} implies that $\mathcal{OB}(W,\mathrm{id}) = (S^{2n+1}, \xi_{\mathrm{std}})$.  See \cref{subsec:mazur}.
\end{example}

In \cref{sec:appearance_high_dim} we present a proof of \cref{thm:moves_for_obds} using our bypass moves, without making explicit reference to fillings. The hope is to use bypass moves to study non-identity monodromies, in particular those that may not obviously factor as Dehn twists; toward this goal our proof may be instructive.

However, \cref{thm:moves_for_obds} can alternatively be deduced by applying \cite[Theorem 14.3]{cieliebak2012stein}, Cieliebak-Eliashberg's $h$-principle for flexible Weinstein domains, to subcritical fillings of the contact manifolds in question. Indeed, \cref{thm:main_moves} describes ($2n+1$)-dimensional contact geometric bypass moves, but underlying them are local modifications of $2n$-dimensional Weinstein domains. In the case of identity open book monodromies, these page modifications induce modifications of splittings of subcritical ($2n+2$)-dimensional fillings. We take this perspective in the following theorem, which subsumes \cref{thm:moves_for_obds}.

\begin{restatable}{theorem}{movesSufficient}\label{thm:moves_sufficient_for_stable_equivalence}
Let $(W,\lambda,\phi)$ and $(W',\lambda',\phi')$ be Weinstein domains of dimension $2n \geq 4$.  The following are equivalent:
\begin{enumerate}
    \item $(W\times D^2, d\lambda + \omega_{\mathrm{std}})$ is almost symplectomorphic to $(W'\times D^2, d\lambda' + \omega_{\mathrm{std}})$;\label{thm-part:stably-almost-symplectomorphic}
    \item $(W,\lambda,\phi)$ and $(W',\lambda',\phi')$ admit Weinstein handle decompositions related to one another by the moves described in \cref{thm:main_moves}.\label{thm-part:related-by-moves}
\end{enumerate}
\end{restatable}

\noindent It is not difficult to see that \eqref{thm-part:related-by-moves} implies \eqref{thm-part:stably-almost-symplectomorphic}. Indeed, if $f\colon S^{k-1}\hookrightarrow \partial W$ is the attaching sphere of a Weinstein $k$-handle, then both moves described by \cref{thm:main_moves} preserve the isotopy class of $df\colon TS^{k-1}\to f^*\xi$ in the space of isotropic monomorphisms, where $\xi$ is the contact structure induced on $\partial W$ by the Weinstein structure. (In the case $k=n$, this is simply the rotation class of $f$.)  Whatever the value of $k$, $f$ is subcritical as an isotropic sphere in $\partial(W\times D^2)$, and the $h$-principle \cite[Theorem 12.4.1]{eliashberg2002h} tells us that both of our moves preserve the isotropic isotopy class of $f$ in $\partial(W\times D^2)$. Alternatively, our proof of \cref{thm:moves_for_obds} also gives the implication \eqref{thm-part:related-by-moves} $\implies$ \eqref{thm-part:stably-almost-symplectomorphic} by upgrading the argument to track its effect on the fillings. We prove \eqref{thm-part:stably-almost-symplectomorphic} $\implies$ \eqref{thm-part:related-by-moves} in \cref{sec:appearance_high_dim}.

In \cref{sec:subcritically-examples} we use the following corollary of \cref{thm:moves_sufficient_for_stable_equivalence} to exhibit a menagerie of interesting convex decompositions of hypersurfaces in closed contact manifolds, in particular the spheres described in \cref{cor:darboux_spheres}.

\begin{restatable}{corollary}{highDimCor}\label{cor:subcritically-fillable}
Let $W\subset (M^{2n+1}, \xi=\ker\alpha)$ be an embedded Weinstein hypersurface and let $\Sigma = \partial N(W)$ be the convex boundary of its standard neighborhood. Let $(W', \lambda',\phi')$ be a Weinstein domain such that $(W'\times D^2, d\lambda' + \omega_{\mathrm{st}})$ is almost symplectomorphic to $(W \times D^2,d\alpha\vert_W + \omega_{\mathrm{st}})$. Then there is a convex hypersurface $\Sigma'\subset M$ smoothly isotopic to $\Sigma$ such that $R_{\pm}(\Sigma')$ is deformation equivalent to $(W', \lambda',\phi')$.
\end{restatable}

As a final application of our bypass moves, we recover the following $h$-principle for Weinstein hypersurface embeddings.  This result seems to be known to experts (see, for example, the proof of \cite[Theorem 4.12]{lazarev2020maximal}), but the authors are unaware of an explicit statement in the literature.  We thank Oleg Lazarev for correspondence which led to its inclusion here.

\begin{restatable}{corollary}{weinsteinHPrinciple}\label{cor:weinstein-h-principle}
Let $W\subset (M^{2n+1}, \xi=\ker\alpha)$ be an embedded Weinstein hypersurface of dimension $2n\geq 6$ and let $(W',\lambda',\phi')$ be a Weinstein domain with $(W',d\lambda')$ almost symplectomorphic to $(W,d\alpha\vert_W)$.  Then there is a smooth isotopy of $W\subset M$ to the image of a Weinstein embedding $(W',\lambda',\phi') \hookrightarrow (M, \xi)$.
\end{restatable}

The bypass moves in \cref{thm:main_moves} are bound to show up in many other contexts in higher dimensions. For example, \eqref{part:graph} of \cref{thm:main_dim3} suggests to look for them near neighborhoods of \textit{arboreal Legendrian singularities}. First introduced by Nadler \cite{nadler2017arboreal}, arborealization has since become a significant program in higher dimensional symplectic topology \cite{starkston2018arboreal,alvarez2023stability,alvarezgavela2022positivearborealizationpolarizedweinstein}. Arboreal Legendrian singularities admit canonical Weinstein thickenings and thus canonical contact handlebody neighborhoods. We close with the following problem: 

\begin{problem}
Show that the canonical contact handlebody neighborhoods of arboreal singularities of Legendrian submanifolds are related to each other by the bypass moves of \cref{thm:main_moves}, generalizing \eqref{part:graph} of \cref{thm:main_dim3}.    
\end{problem}

\subsection{Organization}

In \cref{sec:preliminaries} we provide background information on convex hypersurface theory, contact handles, bypass attachments, and open book decompositions. In \cref{sec:bypass_moves} we establish the bypass moves of \cref{thm:main_moves}. In \cref{sec:appearance_dim_3} we identify $3$-dimensional analogues of the bypass moves and prove \cref{thm:main_dim3} and \cref{cor:heegaard}. Finally, in \cref{sec:appearance_high_dim} we prove \cref{thm:moves_for_obds} and \cref{thm:moves_sufficient_for_stable_equivalence} and apply \cref{cor:subcritically-fillable} to some noteworthy examples from the literature.

\subsection{Acknowledgments}

The authors would like to thank John Etnyre and Ko Honda for helpful discussions in the preparation of this article, and Oleg Lazarev for a fruitful correspondence following an early draft. We additionally thank an anonymous referee for several helpful suggestions and corrections, including a significant simplification of the proof of \cref{thm:moves_sufficient_for_stable_equivalence}.
\section{Preliminaries}\label{sec:preliminaries}

In this section we collect the background material needed to use bypass attachments in higher dimensions. \cref{subsec:convex} covers convexity, \cref{subsec:kirby} and \cref{subsec:dec_diags} discuss Legendrian Kirby calculus and contact handles, \cref{subsec:Bypass} defines bypass attachments, \cref{subsec:OBD} and \cref{subsec:OBDandBypass} describe open book decompositions and their relation to contact handles and bypasses, and \cref{subsec:flex} discusses the flexibility and looseness of Weinstein hypersurfaces.

\subsection{Convex hypersurface theory}\label{subsec:convex}

First, we recall the basics of convex hypersurface theory \cite{giroux1991convexite, HH18} and introduce notation and conventions that will be used throughout. 

\begin{definition}
Let $\Sigma \subset (M, \xi)$ be an embedded hypersurface in a contact manifold. Its \textbf{characteristic foliation} is the singular $1$-dimensional foliation of $\Sigma$ given by $\Sigma_{\xi} := (T\Sigma \cap \xi)^{\bot}$, where $\bot$ is the symplectic orthogonal complement taken with respect to the conformal symplectic structure on $\xi$. 
\end{definition}

If $\xi$ and $\Sigma$ are cooriented, then singular points of the foliation, i.e., points $p\in \Sigma$ where $T_p\Sigma = \xi_p$, are by definition classified as \textit{positive} or \textit{negative} according to whether the coorientations of $T_p\Sigma$ and $\xi_p$ agree or disagree, respectively. Moreover, at nonsingular points, the foliation is oriented according to the coorientations of $\Sigma$ and $\xi$. We may therefore choose a vector field $v\in \mathfrak{X}(\Sigma)$ which directs the oriented foliation $\Sigma_{\xi}$. Zeroes of $v$ correspond to singular points of the foliation, and are likewise classified as positive or negative. 

\begin{definition}
Let $\Sigma$ be a closed and oriented manifold of dimension $2n\geq 2$. We say that $\phi\in\mathfrak{Morse}_n^n(\Sigma)$ if the following properties hold. 
\begin{enumerate}
    \item The function $\phi:\Sigma \to \R$ is generalized Morse, allowing birth-death  critical points.
    \item There is a regular value $c\in \R$, called a \textbf{splitting value for $\phi$}, such that
    \begin{enumerate}
        \item $\phi\vert_{\phi^{-1}(-\infty,c)}$ has critical points of index at most $n$, and likewise 
        \item $\phi\vert_{\phi^{-1}(c,\infty)}$ has critical points of index at least $n$.
    \end{enumerate}
\end{enumerate}
\end{definition}
\noindent Our notation follows that of $\mathfrak{Morse}_n(W^{2n})$ used by \cite{cieliebak2012stein} to identify generalized Morse functions with critical points of index at most $n$. 

\begin{definition}[\cite{giroux1991convexite,honda2019convex,eliashberg2022hondahuangs}]\label{def:convex2.0}
Let $(M, \xi)$ be a cooriented contact manifold of dimension $2n+1\geq 3$, and let $\Sigma\subset M$ be a cooriented embedded closed hypersurface. We say that $\Sigma$ is \textbf{convex} if there is a contact vector field $X$ everywhere transverse to $\Sigma$. We say that $\Sigma$ is \textbf{Weinstein convex} if there is a smooth function $\phi\in\mathfrak{Morse}_n^n(\Sigma)$, called a \textbf{good Lyapunov function}, satisfying the following properties: 
\begin{enumerate}
    \item There is a vector field $v$ directing the characteristic foliation $\Sigma_{\xi}$ that is gradient-like for $\phi$; i.e., $\phi$ is a Lyapunov function for $v$. 
    \item There is a \textbf{good splitting value} $c\in \R$ for $\phi$ such that all positive singular points of $\Sigma_{\xi}$ are in $\phi^{-1}(-\infty,c)$ and all negative singular points of $\Sigma_{\xi}$ are in $\phi^{-1}(c,\infty)$. \label{part:convex2.0_def2}
\end{enumerate}
\end{definition}

\begin{remark}
We adopt the language \textit{good Lyapunov function} from \cite{eliashberg2022hondahuangs}. Moreover, we note if a characteristic foliation admits a good Lyapunov function, it admits many different good Lyapunov functions, as well as Lyapunov functions which are not good, i.e., that do not satisfy \eqref{part:convex2.0_def2}. In fact, it is equivalent to require the existence of any Lyapunov function, together with the assumption that there are no trajectories of the characteristic foliation from negative singular points to positive singular points; see the discussion following \cite[Definition 3.6]{eliashberg2022hondahuangs}. For this reason, we do not explicitly include $\phi$ or any other data when referring to the Weinstein convex hypersurface $\Sigma$.  
\end{remark}

The following lemma justifies the language \emph{Weinstein convex}.

\begin{lemma}\cite[Lemma 3.10]{eliashberg2022hondahuangs}
If $\Sigma$ is Weinstein convex, it is convex.  
\end{lemma}

By integrating a complete\footnote{Completeness may always be achieved (in an arbitrarily small neighborhood of $\Sigma$) by decaying the vector field via a bump contact Hamiltonian function.} transverse contact vector field $X =: \partial_z$ giving the coorientation of $\Sigma$, we identify a vertically invariant neighborhood
\[
N(\Sigma) \cong \left(\Sigma \times \R_z, \, \alpha = f\, dz + \beta\right)
\]
where $\Sigma = \Sigma \times \{0\}$, and $f: \Sigma \to \R$ and $\beta\in \Omega^1(\Sigma)$ are $z$-invariant. The contact condition forces $f\pitchfork 0$, hence
\[
\Gamma:=\{f=0\}=\{\alpha(X)=0\} = \{X(p) \in \xi_p\}
\]
is a smoothly embedded codimension-$1$ contact submanifold of $\Sigma$ called the \textit{dividing set}. While $\Gamma$ depends on the choice of contact vector field $X$, its (contact) isotopy class in $\Sigma$ does not, as the space of contact vector fields transverse to $\Sigma$ is contractible. The \textit{positive} (resp. \textit{negative}) \textit{region of $\Sigma$} is given by $R_{\pm}:= \{\pm f > 0\} = \{\pm\alpha(X) >0\}$.

The language \textit{Weinstein convex} used in \cref{def:convex2.0} is appropriate for the following reason. Let $\Sigma$ be a Weinstein convex hypersurface and $N(\Sigma) \cong (\Sigma \times \R, \alpha = f\, dz + \beta)$ an invariant neighborhood. Let $\phi\in\mathfrak{Morse}_n^n(\Sigma)$ be a good Lyapunov function with good splitting value $c$. We may assume that $\phi^{-1}(c) = \Gamma = \{f=0\}$. Let $\phi_{\pm}:= \pm \phi\vert_{R_{\pm}}$. By \eqref{part:convex2.0_def2} of \cref{def:convex2.0} it follows that $\phi_{\pm}\in \mathfrak{Morse}_n(R_{\pm})$. For $\ve > 0$ sufficiently small, set $\lambda_{\pm}:=\frac{\ve\, \beta}{\ve \pm f}\vert_{\bar{R}_{\pm}}$. Then $\lambda_{\pm}$ is a Liouville form on $\bar{R}_{\pm}$ and the Liouville vector field $X_{\lambda_{\pm}}$ is gradient-like for $\phi_{\pm}$. Thus, 
$(\bar{R}_{\pm}, \lambda_{\pm}, \phi_{\pm})$ is a Weinstein domain. \footnote{Alternatively, we may consider $\bar{R}_+$ and $\bar{R}_-$ as ideal Liouville domains in the sense of \cite{giroux2020ideal}.}

The most important example of a convex hypersurface for our purposes is the smoothed boundary of a contact handlebody $H = (W \times [-1,1],\, \xi = \ker(dt +\lambda))$. After rounding along the edges $\{\pm 1\} \times \partial W$ (see \cite[\S 3.2]{avdek2012liouville} for details), the boundary $\partial H$ is Weinstein convex with dividing set $\{0\}\times \partial W$, positive region the rounding of $((0,1] \times \partial W) \cup (\{1\} \times W)$ , and negative region the rounding of $([-1,0) \times \partial W) \cup (\{-1\} \times W)$.

\subsection{Legendrian Kirby calculus}\label{subsec:kirby}

Bypass calculations beyond dimension $5$ involve manipulations of high dimensional Legendrian Kirby diagrams. One way to present high dimensional front projections is to spin a planar front around a central vertical axis; see \cref{fig:spin1}. Casals-Murphy \cite{casals2019fronts} established conventions for a wider class of diagrams that do not necessarily exhibit \textit{global} spin symmetry, but we do not need this level of generality; all of our high dimensional front projections are globally centrally spun. Note that care must be taken when manipulating spun diagrams --- for instance, Reidemeister moves must respect any spin symmetry. 

\begin{figure}[ht]
	\begin{overpic}[scale=0.5]{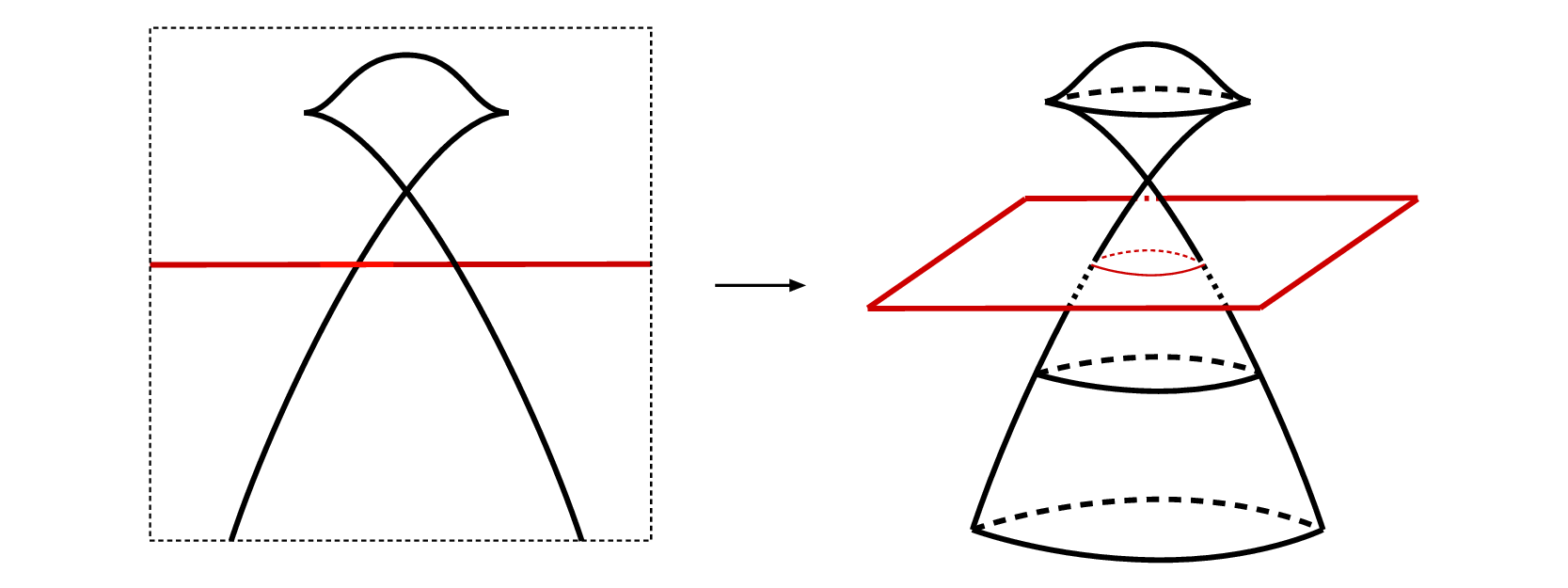}
	 \put(25.25,0.5){\small $\Gamma'$}
     \put(39,3.75){\small $\Gamma$}
     \put(35,27){\small $(-1)$}
	\end{overpic}
	\caption{A planar front projection implicitly $S^1$-spun around the central vertical axis to produce a Legendrian surface front projection. Such a diagram depicts, for example, Legendrian surfaces in the $5$-dimensional dividing set of a $6$-dimensional convex hypersurface in a $7$-dimensional contact manifold.}
	\label{fig:spin1}
\end{figure}

We typically enclose surgery diagrams in a box to emphasize their locality; contrast \cref{fig:mazur}, which is a global surgery diagram. In the presence of surgeries, we often include two additional labels beyond surgery coefficients: one in the lower right corner to indicate the contact manifold represented by the empty diagram, and one below to indicate the result of the surgeries. Our surgery coefficients are always explicit: if a Legendrian is drawn without a $(\pm 1)$-surgery coefficient, no surgery is being performed. For instance, in \cref{fig:spin1}, the contact manifold $\Gamma'$ is produced by performing $(-1)$ surgery on $\Gamma$ along only the black Legendrian; the red Legendrian is just another Legendrian (viewed either in $\Gamma$ or $\Gamma'$) of which we may wish to keep track. 

Thanks to the following proposition from \cite{weinstein1991surgery}, a Legendrian surgery diagram can present not only a contact manifold, but also a Weinstein domain filling it. 

\begin{proposition}[Weinstein handle attachment and removal]
Let $(W, \lambda, \phi)$ be a Weinstein domain with contact boundary $M := \partial W$. Let $\Lambda\subset M$ be a Legendrian sphere. 
\begin{enumerate}
    \item If $W'$ is obtained from $W$ by attaching a critical handle along $\Lambda$, then $M':= \partial W'$ is obtained from $M$ by contact-$(-1)$ surgery along $\Lambda$. 

    \item If $D\subset W$ is an embedded regular Lagrangian disk with $\partial D = \Lambda$ and $W'$ is obtained from $W$ by removing a critical handle with cocore $D$, then $M':= \partial W'$ is obtained from $M$ by contact-$(+1)$ surgery along $\Lambda$.  
\end{enumerate}
\end{proposition}

\begin{remark}
One cannot perform a $(+1)$ surgery on an arbitrary Legendrian and expect the result to be Weinstein; the $(+1)$-surgered Legendrian must be filled by a regular Lagrangian disk (c.f. \cite{eliashberg2018flexiblelagrangians,conway2021symplectic}). See \cref{subsec:dec_diags} below.   
\end{remark}

Given a Weinstein handlebody diagram for a Weinstein domain, we can perform Weinstein homotopies via Legendrian isotopies of the attaching spheres, as well as Legendrian handleslides across surgeries. Local front models for handleslides across both $(-1)$ and $(+1)$ surgeries in all dimensions were established in \cite{casals2019fronts}, generalizing work of Ding-Geiges in dimension $3$ \cite{Ding2009HandleMI}, and are depicted in the bottom two rows of \cref{fig:slides}. We use the $\uplus$ operation of \cite{HH18} to describe handleslides notationally. This operation takes as input two Legendrians $\Lambda_1, \Lambda_2$ intersecting $\xi$-transversally at a single point, and produces a new Legendrian $\Lambda_1 \uplus \Lambda_2$, smoothly isotopic to their connected sum, whose front projection is given by the modified front projection in the top row of \cref{fig:slides}. Note that $\Lambda_1 \uplus \Lambda_2$ and $\Lambda_2 \uplus \Lambda_1$ are generally not Legendrian isotopic. Finally, we let $\Lambda^c$ denote a time-$c$ Reeb shift of $\Lambda$.  

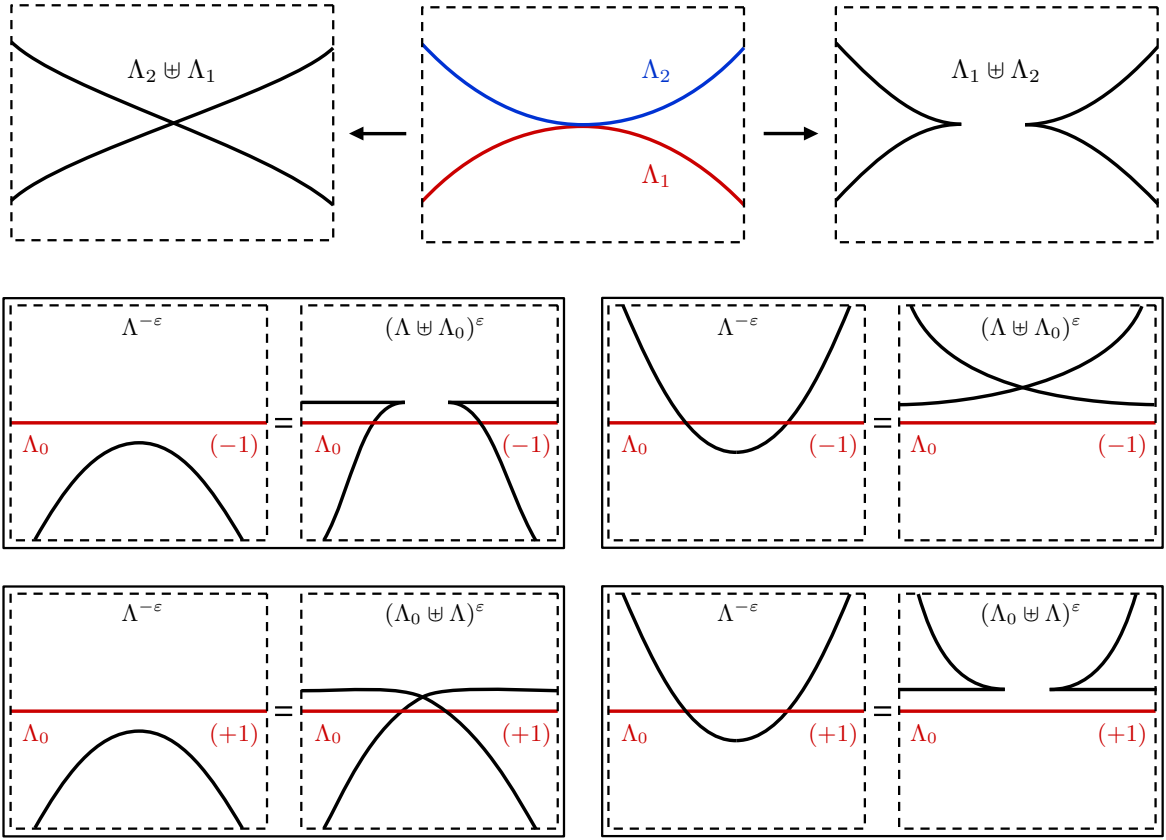
\begin{figure}[ht]
	\begin{overpic}[scale=0.43]{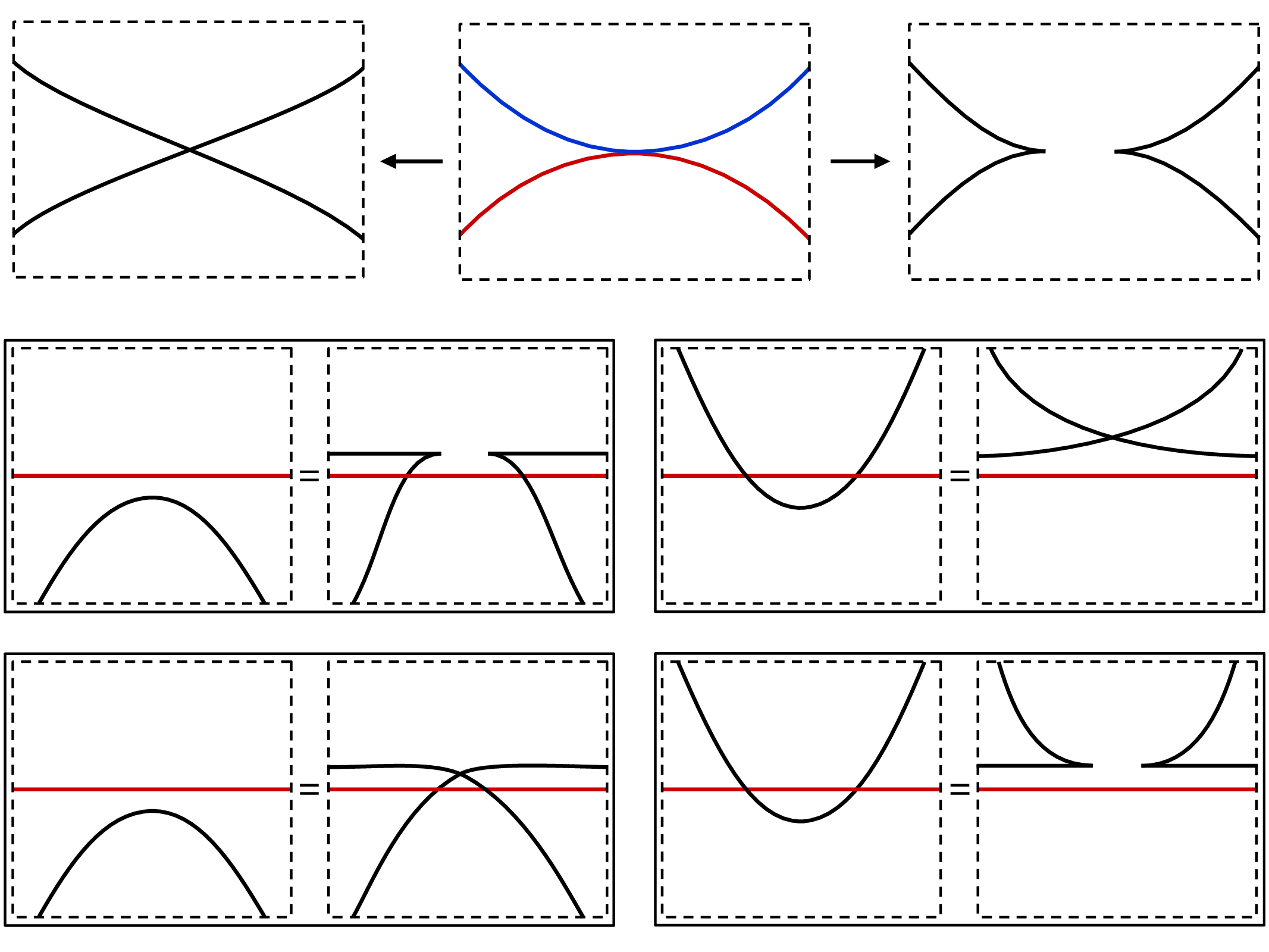}
	 \put(55,58){\small \textcolor{darkred}{$\Lambda_1$}}
  \put(55,67){\small \textcolor{darkblue}{$\Lambda_2$}}
  \put(81.5,67){\small $\Lambda_1 \uplus \Lambda_2$}
  \put(11,67){\small $\Lambda_2 \uplus \Lambda_1$}

  \put(18,35){\footnotesize \textcolor{darkred}{$(-1)$}}
  \put(43,35){\footnotesize \textcolor{darkred}{$(-1)$}}
  \put(69.25,35){\footnotesize \textcolor{darkred}{$(-1)$}}
  \put(94,35){\footnotesize \textcolor{darkred}{$(-1)$}}

  \put(18,10.25){\footnotesize \textcolor{darkred}{$(+1)$}}
  \put(43,10.25){\footnotesize \textcolor{darkred}{$(+1)$}}
  \put(69.25,10.25){\footnotesize \textcolor{darkred}{$(+1)$}}
  \put(94,10.25){\footnotesize \textcolor{darkred}{$(+1)$}}

  \put(2,10.25){\footnotesize \textcolor{darkred}{$\Lambda_0$}}
  \put(27,10.25){\footnotesize \textcolor{darkred}{$\Lambda_0$}}
  \put(53.25,10.25){\footnotesize \textcolor{darkred}{$\Lambda_0$}}
  \put(78,10.25){\footnotesize \textcolor{darkred}{$\Lambda_0$}}

  \put(2,35){\footnotesize \textcolor{darkred}{$\Lambda_0$}}
  \put(27,35){\footnotesize \textcolor{darkred}{$\Lambda_0$}}
  \put(53.25,35){\footnotesize \textcolor{darkred}{$\Lambda_0$}}
  \put(78,35){\footnotesize \textcolor{darkred}{$\Lambda_0$}}

   \put(10.5,45){\footnotesize $\Lambda^{-\ve}$} 
   \put(33,45){\footnotesize $(\Lambda \uplus \Lambda_0)^{\ve}$} 
   \put(61.5,45){\footnotesize $\Lambda^{-\ve}$} 
   \put(84,45){\footnotesize $(\Lambda \uplus \Lambda_0)^{\ve}$} 

   \put(10.5,20.5){\footnotesize $\Lambda^{-\ve}$} 
   \put(33,20.5){\footnotesize $(\Lambda_0 \uplus \Lambda)^{\ve}$} 
   \put(61.5,20.5){\footnotesize $\Lambda^{-\ve}$} 
   \put(84,20.5){\footnotesize $(\Lambda_0 \uplus \Lambda)^{\ve}$} 
  
	\end{overpic}
	\caption{Legendrian handleslides. The top row describes the $\uplus$ operation. The second row describes a handleslide of $\Lambda^{-\ve}$ up over a $(-1)$-surgery along $\Lambda_0$, and the third row is the same but with a $(+1)$-surgery along $\Lambda_0$. All diagrams are spun around their central vertical axis. There are additional models obtained by reflecting the bottom two rows across their central horizontal axes.}
	\label{fig:slides}
\end{figure}

\subsection{Contact handles and decorated diagrams}\label{subsec:dec_diags}

Let $(M^{2n+1}, \xi)$ be a contact manifold with Weinstein convex boundary $\Sigma = R_+ \cup \Gamma \cup R_-$. 

The first principle is that the calculus of contact handles of index at most $n$ is equivalent to the underlying Weinstein handle calculus. That is, a contact $n$-handle is modeled on the contactization of of a Weinstein $n$-handle, and its attaching data on a convex boundary hypersurface is a Legendrian ($n-1$)-sphere $\Lambda\subset \Gamma$ in the dividing set. The first item of the proposition below summarizes the effect of attachment. Contact ($n+1$)-handles are perhaps less familiar and are described by the second item of the same proposition. Notably, the attaching $S^n$-sphere of a contact ($n+1$)-handle is the union of two Lagrangian disks, one in $R_+$ and one in $R_-$, that meet along a common equatorial Legendrian sphere in the dividing set.

\begin{proposition}
Let $(M^{2n+1}, \xi)$ be a contact manifold with Weinstein convex boundary $\Sigma = R_+ \cup \Gamma \cup R_-$. 
\begin{enumerate}
    \item \cite[Proposition 3.1.4]{HH18} If $(M', \xi')$ is obtained by attaching a contact $n$-handle along $\Lambda\subset \Gamma$, the new convex boundary $\Sigma' = R_+' \cup \Gamma' \cup R_-'$ is obtained by simultaneous critical Weinstein handle attachment to $R_{\pm}$ along $\Lambda$. Consequently, the new dividing set $\Gamma'$ is obtained from $\Gamma$ by performing contact-$(-1)$ surgery along $\Lambda$. In particular, attaching a contact $n$-handle to a contact handlebody amounts to attaching a Weinstein $n$-handle to the underlying Weinstein domain.
    \item \cite[Proposition 3.2.6]{HH18} Let $D_{\pm}\subset R_{\pm}$ be regular Lagrangian disks such that $\partial D_+ = \Lambda = \partial D_-$. Let $(M', \xi')$ be obtained by attaching a contact ($n+1$)-handle along the $n$-sphere $D_- \cup D_+$ and let $\Sigma' = R_+' \cup \Gamma' \cup R_-$ be the new convex boundary. Then $R_{\pm}'$ is obtained by removing a neighborhood of $D_{\pm}$ from $R_{\pm}$, and consequently $\Gamma'$ is obtained by performing contact $(+1)$-surgery along $\Lambda$. 
\end{enumerate}
\end{proposition}

Although the data of a contact $n$-handle attachment is completely determined in a surgery diagram for the dividing set by its Legendrian attaching sphere, the equatorial Legendrian sphere in the dividing set is not a sufficient amount of data to specify a contact ($n+1$)-handle attachment; we need to keep track of the disk fillings in both $R_+$ and $R_-$. This is a nontrivial task when carrying out Kirby calculus in the dividing set. 

\begin{figure}[ht]
	\begin{overpic}[scale=0.2]{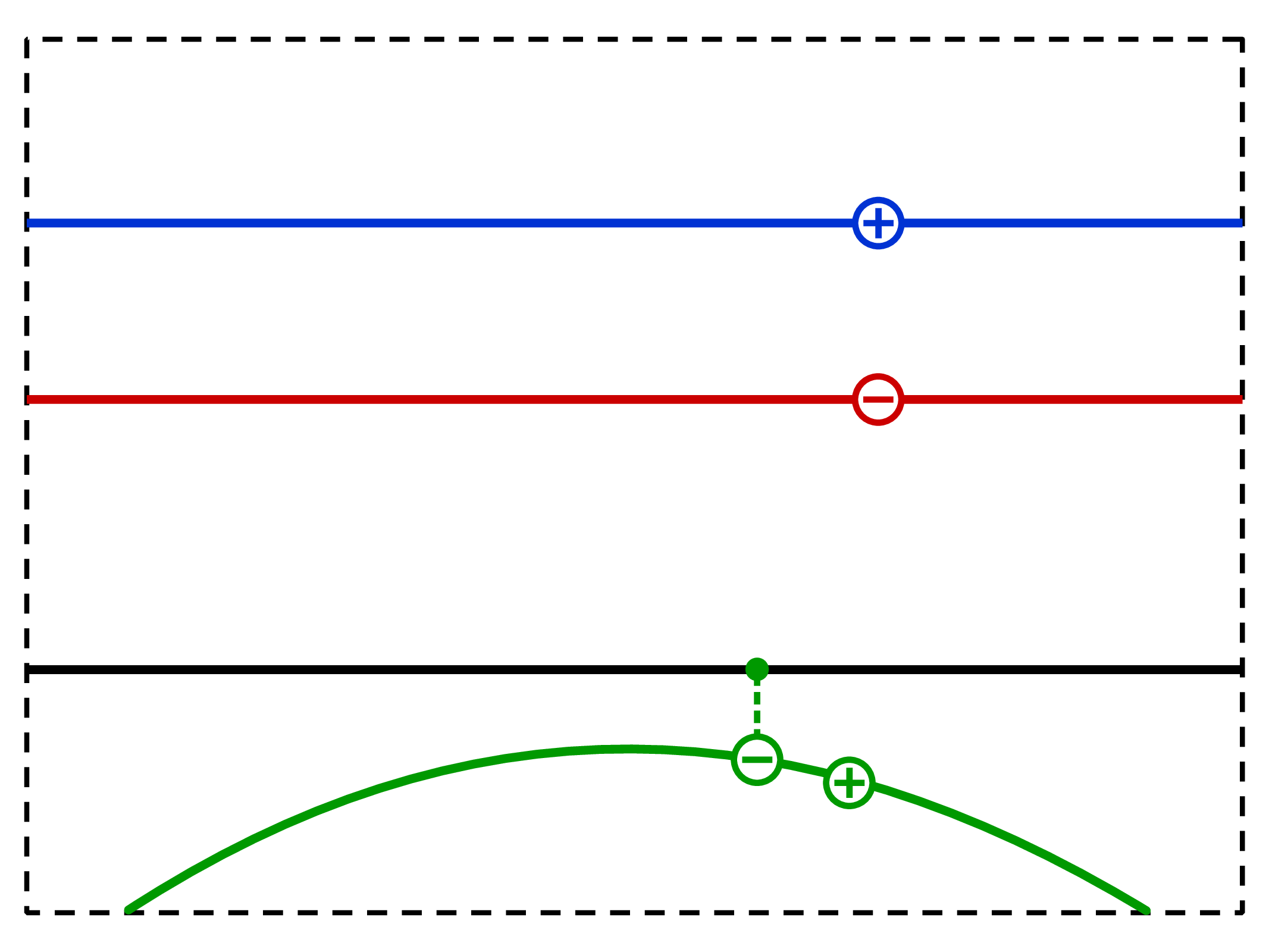}
	 \put(50,-2){\small $\Gamma'$}
     \put(93,5){\small $\Gamma$}
     \put(87,26){\scriptsize $(-1)$}

    \put(4,25){\small $\Lambda_0$}
    \put(4,46){\small \textcolor{darkred}{$\Lambda_-$}}
    \put(4,60){\small \textcolor{darkblue}{$\Lambda_+$}}
    \put(5,6){\small \textcolor{darkgreen}{$\Lambda_1$}}

     \put(75,14){\scriptsize \textcolor{darkgreen}{$(+1)$}}
	\end{overpic}
	\caption{A convex hypersurface $\Sigma' = R_+' \cup \Gamma' \cup R_-'$ is obtained from $\Sigma = R_+ \cup \Gamma \cup R_-$ by attaching a contact $n$-handle along $\Lambda_0$ and a contact ($n+1$)-handle along a sphere with equator $\Lambda_1$. The negative disk filling of $\Lambda_1$ intersects the cocore of the negative Weinstein $n$-handle attached along $\Lambda_0$. The Legendrian $\Lambda_{\pm}$ has a preferred disk filling in $R_{\pm}$, and no surgery or handle attachment is being performed.}
	\label{fig:dec}
\end{figure}

To aid with the need to track additional data in a surgery presentation of the dividing set, we adopt the conventions of \cite{BHH23}, presented below, regarding diagrammatic decorations associated to contact ($n+1$)-handles, though we emphasize that the decorations alone are still not sufficient to determine the handle attachment.

\begin{enumerate}
    \item If a Legendrian sphere $\Lambda\subset \Gamma$ in the dividing set is filled by a preferred Lagrangian disk in $R_+$, we decorate the strand representing $\Lambda$ with a circular $(+)$-node. Likewise, we use a circular $(-)$-node to indicate a preferred disk filling in $R_-$.\footnote{This is meant to be reminiscent of dotted circle notation for $1$-handles in $4$-dimensional Kirby calculus.}

    \item Contact ($n+1$)-handle attachment along a sphere with equator $\Lambda \subset \Gamma$ is indicated by the presence of both a $(+)$-node and $(-)$-node together with a contact-$(+1)$ surgery coefficient.

    \item If there is a contact $n$-handle present in the diagram, and a preferred Lagrangian disk filling in $R_{\pm}'$ has nonzero geometric intersection with the cocore of the associated Weinstein $n$-handle attached to $R_{\pm}$, a dashed line connects the $(\pm)$-node with the $n$-handle attaching sphere. 
\end{enumerate}
See \cref{fig:dec} for general examples of these conventions, and \cref{ex1} for a concrete example.

\begin{example}\label{ex1}
For $n\geq 1$, let $\D^{2n}$ be the standard Weinstein structure on the ball with one critical point, so that $\partial \D^{2n}$ is the standard tight contact $S^{2n-1}$. Attaching a Weinstein $n$-handle along the standard Legendrian unknot $\Lambda\subset S^{2n-1}$ gives the standard Weinstein structure on $\D^*S^n$, the unit disk cotangent bundle. Let $H=[-1,1]\times \D^*S^n$ be the contact handlebody based over $\D^*S^n$, so that $R_{\pm}$ of $\partial H$ is identified with $\D^*S^n$.

\begin{figure}[ht]
	\begin{overpic}[scale=0.43]{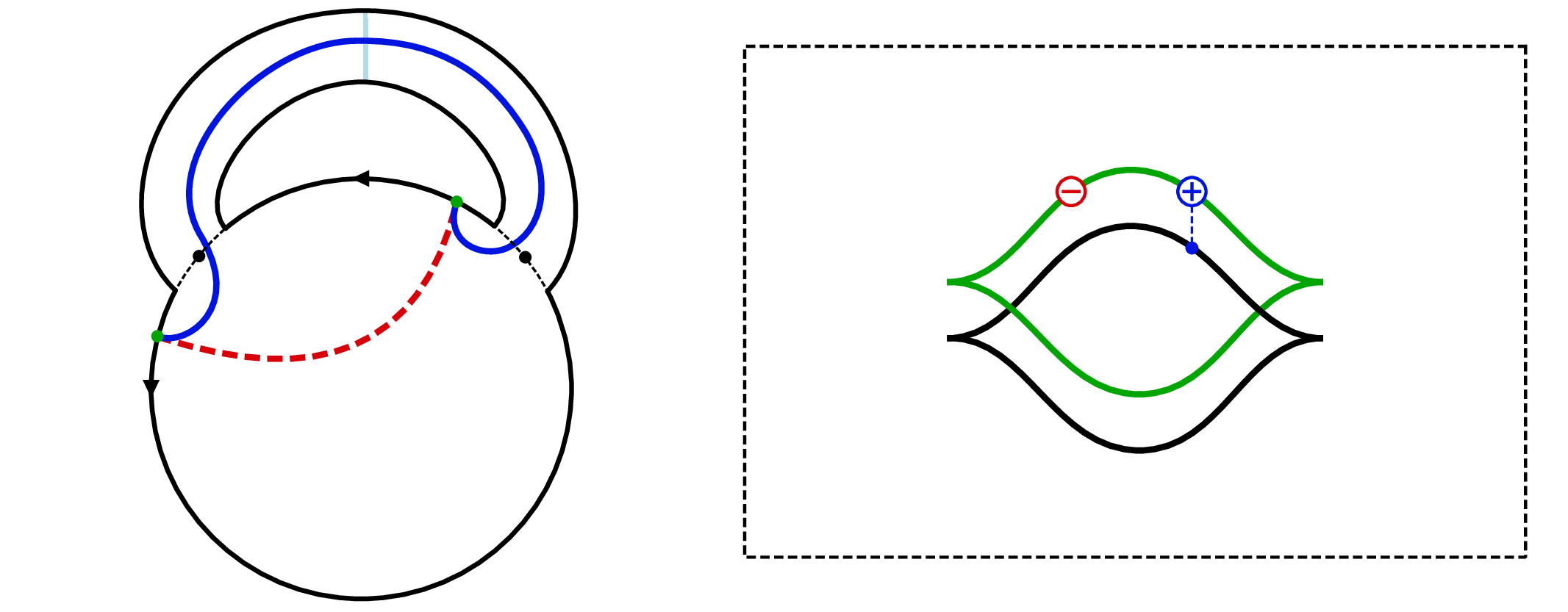}
	 \put(71,0.5){\small $S^{2n-1}$}
     \put(90.5,5){\small $S^{2n-1}$}

    \put(57,16.75){\small $\Lambda$}
    \put(22,13.5){\small \textcolor{darkred}{$D_-$}}
    \put(15,20){\small \textcolor{darkblue}{$D_+$}}
    \put(57,20.25){\small \textcolor{darkgreen}{$\Lambda'$}}

     \put(86,20.5){\scriptsize \textcolor{darkgreen}{$(+1)$}}
     \put(86,17){\scriptsize $(-1)$}
	\end{overpic}
	\caption{Diagrams for \cref{ex1}. On the left is the $n=1$ case; we visualize the abstract contact handlebody $[-1,1]\times \D^*S^1$ as having convex boundary with $R_{\pm}$ on the front (resp. back) of the page. On the right is the $n>1$ case.}
	\label{fig:example1}
\end{figure}
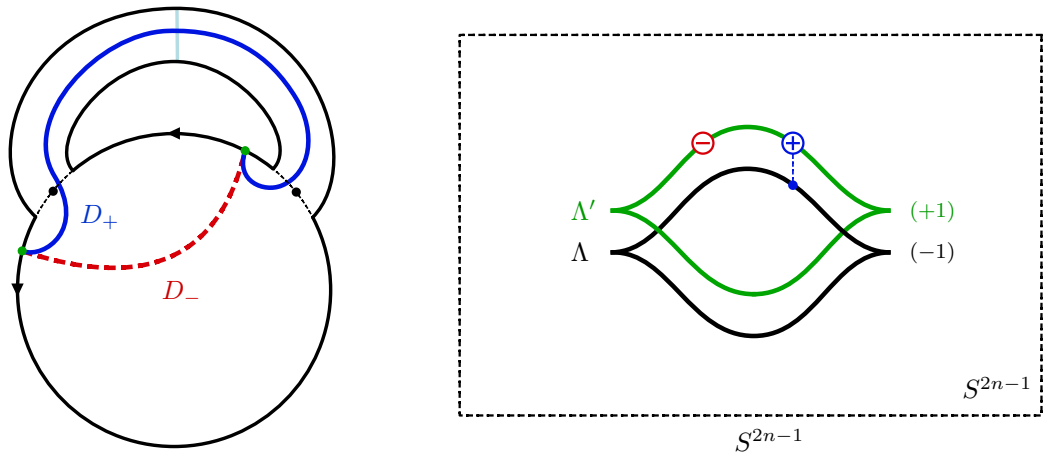

We can attach a contact ($n+1$)-handle to $\partial H$ with the following data. Let $\Lambda'$ be a small positive-Reeb shift of $\Lambda$ in the dividing set $\Gamma = \partial \D^*S^n$. Note that $\Lambda'$ is also a standard unknot in $S^{2n-1}$, hence it admits a standard Lagrangian disk filling $D_-$ in the $0$-handle of $R_-$. Then we observe that the belt sphere of the Weinstein $n$-handle in $R_+$ is isotopic in $\Gamma$ to $\Lambda'$; we let $D_+$ be the image of the cocore under a Hamiltonian isotopy in $R_+$ that induces such an isotopy in the dividing set. 

Let $(M, \xi)$ be the resulting contact manifold obtained by attaching a contact ($n+1$)-handle to $H$ along $D_+ \cup D_-$. This handle smoothly cancels the $n$-handle, so $M$ is diffeomorphic to a ball. The contact type of the dividing set of $\partial M = S^{2n}$ is given by the surgery diagram in \cref{fig:example1}. The surgeries cancel, meaning the dividing set is a copy of standard tight $S^{2n-1}$. We will see below (c.f. \cref{lemma:trivial_bypass_lemma}) that, in fact, the contact handles cancel, but this is not obvious at this stage. 
\end{example}

\subsection{Bypass attachments}\label{subsec:Bypass}

Bypass attachments in higher dimensional contact topology were first introduced in \cite{HH18}. This subsection adopts much of the preliminary material from \cite[\S 2]{BHH23}.

\begin{definition}\label{def:bypass}
Let $\Sigma$ be a convex hypersurface with convex decomposition $\Sigma = R_+ \cup \Gamma \cup R_-$. \textbf{Bypass attachment data} is a tuple $(\Lambda_{-}, \Lambda_+; D_{-}, D_+)$ where 
$D_{\pm}\subset R_{\pm}$ is a properly embedded Lagrangian disk filling of $\Lambda_{\pm}\subset\Gamma$ such that $\Lambda_-$ and $\Lambda_+$ intersect $\xi_{\Gamma}$-transversely at one point; see the middle panel of \cref{fig:bypass_attachment}. The bypass attachment data is \textbf{trivial} if either of the following holds: 
\begin{itemize}
    \item[(TB1)] $(\Lambda_+; D_+)$ is a pair consisting of a standard Legendrian unknot and standard Lagrangian disk such that $\Lambda_+$ is \textbf{below} $\Lambda_-$, by which we mean that for every $\ve > 0$,  $\Lambda_+^{-\ve}$ bounds an $n$-dimensional Seifert ball $\Theta^{-\ve}$ in a Darboux ball such that $\Lambda_- \cap \Theta^{-\ve} = \emptyset$. When $n=1$, this means that $D_+ \cup \Gamma$ forms a bigon in $R_+$ whose upper corner (measured with respect to the Reeb direction of $\Gamma$) is $\Lambda_+ \cap \Lambda_-$; see the top left panel of \cref{fig:trivPOBDstab}.
    \item[(TB2)] $(\Lambda_-; D_-)$ is a pair consisting of a standard Legendrian and standard Lagrangian disk such that $\Lambda_-$ is \textbf{above} $\Lambda_+$, defined analogously.
\end{itemize}
See \autoref{fig:trivial_data}. A bypass attachment with trivial bypass data is also called \textbf{trivial}.   
\end{definition}

 \begin{figure}[ht]
	\begin{overpic}[scale=0.4]{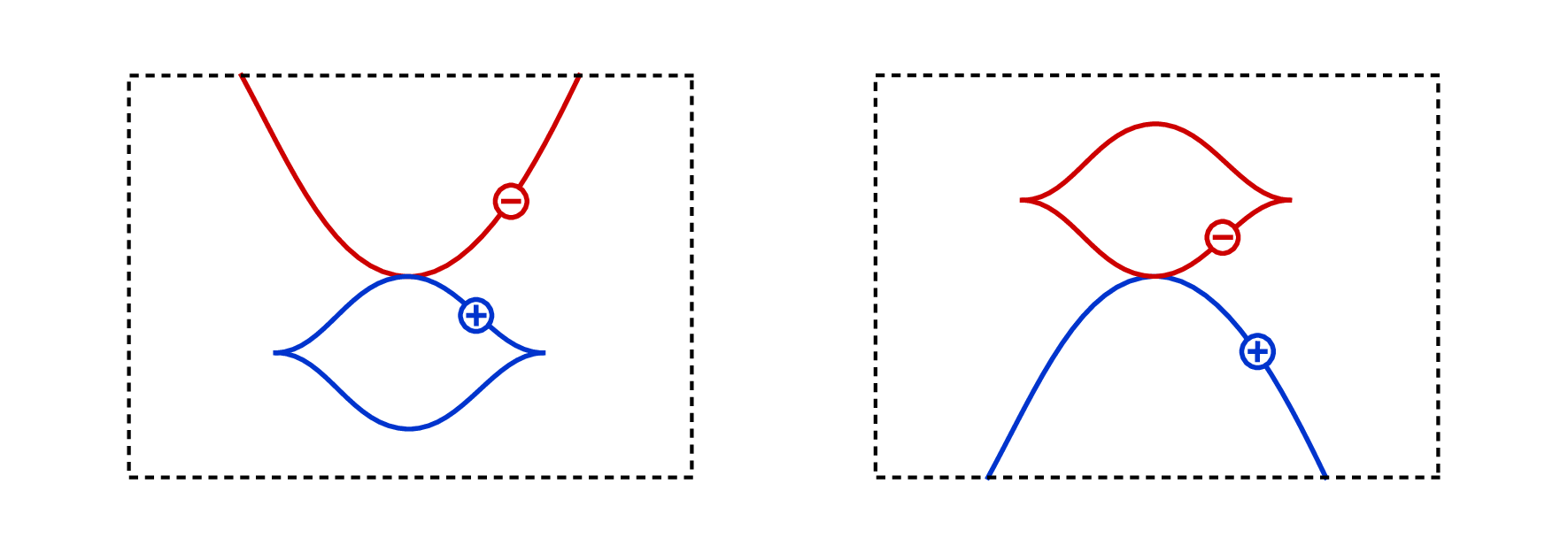}
	    \put(23,31.5){\small (TB1)}
	    \put(25,2){\small $\Gamma$}
	    \put(41,6){\small $\Gamma$}
	    \put(71,31.5){\small (TB2)}
	    \put(73,2){\small $\Gamma$}
	    \put(89,6){\small $\Gamma$}
	\end{overpic}
	\caption{Trivial bypass data; figure adopted from \cite{BHH23}.}
	\label{fig:trivial_data}
\end{figure}
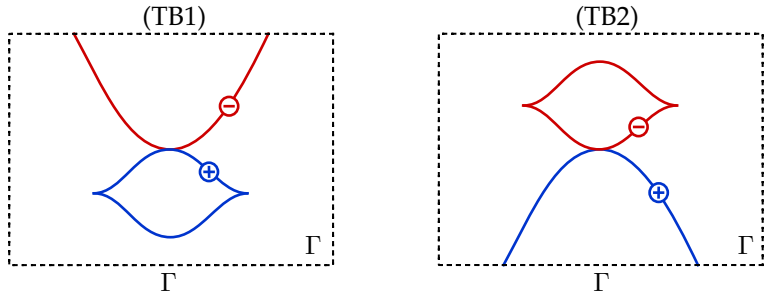

The following theorem specifies the meaning of \textit{bypass attachment} as a pair of contact handles. It is a combination of \cite[Theorem 5.1.3, Proposition 8.3.2]{HH18}.

\begin{theorem}[Bypass attachment]\label{theorem:bypass_attachment}
Let $\Sigma^{2n} = R_+ \cup \Gamma \cup R_-$ be convex with bypass attachment data $(\Lambda_{-}, \Lambda_+; D_{-}, D_+)$. Then a pair of smoothly canceling contact $n$- and ($n+1$)-handles can be attached according to either of the following models:
\begin{enumerate}
    \item[($R_+$)] Attach a contact $n$-handle along $\Lambda_- \uplus \Lambda_+$ and a contact ($n+1$)-handle along $\tilde{D}_- \cup \Lambda_+^{-\ve} \cup D_+^{-\ve}$, where $\tilde{D}_-$ is obtained by sliding $D_-^{\ve}$ down across the $(-1)$-surgery along $\Lambda_- \uplus \Lambda_+$. 
    \item[($R_-$)] Attach a contact $n$-handle along $\Lambda_- \uplus \Lambda_+$ and a contact ($n+1$)-handle along $D_-^{\ve} \cup \Lambda_-^{\ve} \cup \tilde{D}_+$, where $\tilde{D}_+$ is obtained by sliding $D_+^{-\ve}$ up across the $(-1)$-surgery along $\Lambda_- \uplus \Lambda_+$. 
\end{enumerate}
The two models, called the {\em $R_{\pm}$-centric models,} respectively, are identified by a handleslide of the ($n+1$)-handle across the $n$-handle; see \autoref{fig:bypass_attachment}. Moreover, if the bypass attachment data is trivial, the resulting contact cobordism $\Sigma \times [0,1]$ has a vertically invariant contact structure.
\end{theorem}

\begin{figure}[ht]
	\begin{overpic}[scale=0.51]{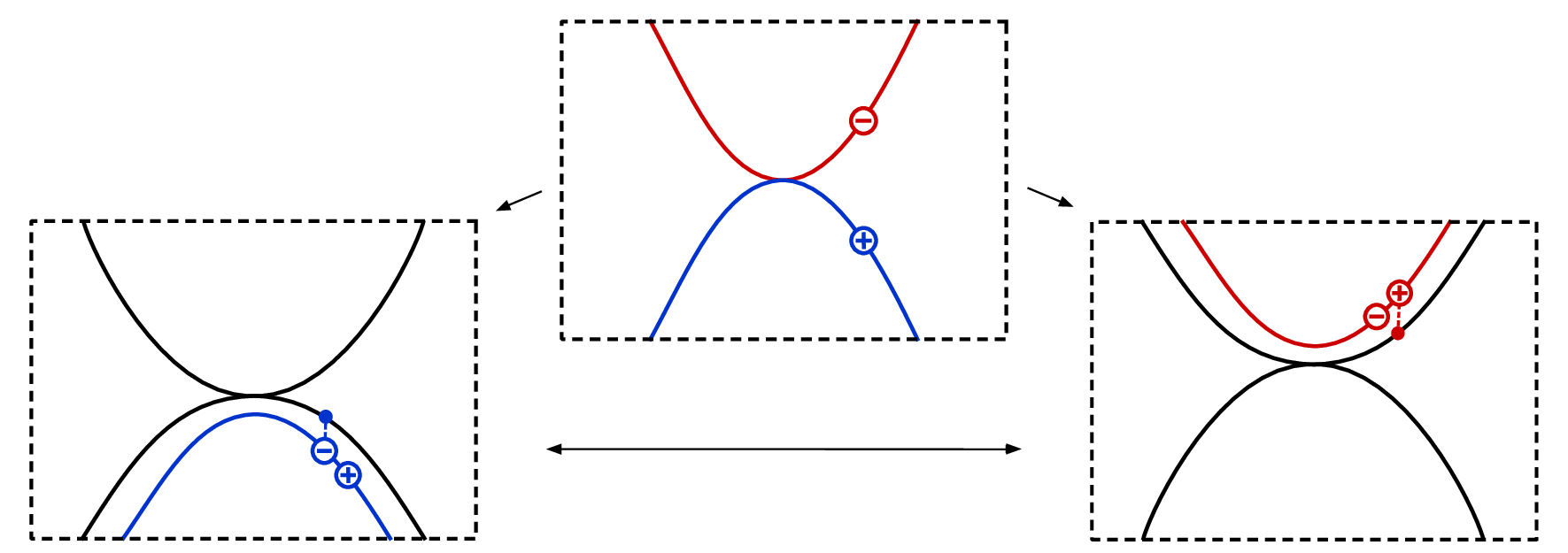}
	    \put(36.75,35){\small (Bypass attachment data)}
	    \put(49.5,12){\tiny $\Gamma$}
	    \put(62.5,14.5){\tiny $\Gamma$}
	    \put(45,29.5){\small {\color{darkred}  $\Lambda_-$}}
	    \put(45,18){\small \textcolor{darkblue}{$\Lambda_{+}$}}
	    \put(78,22){\small ($R_-$-centric)}
	    \put(83,-1){\tiny $\Gamma'$}
	    \put(96,1.5){\tiny $\Gamma$}
	    \put(92.75,12){\tiny $(-1)$}
	    \put(92.75,15){\tiny {\color{darkred}  $(+1)$}}
	    \put(79.5,17){\small {\color{darkred}  $\Lambda_-^{\ve}$}}
	    \put(79,8){\small $\Lambda_- \uplus \Lambda_+$}
	    \put(10.5,22){\small ($R_+$-centric)}
	    \put(15,-1){\tiny $\Gamma'$}
	    \put(28.5,1.5){\tiny $\Gamma$}
	    \put(25,10){\tiny $(-1)$}
	    \put(25,7){\tiny \textcolor{darkblue}{$(+1)$}}
	    \put(12,3.5){\small \textcolor{darkblue}{$\Lambda_+^{-\ve}$}}
	    \put(11.5,13.5){\small $\Lambda_- \uplus \Lambda_+$}
		\put(45,3.75){\small Handleslide}	
	\end{overpic}
	\caption{Two equivalent models of a bypass attachment; figure adopted from \cite{BHH23}.}
	\label{fig:bypass_attachment}
\end{figure}
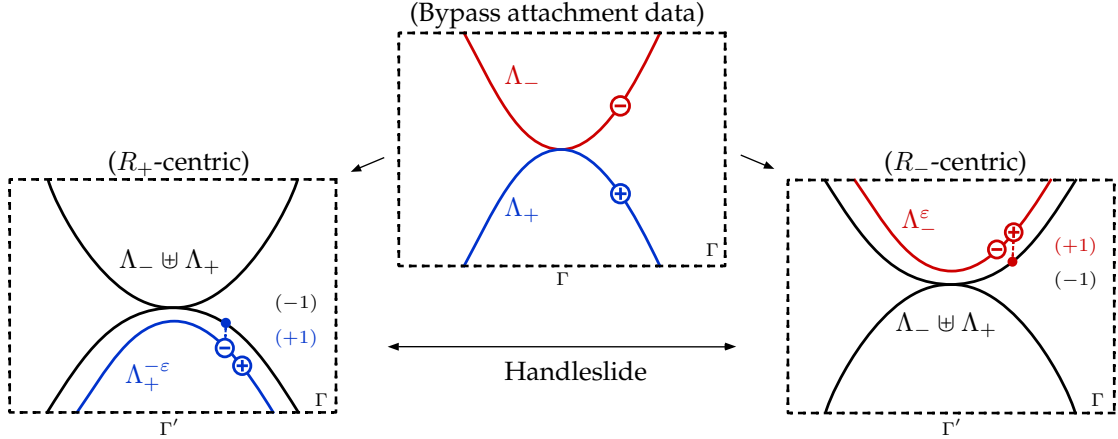

The following lemma lets us identify handle pairs that form trivial bypasses in a decorated Legendrian diagram. 

\begin{lemma}\cite[Lemma 2.5.3]{BHH23}\label{lemma:trivial_bypass_lemma}
Let $\Sigma^{2n} = R_+ \cup \Gamma \cup R_-$ be the convex boundary of $(M^{2n+1}, \xi)$. Let $h_n$ be a contact $n$-handle attached to $M$ along $\Lambda_n \subseteq \Gamma$, and let $h_{n+1}$ be a contact ($n+1$)-handle attached to $M \cup h_n$ along $D_+ \cup \Lambda_{n+1} \cup D_-$. Then the pair of handles $h_n \cup h_{n+1}$ forms a trivial bypass if either of the following conditions hold; see \autoref{fig:trivialbypasslemma}:
\begin{itemize}
    \item[(TB1)] $D_+$ is a positive Reeb shift of the cocore of the $n$-handle attached to $R_+$, and $D_- \subseteq R_-$. 
    \item[(TB2)] $D_-$ is a negative Reeb shift of the cocore of the $n$-handle attached to $R_-$, and $D_+ \subseteq R_+$.
\end{itemize}
\end{lemma}

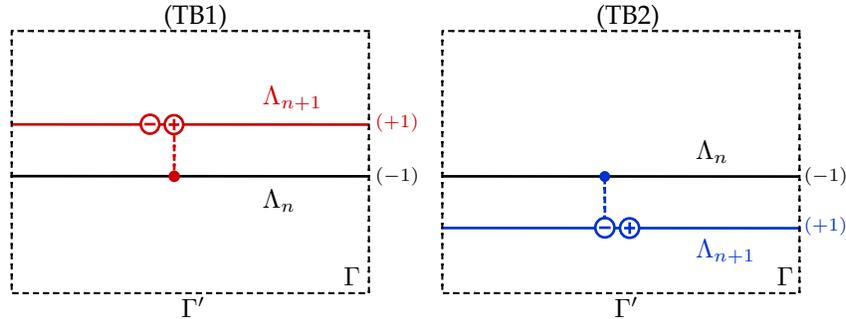
\begin{figure}[ht]
	\begin{overpic}[scale=0.4]{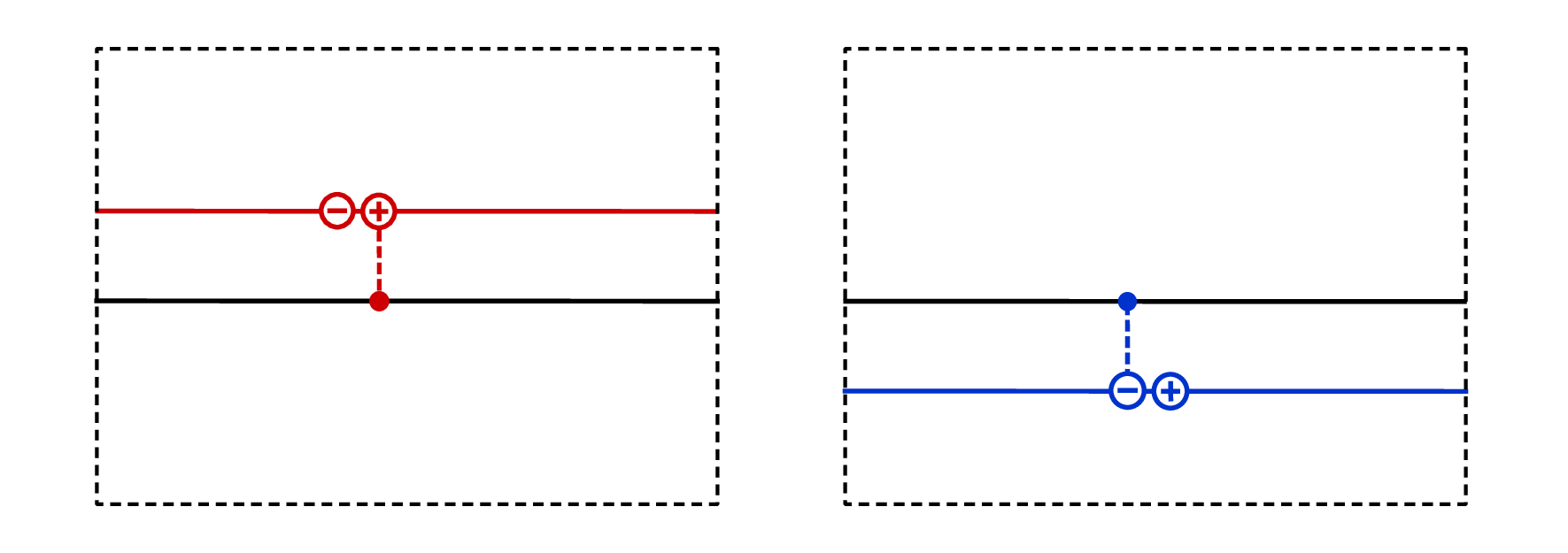}
	    \put(23,33){\small (TB1)}
	    \put(25,0.5){\small $\Gamma'$}
	    \put(43,4){\small $\Gamma$}
	    \put(34,12.5){\small $\Lambda_n$}
	    \put(34,24){\small {\color{darkred}  $\Lambda_{n+1}$}}
	    \put(46.5,15.5){\tiny $(-1)$}
	    \put(46.5,21.5){\tiny {\color{darkred}  $(+1)$}}
	    \put(71,33){\small (TB2)}
	    \put(73,0.5){\small $\Gamma'$}
	    \put(91,4){\small $\Gamma$}
	    \put(94,15.5){\tiny $(-1)$}
	    \put(94,10){\tiny \textcolor{darkblue}{$(+1)$}}
	    \put(82,18){\small $\Lambda_n$}
	    \put(82,7){\small \textcolor{darkblue}{$\Lambda_{n+1}$}}
	\end{overpic}
	\caption{The two trivial bypasses described in \cref{lemma:trivial_bypass_lemma}; figure adopted from \cite{BHH23}.}
	\label{fig:trivialbypasslemma}
\end{figure}

\subsection{Open book decompositions}\label{subsec:OBD}

Here we define (partial) open book decompositions. In dimension $3$, an \textit{open book decomposition} of $(M^3, \xi)$ is a choice of transverse fibered link, and admits an abstract description as the relative mapping torus of the fiber monodromy. Giroux's work in  \cite{Giroux2002GeometrieDC} crystallized the strong connection between contact structures and open book decompositions; we refer to \cite{etnyre2006lectures,geiges2008introduction,licata2023heegaard} for more details on open books, particularly in dimension $3$. In higher dimensions, the relevant definition is as follows. 

\begin{definition}\label{def:obd}
Let $(M^{2n+1}, \xi)$ be a closed contact manifold. An \textbf{open book decomposition} of $(M, \xi)$ is a pair $(B, \pi)$ where 
\begin{enumerate}
    \item $B^{2n-1}\subset M$ is a codimension-$2$ contact submanifold called the \textbf{binding},
    \item $\pi: M \setminus B \to S^1$ is a fibration that agrees with an angular coordinate in a tubular neighborhood $B\times D^2$ of $B$, and 
    \item there is a Reeb vector field $R_{\alpha}$ for $\xi$ everywhere transverse to each \textbf{page} $\pi^{-1}(\theta)$. 
\end{enumerate}
If $(\pi^{-1}(\theta), \alpha\vert_{\pi^{-1}(\theta)})$ admits the structure of a completed $1$-Weinstein domain (i.e., allowing for birth-death critical points) for every $\theta\in S^1$, we say that the open book decomposition is \textbf{strongly Weinstein}.
\end{definition}

Fixing $\theta_0\in S^1$, the monodromy $\psi$ of the fibration $\pi$ is an exact symplectomorphism of $(\pi^{-1}(\theta_0), \alpha\vert_{\pi^{-1}}(\theta_0))$, i.e., $\psi^*\alpha\vert_{\pi^{-1}}(\theta_0) = \alpha\vert_{\pi^{-1}}(\theta_0) + df$, which is the identity near the binding. This gives rise to an \textit{abstract open book}, i.e. a pair $(W, \psi)$ where $W$ is a Liouville domain and $\psi:W \to W$ is an exact symplectomorphism that is the identity in a neighborhood of $\partial W$. The relative mapping torus of the pair $(W, \psi)$ produces a closed contact manifold that we denote $\mathcal{OB}(W,\psi)$, together with an open book decomposition in the sense of \cref{def:obd}.

\begin{remark}
It is unclear when, in general, a pair $(W, \psi)$ gives a strongly Weinstein open book decomposition of $\mathcal{OB}(W,\psi)$. This is the case if $\psi$ is the product of positive symplectic Dehn twists, as $\mathcal{OB}(W,\psi)$ is the boundary of a Weinstein Lefschetz fibration, but it is unknown to what extent the exact symplectic mapping class group of a Weinstein domain is generated by Dehn twists.   
\end{remark}

There are relative versions of (abstract) open books for contact manifolds with boundary, first considered in dimension $3$ by Honda-Kazez-Matić \cite{honda2009sutured} and generalized to all dimensions by Honda-Huang \cite{HH18}. For now, we content ourselves with the abstract version. An \textit{abstract partial open book} is the data $(W, \psi:S \to W)$, where $W$ is a Weinstein domain, $S\subset W$ is a Weinstein subcobordism with $\partial_+ S = \partial W$, and $\psi:S \to W$ is an exact symplectic embedding that is the identity near $\partial_+ S$. The relative mapping torus then produces a contact manifold with boundary that we denote $\mathcal{POBD}(W, \psi:S \to W)$. After rounding corners, the boundary $\Sigma$ is convex with decomposition $R_+ \cup \Gamma \cup R_-$ where 
\begin{align*}
    R_+ &= W \setminus S, \\
    R_- &= W \setminus \psi(S).
\end{align*}

Both open books and partial open books admit a positive stabilization operation that does not affect the contactomorphism type of the resulting contact manifold; see \cite{vankoert2017lecture} for details. In the closed case, \textit{positive stabilization} of an abstract open book $(W, \psi)$ is the passage to $(W \cup h, \tau \circ \tilde{\psi})$, where 
\begin{enumerate}
    \item $D\subset W$ is a regular Lagrangian disk properly embedded in the page, 
    \item $h$ is a critical Weinstein handle attached to $\partial W$ along $\partial D$, 
    \item $\tilde{\psi}: W \cup h \to W \cup h$ is the extension of $\psi$ across $h$ by the identity, and
    \item $\tau: W \cup h \to W \cup h$ is the positive symplectic Dehn twist around the exact Lagrangian sphere $D\cup C$, where $C$ is the core of $h$. 
\end{enumerate}
Positive stabilization has additional subtleties in the relative case. Given a partial open book $(W, \psi:S \to W)$ and a regular Lagrangian disk $D\subset W$, the partial open books 
\begin{align*}
    (W \cup h,& \,\,\, \tau \circ \tilde{\psi}: S\cup h\, \to\, W\cup h),\\
    (W \cup h,& \,\,\,\tilde{\psi}\circ \tau: \tau^{-1}(S) \cup N(D) \,\to\, W \cup h)
\end{align*}
are called the \textit{(TB1)} and \textit{(TB2) positive stabilizations along $D$}, respectively. The terminology is for consistency with \cref{def:bypass}, \cref{lemma:trivial_bypass_lemma}, and the following: 

\begin{lemma}\cite[Lemma 2.5.5]{BHH23}
Let $(M,\xi)= \mathcal{POBD}(W, \phi: S \to W)$ be a partial open book with convex boundary.
\begin{enumerate}
    \item A (TB1) trivial bypass attachment to $M$ with data $(\Lambda_{\pm}; D_{\pm})$ corresponds to a (TB1) positive stabilization of $(W, \phi: S \to W)$ along $D_-$. 

    \item A (TB2) trivial bypass attachment to $M$ with data $(\Lambda_{\pm}; D_{\pm})$ corresponds to a (TB2) positive stabilization of $(W, \phi: S \to W)$ along $D_+$. 
\end{enumerate}
\end{lemma}

We refer to \cite{BHH23} for a proof, but include \cref{fig:trivPOBDstab} for the idea.

\begin{figure}[ht]
\vskip-0.8in
	\begin{overpic}[scale=0.43]{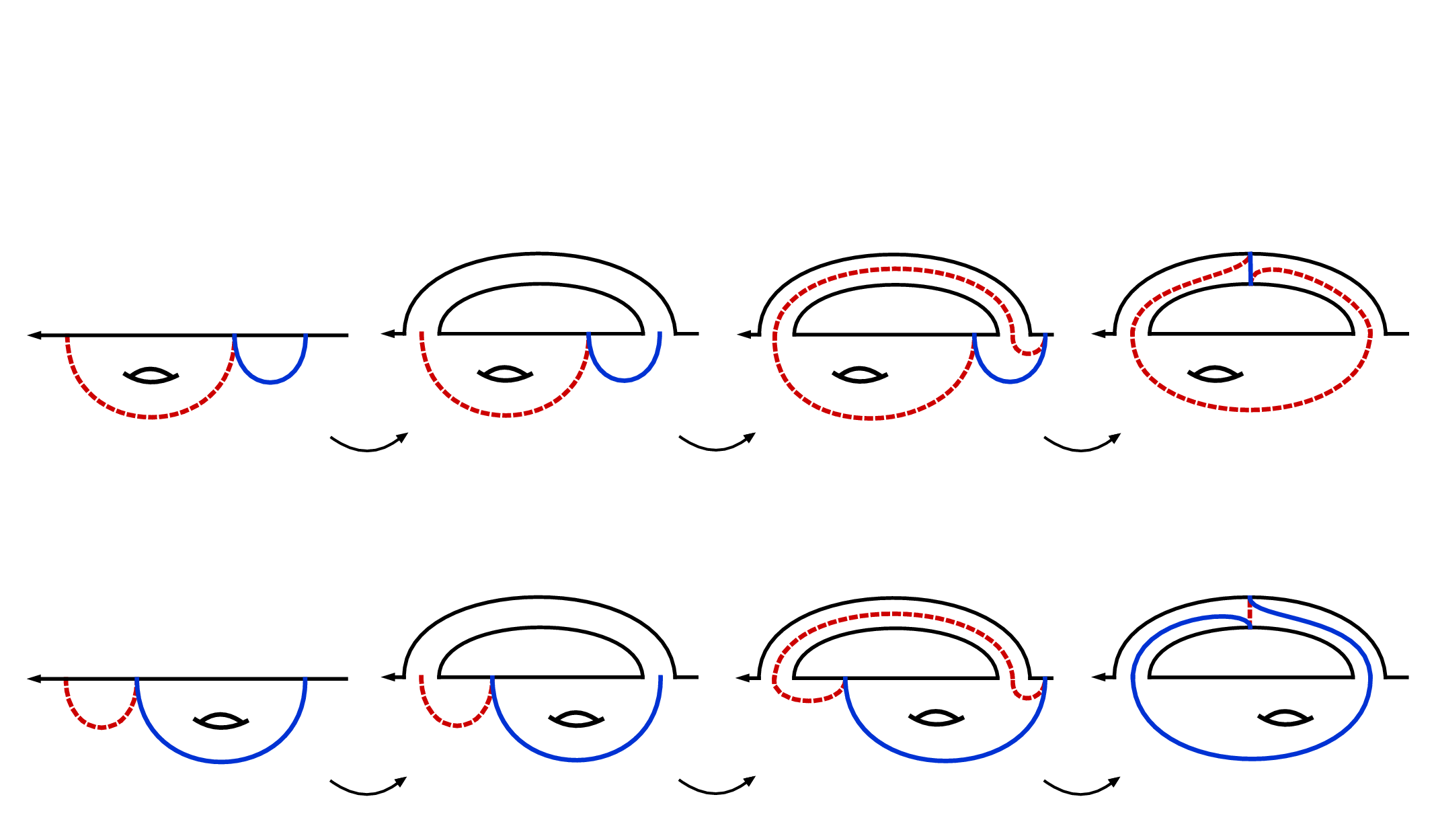}
	    \put(5,42.5){\small (TB1)}
	    \put(5,28){\small {\color{darkred}  $D_-$}}
	    \put(20,30){\small \textcolor{darkblue}{$D_+$}}
	    \put(21,25.25){\tiny attach $n$-handle}
	    \put(42.5,25.25){\tiny identify attaching sphere}
	    \put(44.5,23.25){\tiny of ($n+1$)-handle}
	    
	    \put(68.5,29.5){\small \textcolor{darkblue}{$D_+^{-\ve}$}}
	    \put(5,18){\small (TB2)}
	    \put(5,5){\small {\color{darkred}  $D_-$}}
	    \put(17.5,3.25){\small \textcolor{darkblue}{$D_+$}}
	    \put(21,1.25){\tiny attach $n$-handle}
	    \put(42.5,1.25){\tiny identify attaching sphere}
	    \put(44.5,-0.75){\tiny of ($n+1$)-handle}
	    
	    \put(66,2.75){\small \textcolor{darkblue}{$D_+^{-\ve}$}}
	\end{overpic}
	\caption{Trivial bypasses yielding partial open book stabilizations; figure adopted from \cite{BHH23}.}
	\label{fig:trivPOBDstab}
\end{figure}
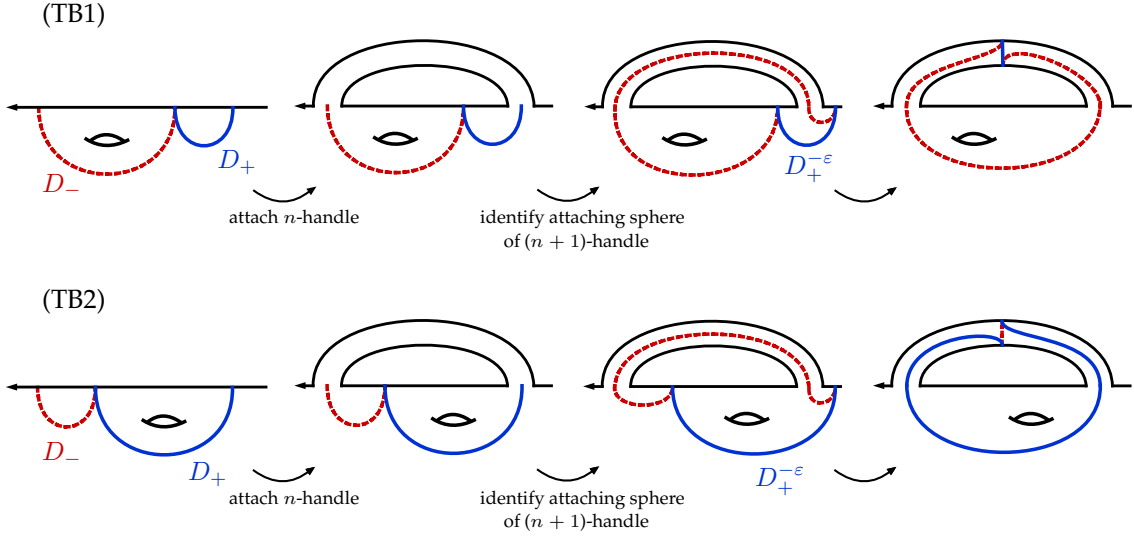

\subsection{Open books, contact handles, and bypass decompositions}\label{subsec:OBDandBypass}

Finally, we discuss the connection between open books and contact handles. The following diagram gives an informal summary:

\begin{center}
\begin{tikzcd}
  \{\text{bypass decompositions}\} \arrow[r,hookrightarrow] \arrow[d,leftrightsquigarrow]
    & \{\text{contact handle decompositions}\} \arrow[d,rightsquigarrow] \\
  \{\text{strongly Weinstein open book decompositions}\} \arrow[r,hookrightarrow] &
\{\text{open book decompositions}\} 
\end{tikzcd}    
\end{center}

\noindent A \textit{bypass decomposition} of a closed contact manifold $(M, \xi)$, also called a \textit{$\Theta$-decomposition} or \textit{mushroom burger} in \cite{honda2019convex,BHH23}, is 
\[
(M, \xi) = H_0 \cup (\Sigma \times [0,1]) \cup H_1
\]
where $H_0$ and $H_1$ are contact handlebodies each with convex boundary diffeomorphic to $\Sigma$, and $\Sigma \times [0,1]$ is given by a specified sequence of bypass attachments. A bypass decomposition induces a contact handle decomposition of $(M, \xi)$ by starting with the contact $k$-handles from $H_0$, attaching the contact $n$- and ($n+1$)-handles from the bypass layer, then attaching each contact $k$-handle in $H_1$ as an upside-down contact ($2n+1-k$)-handle.

\begin{theorem}\cite{Giroux2002GeometrieDC,honda2019convex}
Let $(M, \xi)$ be a closed contact manifold. A contact handle decomposition naturally induces an open book decomposition. Moreover, if the contact handle decomposition arises from a bypass decomposition, the open book decomposition is strongly Weinstein.    
\end{theorem}

\begin{proof}[Proof idea.]
Given a contact handle decomposition, we may shuffle the order of attachment so that all handles of index at most $n$ are attached before handles of index at least $n+1$. The resulting contact handlebodies give a higher dimensional contact Heegaard splitting of $(M, \xi)$, which induces an open book decomposition. 

To show that bypass decompositions induce strongly Weinstein open book decompositions, we appeal to the relative case. By \cite[Lemma 8.3.1]{HH18} and \cite[Lemma 8.4.3]{honda2019convex}, attaching a bypass to a \textit{strongly Weinstein} partial open book yields a new strongly Weinstein partial open book. In a bypass decomposition, we view $H_0$ as a trivial partial open book with empty monodromy --- that is, as a thickening $W\times [-1,1]$ of a page with no gluing occurring --- which in particular is strongly Weinstein. Each bypass attachment in $\Sigma \times [0,1]$ preserves the strongly Weinstein property, as does the final upside-down $H_1$ attachment.  
\end{proof}

\subsection{Flexibility and looseness of Weinstein hypersurfaces}\label{subsec:flex}

Important to our discussion of bypasses are the notions of looseness and flexibility, the former of which is usually considered for Legendrian submanifolds, the latter for Weinstein domains.  Either adjective --- loose or flexible --- may be applied to Weinstein hypersurfaces of appropriate dimensions, and we collect the relevant definitions due to Eliashberg \cite{eliashberg2018weinsteinrevisited} in this subsection.  Throughout, we consider a Weinstein hypersurface $(W,\alpha\vert_W,\phi)$ of dimension $2n$ embedded in a contact manifold $(M,\xi=\ker\alpha)$ of dimension $2n+1$.

Recall that Murphy defined in \cite{murphy2012loosev2} the notion of a \emph{loose} Legendrian submanifold $\Lambda^n\subset(M,\xi)$, provided $n\geq 2$. A \emph{flexible} Weinstein domain of dimension $2n\geq 6$ is then one whose critical attaching locus $\Lambda^{n-1}$ is loose as a Legendrian submanifold of its ambient contact level set. Loose Legendrians and flexible Weinstein domains satisfy a number of $h$-principles \cite{cieliebak2012stein}.

\begin{definition}
Let $W$ be a Weinstein hypersurface in $(M^{2n+1},\xi)$.  For $2n\geq 6$, we call $W$ \textbf{flexible} if $(W, \alpha\vert_W, \phi)$ is a flexible Weinstein domain.  For $2n\geq 4$, we call $W$ \textbf{loose} if the $n$-dimensional strata of its skeleton $\mathrm{Skel}(W)$ are loose as Legendrian submanifolds of $(M,\xi)$.
\end{definition}

Looseness of a Weinstein hypersurface refers to its particular embedding into $(M,\xi)$, while flexibility refers to its intrinsic Weinstein topology. The following example is useful to keep in mind. 

\begin{example}\label{ex:loose}
For $2n+1 \geq 7$, let $\Lambda, \Lambda_{\mathrm{loose}} \subset (M^{2n+1}, \xi)$ be a standard and loose Legendrian unknot, respectively. Let $W, W_{\mathrm{loose}}\subset M$ be their canonical Weinstein ribbon neighborhoods. Neither $W$ nor $W_{\mathrm{loose}}$ are flexible hypersurfaces; both are deformation equivalent to the standard Weinstein structure on $T^*S^n$. However, $W_{\mathrm{loose}}$ is loose, while $W$ is not.  
\end{example}

Nonetheless, flexibility implies looseness where both phenomena are possible.

\begin{proposition}[{\cite[Proposition 4.4]{eliashberg2018weinsteinrevisited}}]\label{prop:flexible-implies-loose}
If $(M,\xi)$ is a contact manifold of dimension $2n+1\geq 7$, then any flexible Weinstein hypersurface in $(M,\xi)$ is loose.
\end{proposition}

This implication is restricted to dimensions 7 and above only because the definition of flexibility is restricted to Weinstein domains of dimension at least 6.  However, if a $4$-dimensional Weinstein hypersurface has critical attaching spheres that admit both positive and negative destabilizations, then the argument used by Eliashberg to prove \cref{prop:flexible-implies-loose} works just as well in dimension $5$.

Finally, we preview the role played by the stabilizing move of \cref{thm:main_moves} by offering an informal description of how one might "loosify" a Weinstein hypersurface while preserving its smooth isotopy type. Given the above discussion, there are two natural ways of doing so.

\begin{enumerate}
    \item (\emph{Loosification when $2n\geq 4$.}) Given a Weinstein hypersurface $W$ of dimension $2n\geq 4$, one may produce a loose Weinstein hypersurface $W'$ by installing a loose chart in the interior of --- that is, stabilizing --- each $n$-dimensional stratum of $\mathrm{Skel}(W)$.  Each of the resulting strata is formally Legendrian isotopic to its corresponding stratum in $\mathrm{Skel}(W)$, and thus we have a loose Weinstein thickening $W'$ which is smoothly isotopic and deformation equivalent to $W$; see \cref{ex:loose}.

    Notice that in case $2n=2$, "installation of a loose chart" corresponds to double stabilization of a Legendrian knot $\Lambda$, an operation which changes the contact framing along $\Lambda$, and thus the smooth isotopy type of its Weinstein ribbon.

    \item (\emph{Flexibilization when $2n\geq 6$.}) Given a Weinstein hypersurface $W$ of dimension $2n\geq 6$, one may instead loosify the ($n-1$)-dimensional attaching spheres of the ambient critical handles. The resulting $W'$ is a flexible  Weinstein hypersurface (hence loose by \cref{prop:flexible-implies-loose}) which is smoothly isotopic and almost symplectomorphic (but not deformation equivalent) to $W$.\label{part:flexibilization}
    
    In case $2n=4$, where the relevant Legendrian attaching spheres have dimension $1$, double stabilizing the attaching spheres yields a hypersurface $W'$ that will fail to be even diffeomorphic to $W$ (though the skeleta will still be smoothly isotopic).
\end{enumerate}

We will argue in the proof of \cref{thm:main_moves} that the stabilizing bypass move induces \eqref{part:flexibilization}. It is clear that the bypass move yields a Weinstein hypersurface deformation equivalent to $W'$ as described in \eqref{part:flexibilization}, but it takes more work to reason that the bypass attachment preserves the ambient isotopy type of $W$.

\section{Bypass moves}\label{sec:bypass_moves}

In this section we prove \cref{thm:main_moves}. In \cref{subsec:moves_all_dimensions} we break the theorem into \cref{prop:clasping_bypass} and \cref{prop:stabilizing_bypass} that respectively describe the (un)clasping and (de)stabilizing bypass moves. In \cref{subsec:moves_dim5} we specialize to dimension $2n+1=5$ and record an additional crossing change move, which is a consequence of the moves in \cref{subsec:moves_all_dimensions}. 

\subsection{Bypass moves in dimension $\geq 5$}\label{subsec:moves_all_dimensions}

The first bypass move describes how to (un)clasp two Legendrian sheets in a surgery diagram for the underlying Weinstein domain of a contact handlebody.

\begin{proposition}[The clasping bypass]\label{prop:clasping_bypass}
Let $(W_0, \lambda_0, \phi_0)$ be a Weinstein domain of dimension $2n\geq 4$. Suppose that in a surgery diagram for $W_0$ there is a chart with front projection given by the left side of \cref{fig:clasp}. Let $(W_1,\lambda_1, \phi_1)$ be the Weinstein domain obtained by modifying the surgery diagram as in the right side of \cref{fig:clasp}. Finally, let $H_i := W_i\times [-1,1]$, $i=0,1$, be the contact handlebody based over $W_i$. Then there are bypass attachments $B_{\mathrm{c}}$ and $B_{\mathrm{u}}$ such that 
\begin{align*}
    H_0 \cup B_{\mathrm{c}} &\cong H_1, \\
    H_1 \cup B_{\mathrm{u}} &\cong H_0.
\end{align*}
Here, $\cong$ indicates contactomorphism relative to the boundary. 
\end{proposition}

\begin{figure}[ht]
	\centering
	\begin{overpic}[scale=.36]{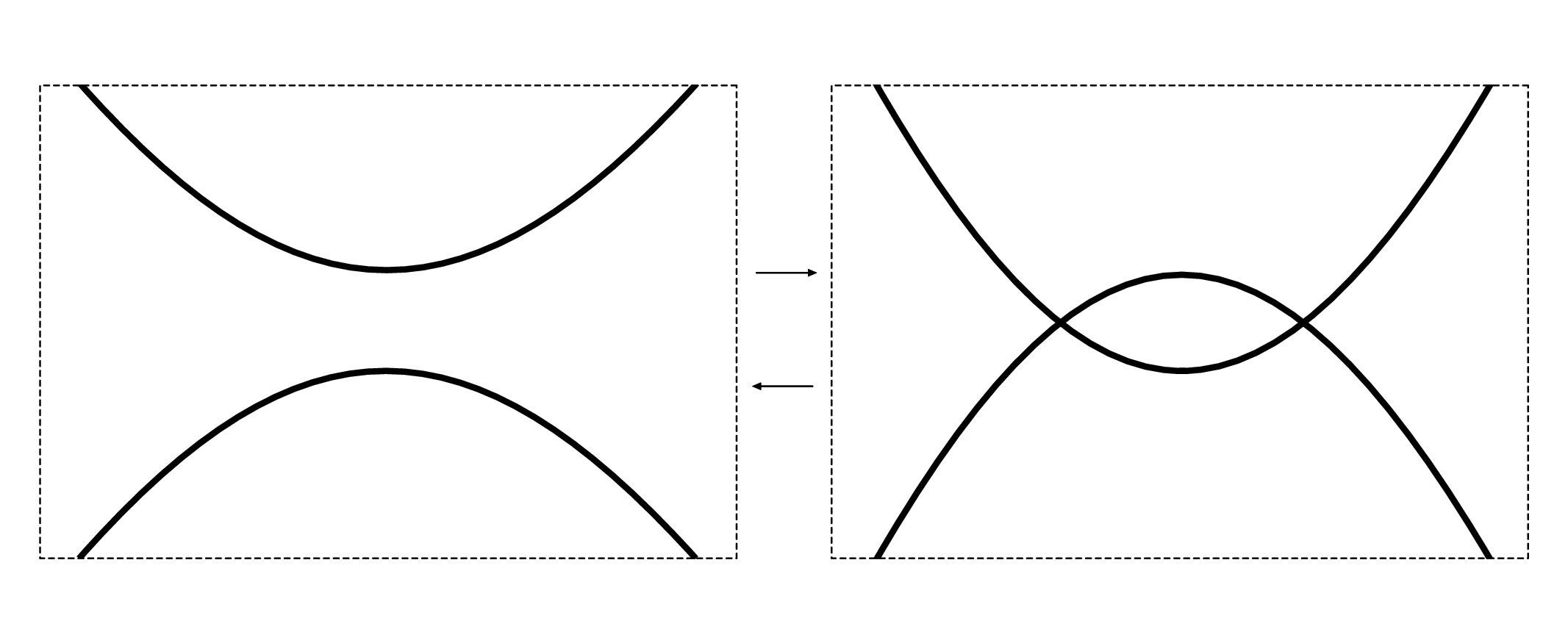}
 
	\put(48.75,24.75){\small $B_{\mathrm{c}}$}
    \put(48.75,14){\small $B_{\mathrm{u}}$}

    \put(41.5,30){\scriptsize $(-1)$}
    \put(41.5,10){\scriptsize $(-1)$}
    \put(92,28){\scriptsize $(-1)$}
    \put(92,12){\scriptsize $(-1)$}

    \put(23.75,2.5){$W_0$}
    \put(74.75,2.5){$W_1$}
 
	\end{overpic}
	\caption{The clasping and unclasping bypasses in \cref{prop:clasping_bypass}. When $2n\geq 6$, both diagrams are $S^{n-2}$-spun around the central vertical axis.}
	\label{fig:clasp}
\end{figure}
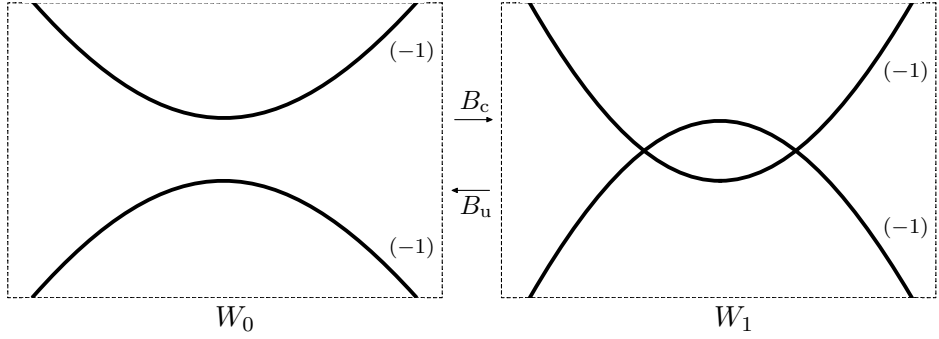

\begin{remark}
The bypasses in \cref{prop:clasping_bypass} are not trivial bypasses.    
\end{remark}

\begin{proof}
We begin by showing the existence of the clasping bypass $B_c$. Assume that $W_0$ has a local handlebody presentation as on the left side of \cref{fig:clasp}. We additionally assume that the two strands are associated to the same handle; the case where the strands belong to different Legendrians is similar, and we will comment on differences below as they appear in the proof. 

After smoothing corners, the contact handlebody $H_0=(W_0 \times [-1, 1]_t, \, \ker(dt + \lambda_0))$ has convex boundary $\Sigma$ with positive (resp. negative) region $R_{\pm}$, and the Weinstein handlebody presentation of $W_0$ induces a Weinstein handlebody presentation of both $R_{+}$ and $R_-$. Let $(\Lambda_+; D_+)$ be given by a small forward-time Reeb isotopy of the belt-sphere and co-core of the positive Weinstein $n$-handle of $R_+$; see the Legendrian decorated with the positive node in top left of \cref{fig:clasp_move}. (In the case that the two strands belong to different Legendrians, the blue pushoff of the upper strand is not present.) Let $(\Lambda_-; D_-)$ be the standard Lagrangian disk filling in $R_-$ of the indicated Legendrian in the same figure. Let $B_{\mathrm{c}}$ be the bypass with data $(\Lambda_{\pm}; D_{\pm})$.

\begin{figure}[ht]
	\centering
	\begin{overpic}[scale=.38]{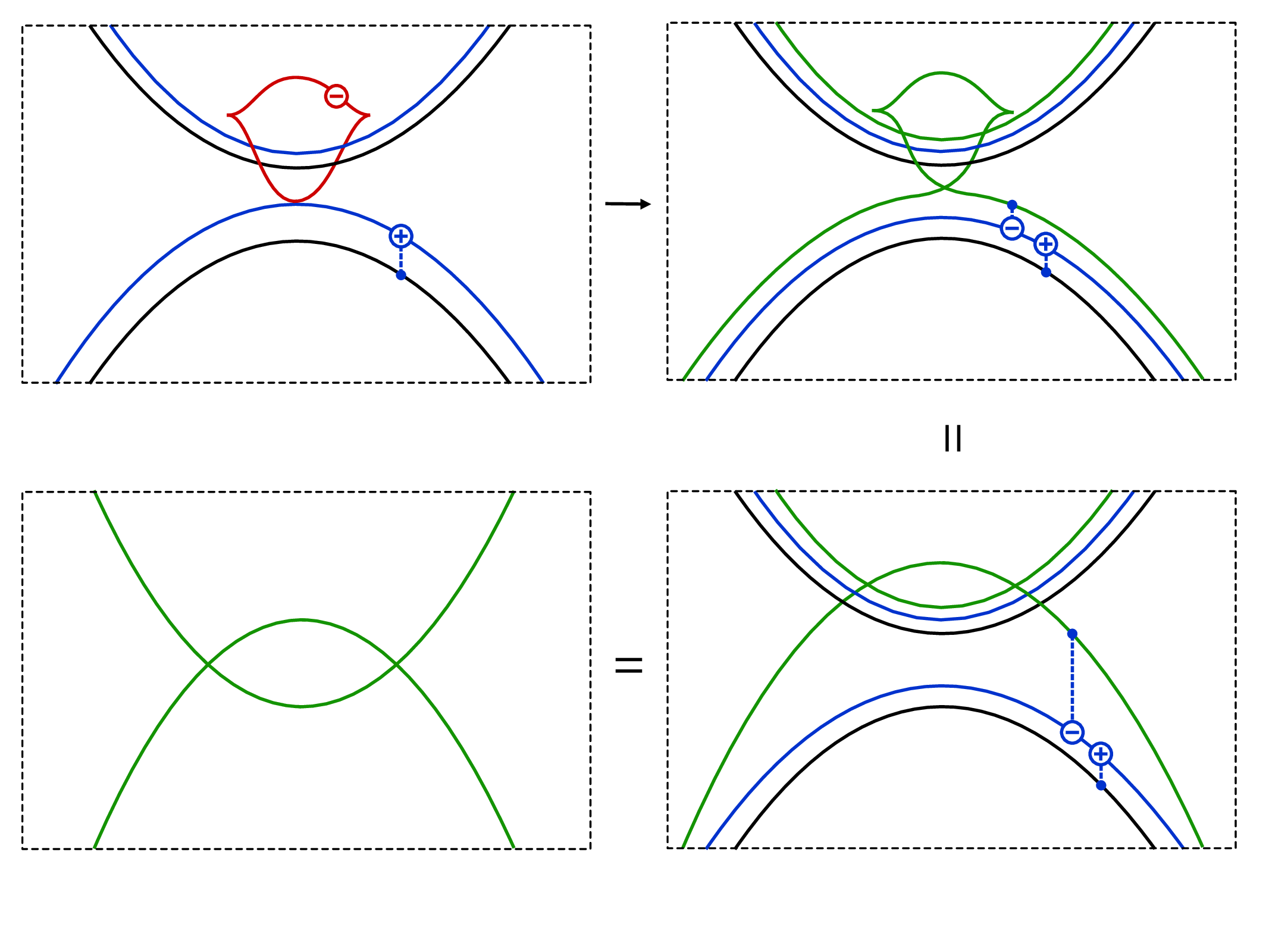}
	\put(16.5,68){\color{darkred}\small $\Lambda_-$}
	\put(58.5,59.5){\small \textcolor{darkgreen}{$\Lambda_-\uplus \Lambda_+$}}
	\put(9.5,55.5){\color{darkblue}\small $\Lambda_+$}
    \put(57,54){\color{darkblue}\small $\Lambda_+^{-\ve}$}

	\put(93,55){\tiny $(-1)$}
	\put(93,61.5){\tiny \textcolor{darkgreen}{$(-1)$}}
	\put(93,59){\tiny \textcolor{darkblue}{$(+1)$}}
	
	\put(93,18){\tiny $(-1)$}
	\put(93,24.5){\tiny \textcolor{darkgreen}{$(-1)$}}
	\put(93,22){\tiny \textcolor{darkblue}{$(+1)$}}
	
	\put(42,24.5){\tiny \textcolor{darkgreen}{$(-1)$}}
	\put(42,55){\tiny $(-1)$}
	
	\put(17, 74){\small (Bypass data)}
	\put(67, 74){\small (Attach bypass)}
	
	\end{overpic}
	\vskip-.3in
	\caption{The clasping bypass $B_{\mathrm{c}}$ in the proof of \cref{prop:clasping_bypass}. In the top left, the blue Lagrangian disk is a isotopic to the co-core of the positive Weinstein $n$-handle in $R_+$ attached along the black Legendrian. When $2n\geq 6$, all diagrams are $S^{n-2}$-spun around the central vertical axis.}
	\label{fig:clasp_move}
\end{figure}
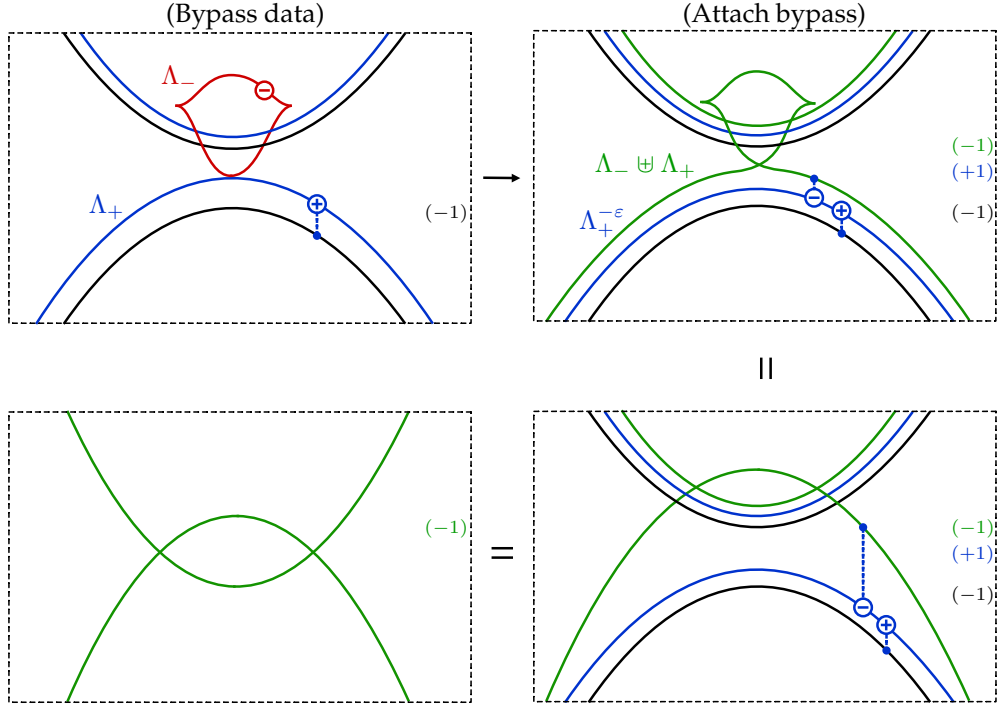

Attaching the $R_+$-centric bypass produces the decorated Legendrian diagram in the top right, which, after Legendrian Reidemeister moves, is equivalent to the decorated diagram in the bottom right. (In the case that the two strands belong to different Legendrians, the blue and green strands parallel to the upper strand are not present in the second and third frames.) Finally, by \cref{lemma:trivial_bypass_lemma} the pair of contact $n$- and ($n+1$)-handles attached along the black and blue Legendrians, respectively, forms a trivial bypass. The handles can be canceled, leaving the desired surgery diagram on the bottom left. (In the different strands case, the resulting surgery diagram is identical but the concave-up strand is the original black Legendrian, rather than a new green Legendrian.) Up to corner rounding, we have produced the contact handlebody $H_1 = (W_1 \times [-1,1]_t,\, \ker(dt + \lambda_1))$ as desired. This proves the existence of the clasping bypass. 

The unclasping bypass is similar. Beginning now with the handle presentation of $W_1$ on the right side of \cref{fig:clasp}, we again take as positive data $(\Lambda_+; D_+)$ a positive Reeb shift of the belt-sphere and co-core of the positive Weinstein $n$-handle attached to $R_+$ along the black Legendrian, and as negative data a standard Lagrangian disk filling $(\Lambda_-; D_-)$ in $R_-$ of the indicated red Legendrian in \cref{fig:unclasp_move}. After attaching the $R_+$-centric bypass and isotoping, the pair of contact $n$- and ($n+1$)-handles attached along the black and blue Legendrians, respectively, forms a trivial bypass. This leaves us with the desired contact handlebody $H_0 = (W_0 \times [-1,1]_t,\, \ker(dt + \lambda_0))$, and proves the existence of the unclasping bypass. 
\end{proof}

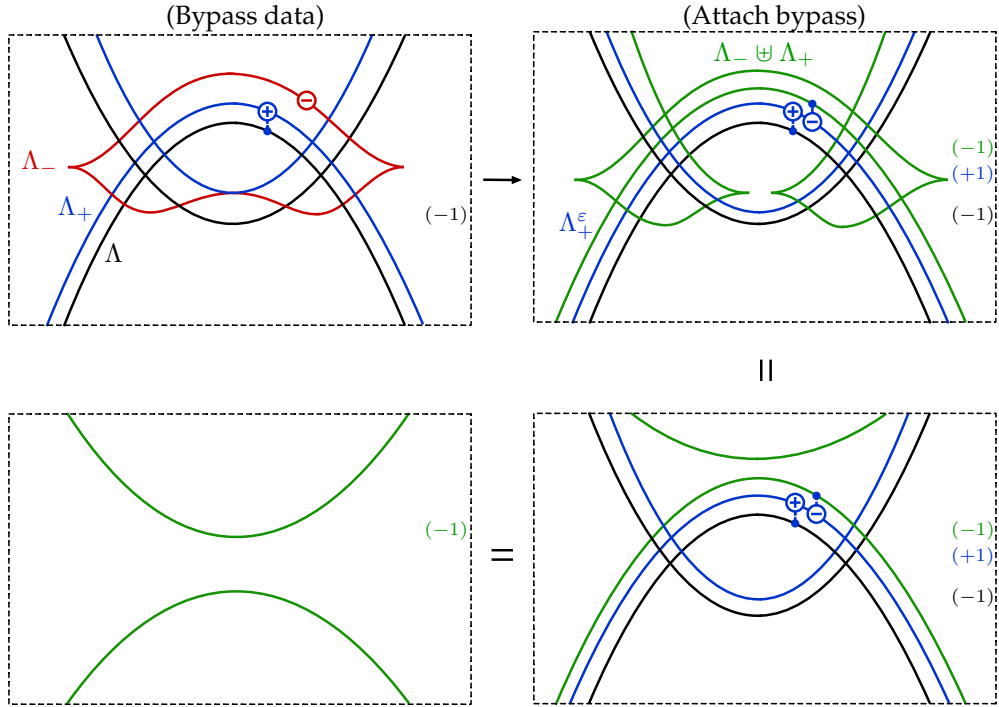
\begin{figure}[ht]
	\centering
	\begin{overpic}[scale=.38]{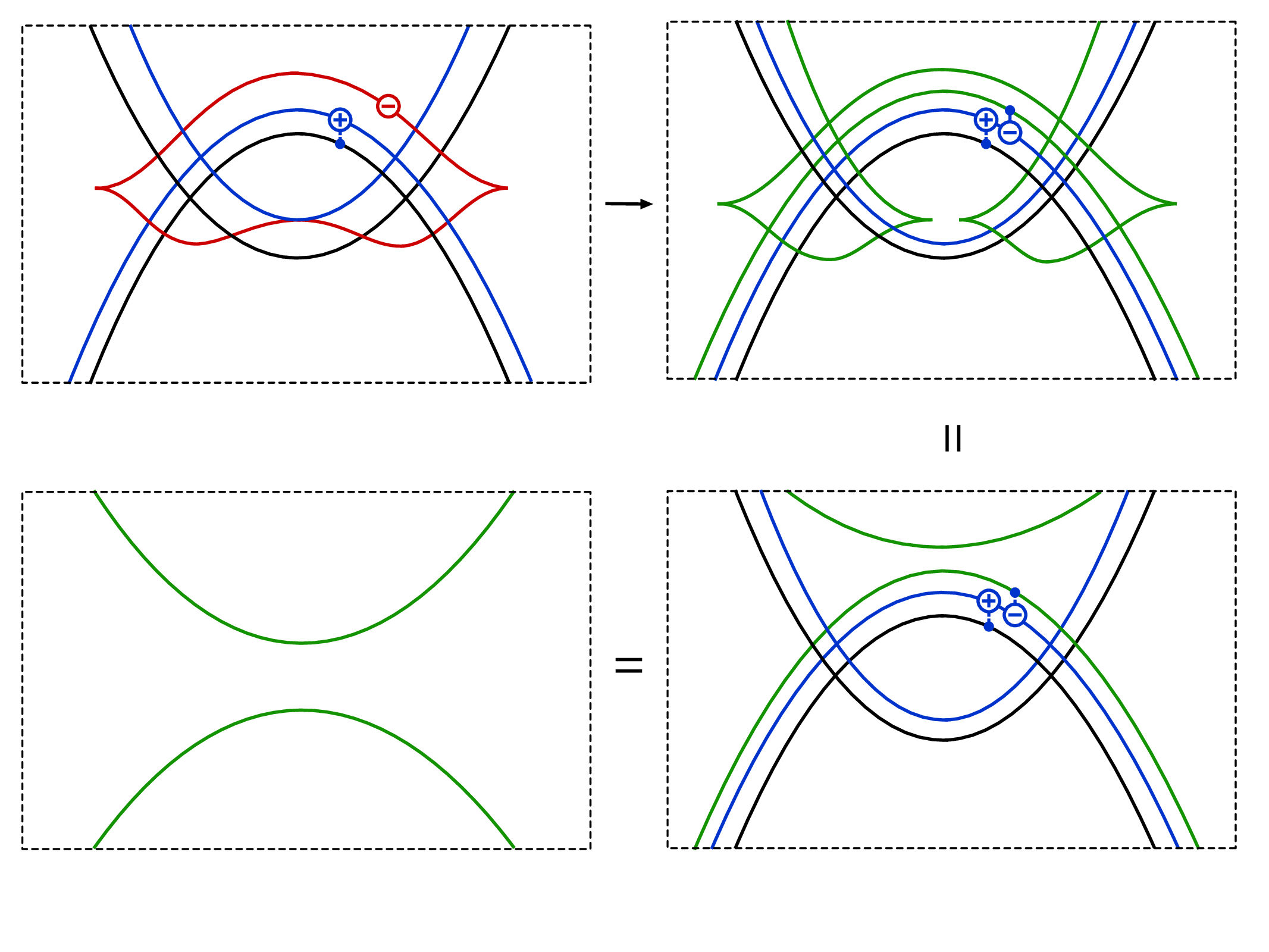}
	\put(11,51){\small $\Lambda$}
	\put(3,60){\color{darkred}\small $\Lambda_-$}
	\put(70,70.5){\small \textcolor{darkgreen}{$\Lambda_-\uplus \Lambda_+$}}
	\put(6.5,55.5){\color{darkblue}\small $\Lambda_+$}
    \put(55,54){\color{darkblue}\small $\Lambda_+^{\ve}$}

	\put(93,55){\tiny $(-1)$}
	\put(93,61.5){\tiny \textcolor{darkgreen}{$(-1)$}}
	\put(93,59){\tiny \textcolor{darkblue}{$(+1)$}}
	
	\put(93,18){\tiny $(-1)$}
	\put(93,24.5){\tiny \textcolor{darkgreen}{$(-1)$}}
	\put(93,22){\tiny \textcolor{darkblue}{$(+1)$}}
	
	\put(42,24.5){\tiny \textcolor{darkgreen}{$(-1)$}}
	\put(42,55){\tiny $(-1)$}
	
	\put(17, 74){\small (Bypass data)}
	\put(67, 74){\small (Attach bypass)}
	
	\end{overpic}
	\vskip-.3in
	\caption{The unclasping bypass $B_{\mathrm{u}}$ in the proof of \cref{prop:clasping_bypass}. When $2n\geq 6$, all diagrams are $S^{n-2}$-spun around the central vertical axis.}
	\label{fig:unclasp_move}
\end{figure}

The second bypass move describes how to stabilize, as well as destabilize, a Legendrian attaching sphere in the sense of \cite[Section 4.3]{ekholm2005non}; see also \cite{eliashberg1990topological}. Note that in case $2n=4$, this stabilization corresponds to stabilizing a Legendrian knot both positively and negatively.

\begin{proposition}[The stabilizing bypass]\label{prop:stabilizing_bypass}
Let $(W_0, \lambda_0, \phi_0)$ be a Weinstein domain of dimension $2n\geq 4$. Suppose that in a surgery diagram for $W_0$ there is a chart with front projection given by the left side of \cref{fig:stab_move}. Let $(W_1,\lambda_1, \phi_1)$ be the Weinstein domain obtained by modifying the surgery diagram as in the right side of \cref{fig:stab_move}. Finally, let $H_i := W_i\times [-1,1]$, $i=0,1$, be the contact handlebody based over $W_i$. Then there are bypass attachments $B_{\mathrm{s}}$ and $B_{\mathrm{d}}$ such that 
\begin{align*}
    H_0 \cup B_{\mathrm{s}} &\cong H_1, \\
    H_1 \cup B_{\mathrm{d}} &\cong H_0.
\end{align*}
Here, $\cong$ indicates contactomorphism relative to the boundary. Moreover, when $2n\geq 6$, the Weinstein hypersurfaces $W_i \times \{0\}$, $i=0,1$, are smoothly isotopic under the above contactomorphisms.
\end{proposition}

\begin{figure}[ht]
	\centering
	\begin{overpic}[scale=.36]{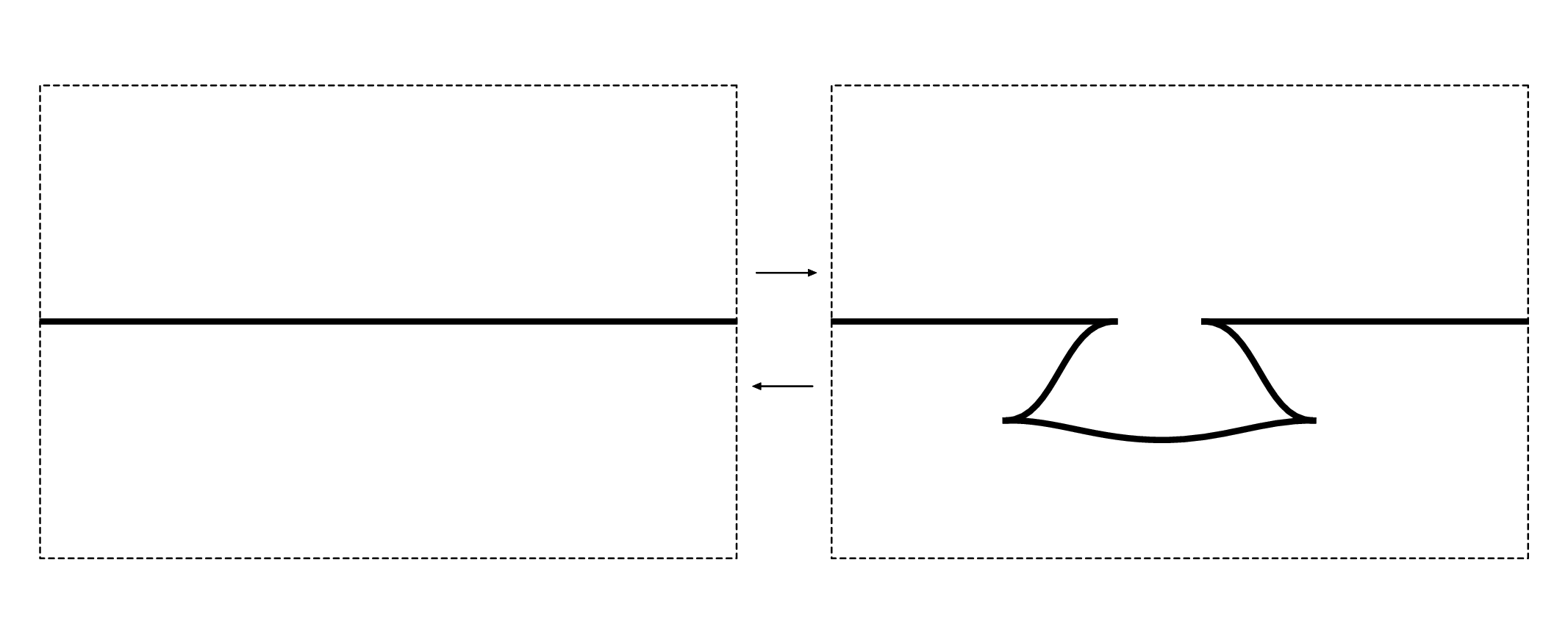}
 
	\put(48.75,24.75){\small $B_{\mathrm{s}}$}
    \put(48.75,14){\small $B_{\mathrm{d}}$}

    \put(41,23){\scriptsize $(-1)$}
    \put(91.5,23){\scriptsize $(-1)$}

    \put(23.75,2.5){$W_0$}
    \put(74.75,2.5){$W_1$}
 
	\end{overpic}
	\caption{The stabilizing and destabilizing bypasses in \cref{prop:stabilizing_bypass}. When $2n\geq 6$, both figures are $S^{n-2}$-spun around the central vertical axis. In particular, the Legendrian on the right is loose.}
	\label{fig:stab_move}
\end{figure}
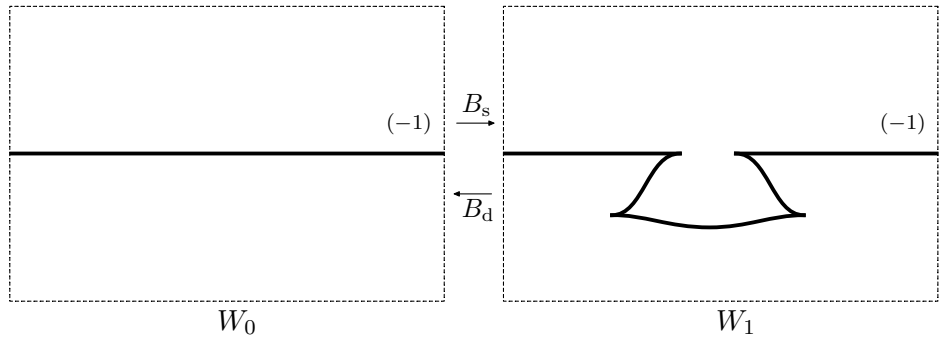

\begin{proof}
We exhibit the stabilizing bypass $B_{\mathrm{s}}$ directly, in the same way as in the proof of \cref{prop:clasping_bypass}. Let $(\Lambda_+; D_+)$ be given by a small forward-time Reeb isotopy of the belt-sphere and co-core of the positive Weinstein $n$-handle of $R_+$ in question, indicated by the blue Legendrian on the left side of \cref{fig:stab}, and let $(\Lambda_-; D_-)$ be the standard Lagrangian disk filling in $R_-$ of the red Legendrian in the same diagram. Let $B_{\mathrm{s}}$ be the bypass with data $(\Lambda_{\pm}; D_{\pm})$. Attaching the $R_+$-centric bypass yields the decorated surgery diagram in the middle diagram. As in the proof of \cref{prop:clasping_bypass}, the black and blue contact handles pair as a trivial bypass, and thus the handles can be erased, leaving the desired stabilized Legendrian. 

\begin{figure}[ht]
	\centering
	\begin{overpic}[scale=.5]{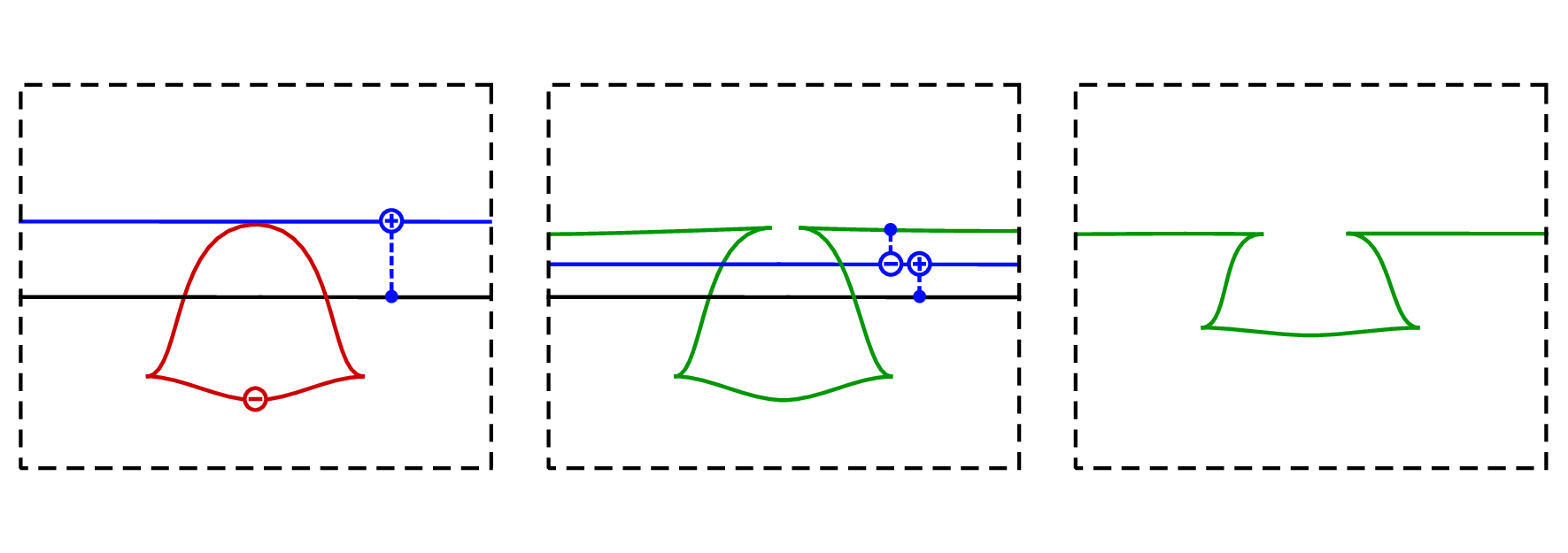}

    \put(26,14){\tiny $(-1)$}
    \put(60,14){\tiny $(-1)$}
    \put(60,22){\tiny \textcolor{darkgreen}{$(-1)$}}
    \put(60,24){\tiny \textcolor{darkblue}{$(+1)$}}
    \put(93.5,22){\tiny \textcolor{darkgreen}{$(-1)$}}
    
	\end{overpic}
	\caption{Identifying the stabilizing bypass directly. When $2n\geq 6$, all diagrams are $S^{n-2}$-spun around the central vertical axis.}
	\label{fig:stab}
\end{figure}
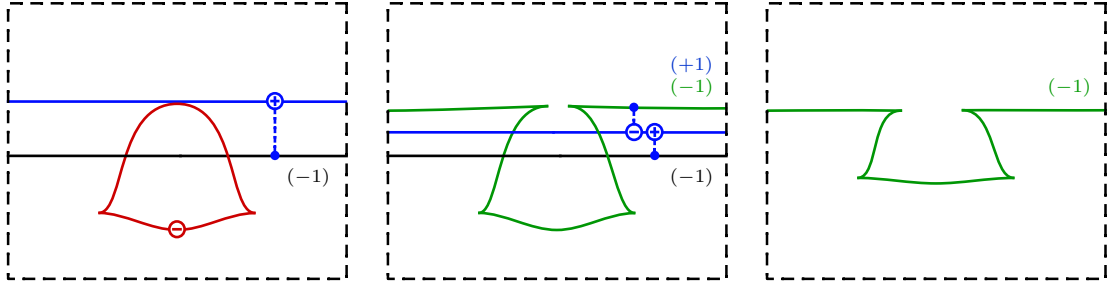

To perform the destabilization in a spin-symmetric manner, we will use a combination of Weinstein homotopy and the unclasping bypass already introduced in \cref{prop:clasping_bypass}. The sequence of steps is given in \cref{fig:destab}. Between panel $1$ and panel $2$ we birth a pair of canceling $n$- and ($n-1$)-handles in the underlying Weinstein domain; the $n$-handle is attached along the yellow Legendrian. In panel $3$ we slide the black Legendrian up across the yellow Legendrian and then simplify with Reidemeister moves to produce panel $4$. In panel $5$ we witness the gray ($n-1$)-handle as a ($+1$)-surgery along the gray unknot, which corresponds to digging out its standard Lagrangian disk filling. Panel $6$ follows from more Reidemeister moves, and panel $7$ is obtained by attaching an unclasping bypass from \cref{prop:clasping_bypass}. Panel $8$ performs a second handleslide of the black Legendrian over the yellow Legendrian, and panel $9$ follows from another sequence of Reidemeister moves. Finally, in panels $10$ and $11$ we re-cancel the yellow and gray pair, leaving the destabilized Legendrian as desired. Note that all Reidemeister moves, slides, and births and deaths of pairs are performed in a spin-symmetric manner and thus are valid in all dimensions. 

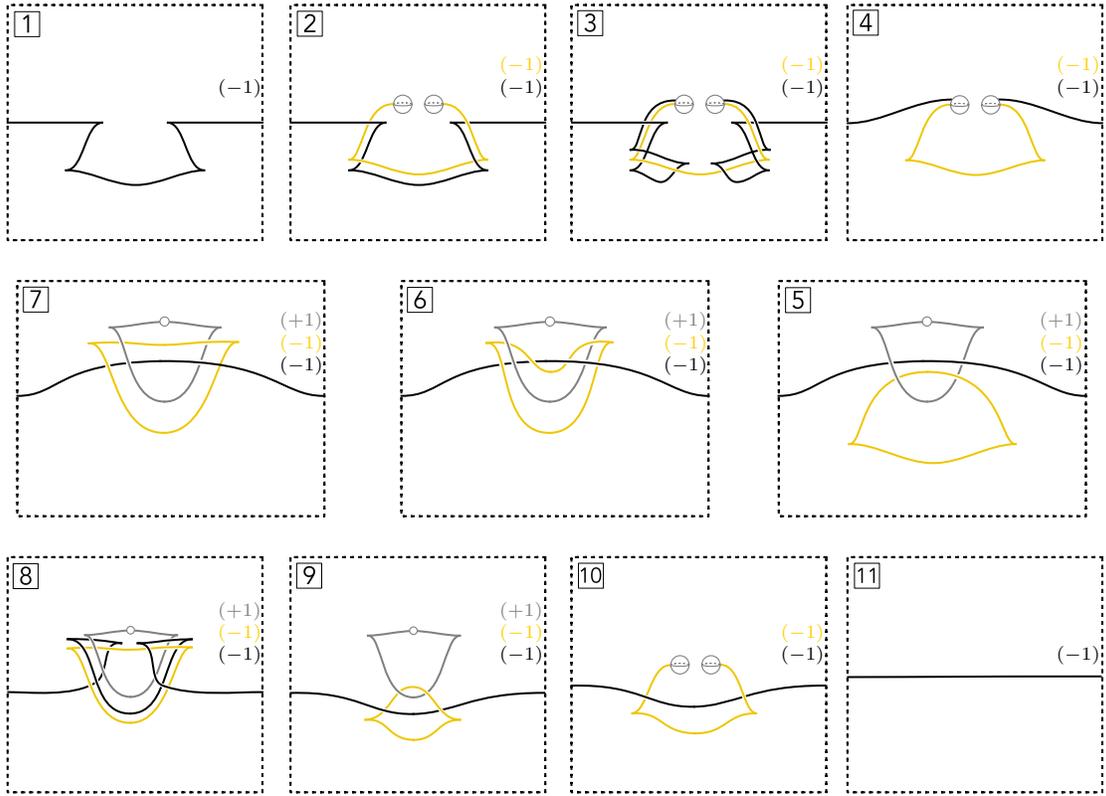
\begin{figure}[ht]
	\centering
	\begin{overpic}[scale=.41]{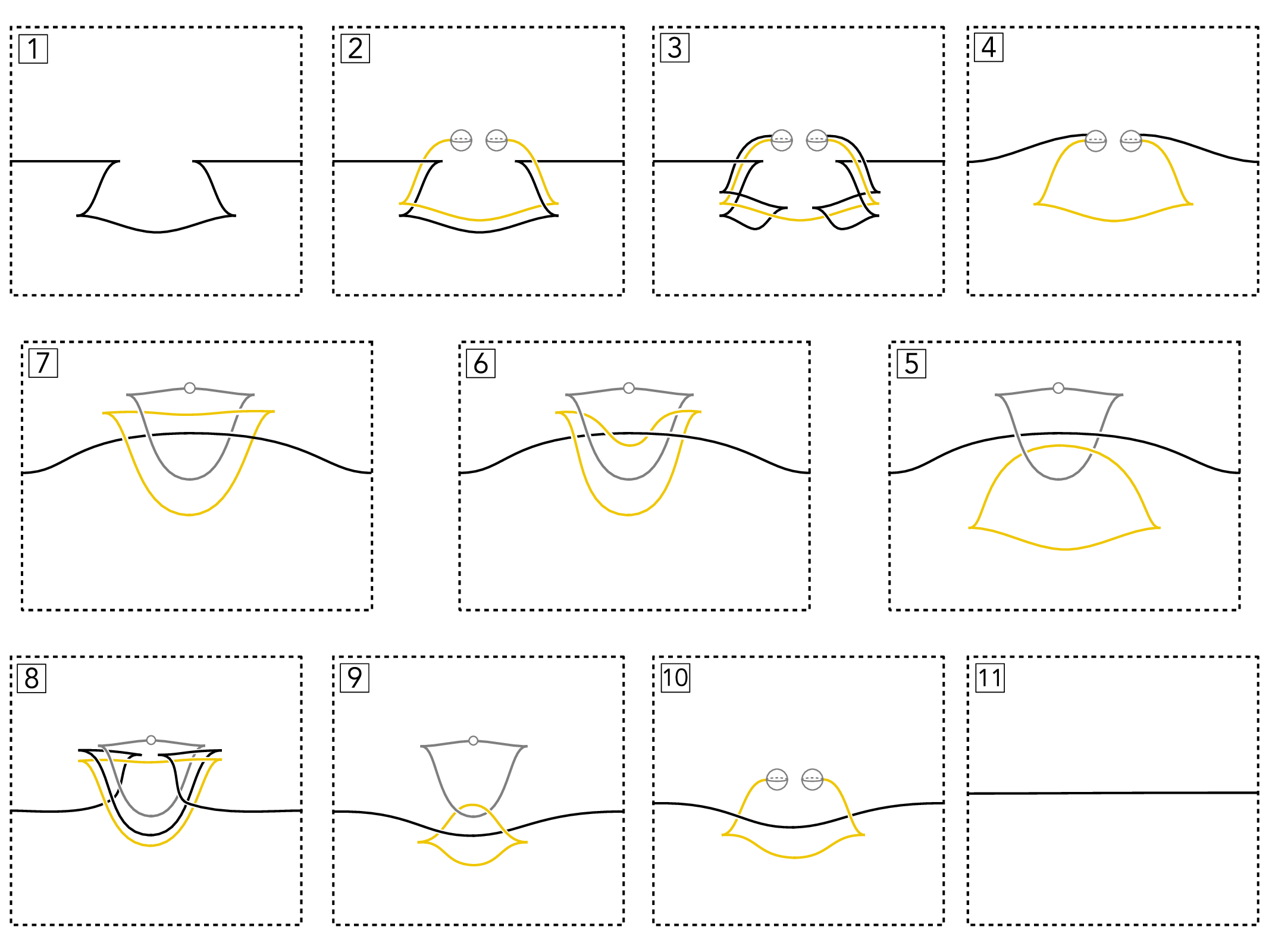}


    \put(19.75,65){\tiny $(-1)$}
    
    \put(45,65){\tiny $(-1)$}
    \put(45,67){\tiny \textcolor{gold}{$(-1)$}}

    \put(70.25,65){\tiny $(-1)$}
    \put(70.25,67){\tiny \textcolor{gold}{$(-1)$}}

    \put(95,65){\tiny $(-1)$}
    \put(95,67){\tiny \textcolor{gold}{$(-1)$}}


    \put(25.25,40){\tiny $(-1)$}
    \put(25.25,42){\tiny \textcolor{gold}{$(-1)$}}
    \put(25.25,44){\tiny \textcolor{gray}{$(+1)$}}

    \put(59.75,40){\tiny $(-1)$}
    \put(59.75,42){\tiny \textcolor{gold}{$(-1)$}}
    \put(59.75,44){\tiny \textcolor{gray}{$(+1)$}}

    \put(93.5,40){\tiny $(-1)$}
    \put(93.5,42){\tiny \textcolor{gold}{$(-1)$}}
    \put(93.5,44){\tiny \textcolor{gray}{$(+1)$}}


    \put(19.75,14){\tiny $(-1)$}
    \put(19.75,16){\tiny \textcolor{gold}{$(-1)$}}
    \put(19.75,18){\tiny \textcolor{gray}{$(+1)$}}

    \put(45,14){\tiny $(-1)$}
    \put(45,16){\tiny \textcolor{gold}{$(-1)$}}
    \put(45,18){\tiny \textcolor{gray}{$(+1)$}}

    \put(70.25,14){\tiny $(-1)$}
    \put(70.25,16){\tiny \textcolor{gold}{$(-1)$}}

    \put(95,14){\tiny $(-1)$}
    
	\end{overpic}
	\caption{Performing the destabilizing bypass move via Weinstein homotopies and unclasping. When $2n\geq 6$, all diagrams are $S^{n-2}$-spun around the central vertical axis.}
	\label{fig:destab}
\end{figure}

It remains to show that $W_0$ and $W_1$ are smoothly isotopic when $2n\geq 6$. We consider the case of attaching a stabilizing bypass $B_{\mathrm{s}}$ to the contact handlebody over $W_0$. Let $\Lambda_0$ denote the attaching sphere in the far left panel of \cref{fig:stab} and let $W_0^{\flat}\subset W_0$ be the Weinstein subdomain obtained by removing the associated handle. Let $\Lambda_1$ denote the attaching sphere in the far right panel, i.e., the stabilization of $\Lambda_0$. Observe that the embedded hypersurface $W_1\times \{0\}$ is obtained by first removing the (ambient) handle attached to $W_0^{\flat}\times \{0\}$ along $\Lambda_0\times \{0\}$ and subsequently attaching the (ambient) handle along $\Lambda_1\times \{0\}$. This process leaves the embedding of $W_0^{\flat}\times\{0\}$ invariant, so it suffices to prove that the ambient handles attached along $\Lambda_i\times \{0\},i=0,1$, are smoothly isotopic. 

\vspace{2mm}
\textit{Claim 1: The cores are smoothly isotopic rel $\partial W_0^{\flat}\times \{0\}$.}
\vspace{2mm}

By construction of a bypass attachment, the core of the $n$-handle attached along $\Lambda_1$ bounds, together with $D_-\uplus_b D_+$, an ($n+1$)-dimensional disk. Here, $\uplus_b$ is a Legendrian operation that induces a smooth boundary connect sum; see the proof of \cref{thm:moves_for_obds} for further discussion. This disk is what is known as the \textit{bypass half-disk} in dimension $3$, and more generally is an example of a \textit{co-Legendrian} as introduced by Huang \cite{huang2020colegendrians}. In particular, the core of the (ambient) handle attached along $\Lambda_1 \times \{0\}$ is smoothly isotopic to $D_-\uplus_b D_+$; see \cref{fig:smoothiso} for a schematic picture. 

\begin{figure}[ht]
	\centering
	\begin{overpic}[scale=.4]{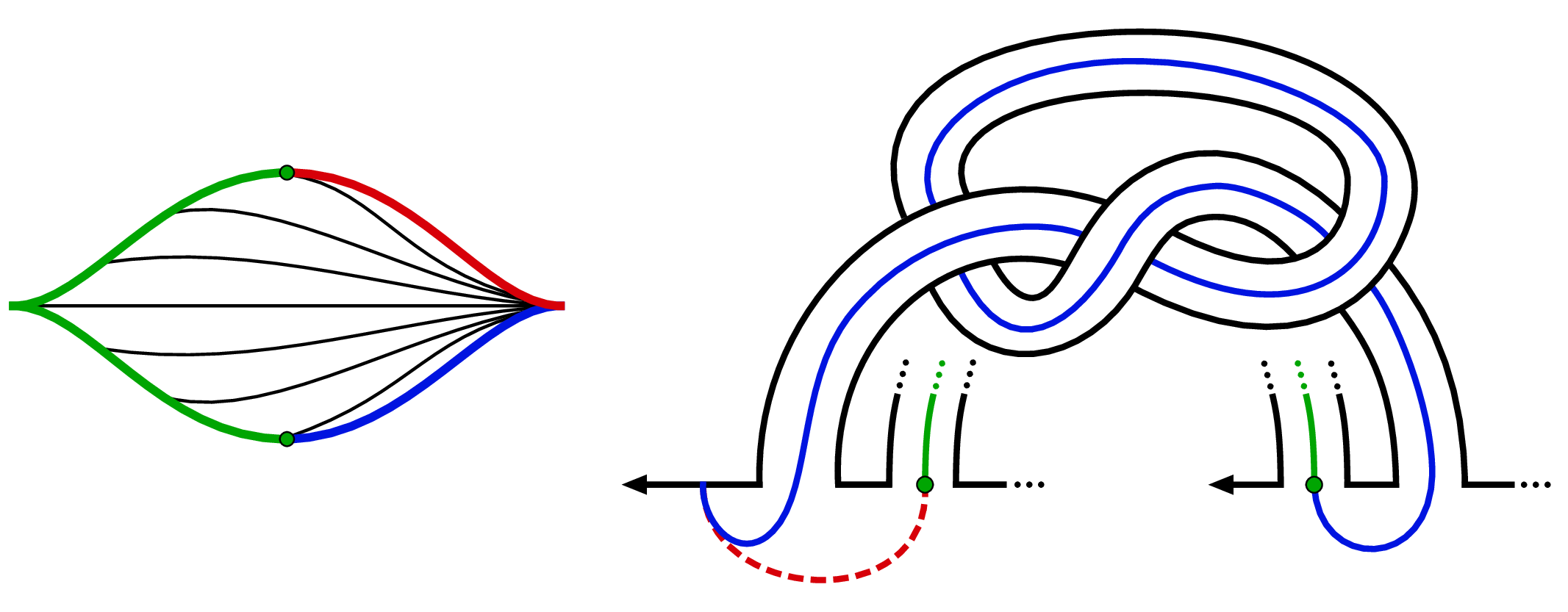}

    \put(28,12){\small \textcolor{darkblue}{$D_+$}}
    \put(28,26){\small \textcolor{darkred}{$D_-$}}

	\end{overpic}
	\caption{A bypass half-disk on the left. The core of the bypass $n$-handle attached along $\Lambda_1$ is in green, and the half-disk is foliated by Legendrian disks from $D_+$ to $D_-$. On the right is the corresponding schematic picture for the ambient handles.}
	\label{fig:smoothiso}
\end{figure}

In our case, $D_-$ is a smoothly unknotted disk, so $D_-\uplus_b D_+$ is smoothly isotopic to $D_+$. On the other hand, $D_+$ fills a parallel copy of $\Lambda_0$ and has geometric intersection number $1$ with the co-core of the handle attached along $\Lambda_0$, so it is smoothly isotopic to the core of the (ambient) handle attached along $\Lambda_0 \times \{0\}$. The bypass-half disk then builds the desired smooth ambient isotopy rel $\partial W_0^{\flat}\times \{0\}$ between the cores of the (ambient) Weinstein $n$-handles. 

\vspace{2mm}
\textit{Claim 2: The framed attaching spheres are framed isotopic in $\partial W_0^{\flat}\times \{0\}$.}
\vspace{2mm}

That $\Lambda_0\times \{0\}$ and $\Lambda_1\times \{0\}$ are smoothly isotopic in $\partial W_0^{\flat}\times \{0\}$ is immediate from \cref{fig:stab}. That they are isotopic as framed spheres follows from the fact that loosifying a Legendrian submanifold preserves its formal data \cite{murphy2012loosev2}. In particular, the framing of the Weinstein handle attachment (which is specified by the formal Legendrian data) is preserved.

\vspace{2mm}

Claim 1 and Claim 2 together imply that the ambient handles attached to $W_0^{\flat}\times\{0\}$ along $\Lambda_i\times \{0\},i=0,1$, are smoothly isotopic. This completes the proof.
\end{proof}

\begin{remark}
We include a number of remarks regarding the (de)stabilizing bypass. 
\begin{enumerate}
    \item A second spin-symmetric and valid proof of the stabilization move is to run the steps in \cref{fig:destab} in reverse, leveraging a clasping bypass. Likewise, in principle it should be possible to identify the destabilizing bypass data directly. 

    \item When $2n=4$, so that the Legendrians in question are knots, it is straightforward to witness double (de)stabilization via (un)clasping and a simpler sequence of Reidemeister moves that are not spin-symmetric. 

    \item The (de)stabilizing bypass attachments are similar to the wrinkling bypass attachments described in \cite[\S 5]{BHH23}.

    \item The assumption $2n\geq 6$ in the "moreover" statement is necessary. When $2n=4$, the Weinstein hypersurfaces before and after a stabilization move will, in general, not even be diffeomorphic; indeed, double stabilization changes $\mathrm{tb}$ and consequently the framing of the $2$-handle.
\end{enumerate}
\end{remark}

\subsection{Bypass moves in dimension $5$}\label{subsec:moves_dim5}

The (un)clasping and double (de)stabilization moves are valid in dimension $5$ and above. Here we record one additional move in dimension $5$ that has no meaning in higher dimensions, which is a crossing change. It is a consequence of double (de)stabilization and (un)clasping. 

\begin{proposition}[The crossing change bypass move]\label{prop:crossing_change}
Let $(W^{4}_0, \lambda_0, \phi_0)$ be a Weinstein domain. Suppose that in a surgery diagram for $W_0$ there is a chart with front projection given by the left side of \cref{fig:crossing}. Let $(W_1,\lambda_1, \phi_1)$ be the Weinstein domain obtained by modifying the surgery diagram as in the right side of \cref{fig:crossing}. Finally, let $H_i := W_i\times [-1,1]$, $i=0,1$, be the contact handlebody based over $W_i$. Then there are bypass sequences $B_{\mathrm{x}} := B_{\mathrm{u}}\circ B_{\mathrm{s}}$ and $B_{\mathrm{y}}:= B_{\mathrm{d}}\circ B_{\mathrm{c}}$ such that 
\begin{align*}
    H_0 \cup B_{\mathrm{x}} &\cong H_1, \\
    H_1 \cup B_{\mathrm{y}} &\cong H_0.
\end{align*}
Here, $\cong$ indicates contactomorphism relative to the boundary. 
\end{proposition}

\begin{figure}[ht]
	\centering
	\begin{overpic}[scale=.36]{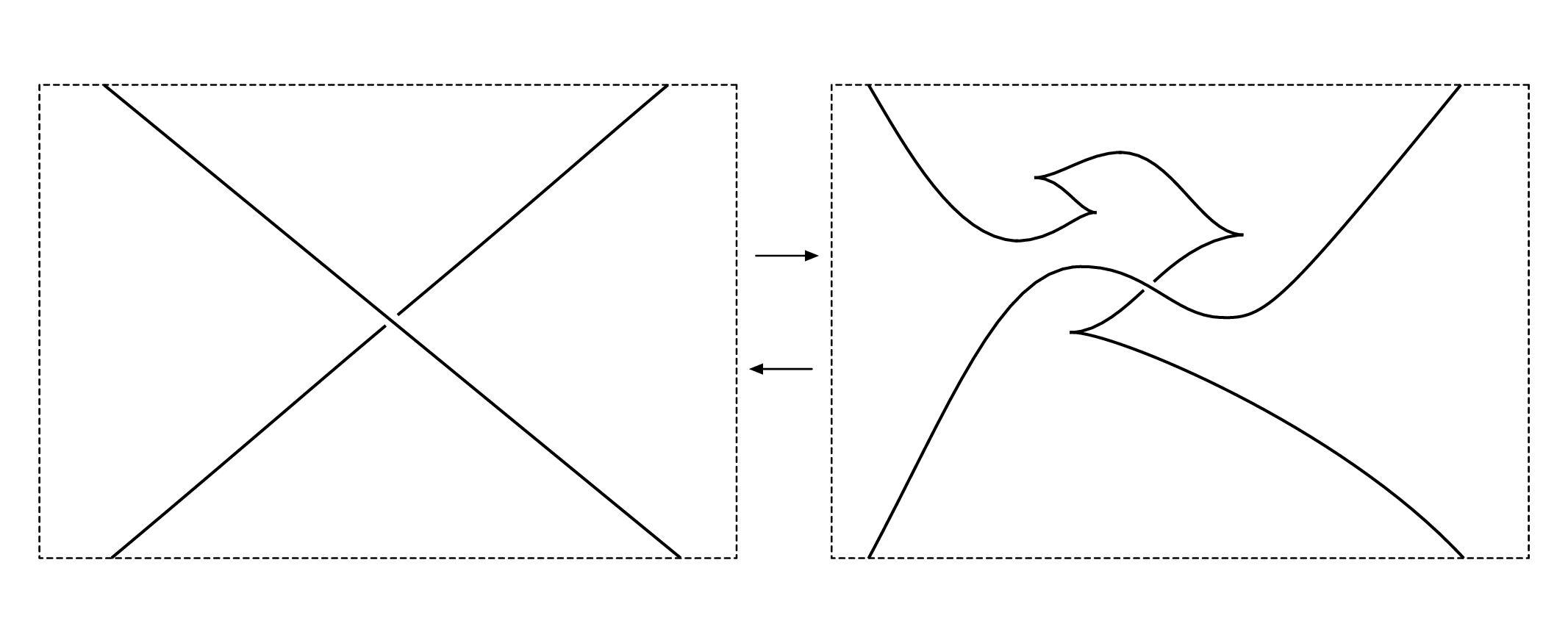}
 
	\put(48.75,26){\small $B_{\mathrm{x}}$}
    \put(48.75,15){\small $B_{\mathrm{y}}$}

    \put(41.5,30){\scriptsize $(-1)$}
    \put(92,28){\scriptsize $(-1)$}

    \put(23.75,2.5){$W_0$}
    \put(74.75,2.5){$W_1$}
 
	\end{overpic}
	\caption{The crossing change move in \cref{prop:crossing_change}.}
	\label{fig:crossing}
\end{figure}

\begin{proof}
The crossing change from left to right in \cref{fig:crossing} is achieved with two bypass attachments. The sequence is given by \cref{fig:crossing_move}. Between the first two panels we attach a double stabilization bypass, and between the fourth and fifth panels we attach an unclasping bypass. The rest of the panels are related by Reidemeister moves. 
\end{proof}

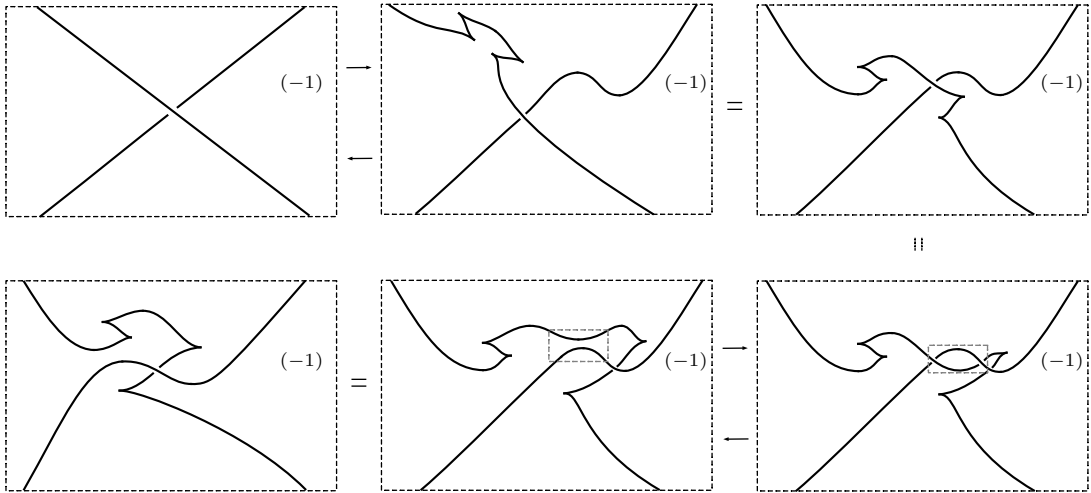
\begin{figure}[ht]
	\centering
	\begin{overpic}[scale=.41]{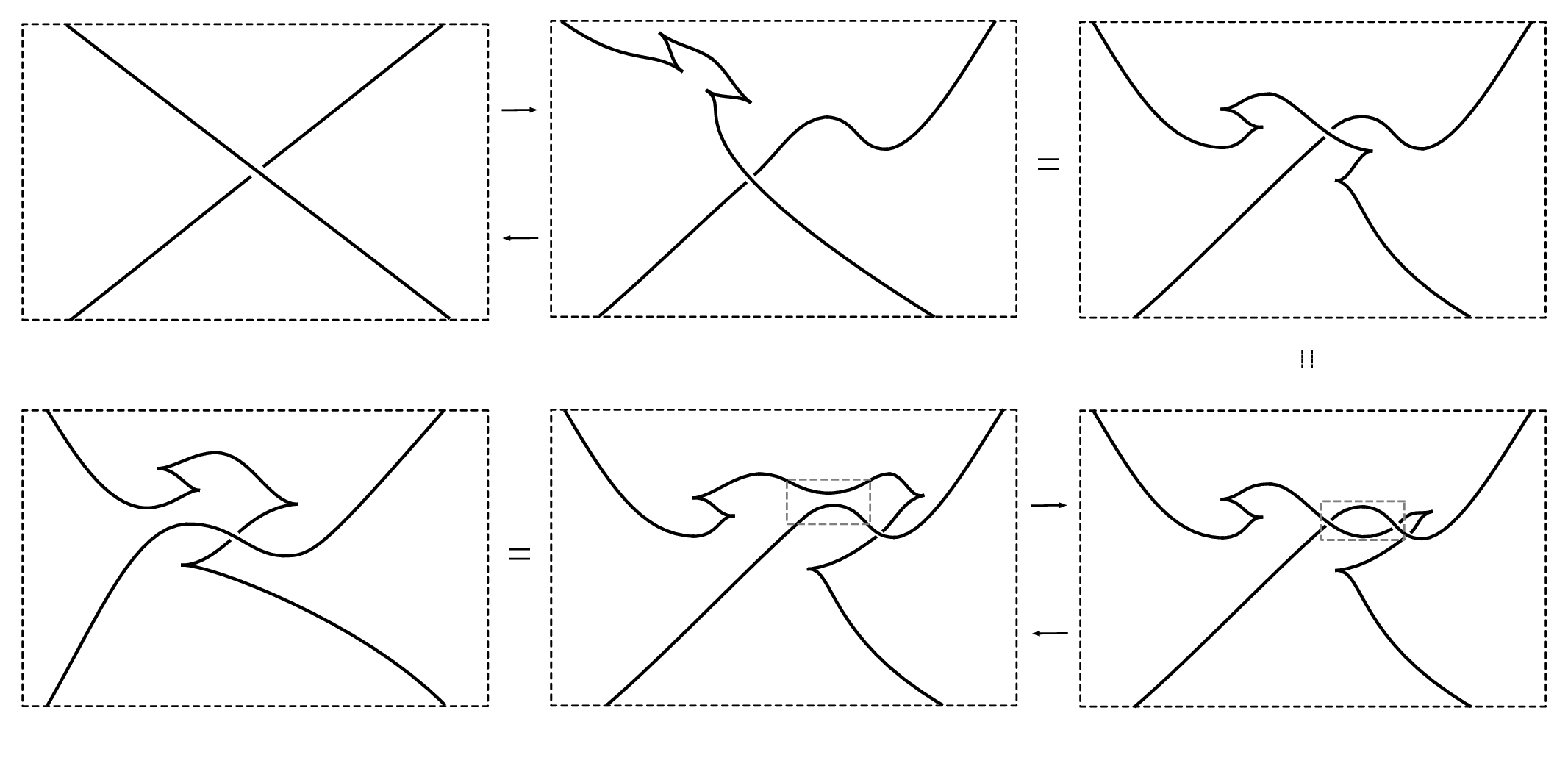}
 
	\put(94.25,40){\tiny $(-1)$}
	\put(94.25,15){\tiny $(-1)$}

    \put(60.5,40){\tiny $(-1)$}
	\put(60.5,15){\tiny $(-1)$}

    \put(26,40){\tiny $(-1)$}
	\put(26,15){\tiny $(-1)$}
	
	\end{overpic}
	\vskip-.2in
	\caption{Changing a crossing via double (de)stabilization and (un)clasping.}
	\label{fig:crossing_move}
\end{figure}

\begin{remark}
In dimension $5$, the stabilizing bypass move from \cref{prop:stabilizing_bypass} is a relative version of Move I from \cite{Ding2012Diagrams}, and moreover recovers the full move for closed manifolds. Likewise, the crossing change bypass move recovers Move II from \cite{Ding2012Diagrams}, though we point out that our move requires only two stabilizations, rather than four. 
\end{remark}
\section{Appearance of bypass moves in dimension 3}\label{sec:appearance_dim_3}

Here we prove \cref{thm:main_dim3}, which describes the clasp move in dimension $3$ along with some occurrences in nature. We break the proof of \cref{thm:main_dim3} across multiple subsections for ease of reference. In \cref{subsec:3dclasp} we identify the clasp move, in \cref{subsec:slopes} we use the clasp move to change the dividing slope of a standard torus, in \cref{subsec:destab_arc} to destabilize a Legendrian arc, and in \cref{subsec:legendrian_graphs} to change the orientation structure of a Legendrian graph. Finally, in \cref{subsec:convex_heegaard} we prove \cref{cor:heegaard} on convex Heegaard splittings.

\subsection{The clasp move in dimension $3$}\label{subsec:3dclasp}

First, we define the bypass data for the clasp move and show that it modifies the contact handlebody as in \cref{fig:3dclasp}. 

\begin{proof}[Proof of existence of bypass attachment in \cref{thm:main_dim3}.]
We mimic the bypass attachment data from \cref{prop:clasping_bypass}. In reference to the disk-band presentation on the left side of \cref{fig:3dclasp}, define a bypass arc on $\partial H_0$ as follows. Let $D_- \subset R_-$ be a small semicircle in the disk centered on the leftmost foot, and let $D_+\subset R_+$ be a small semicircle in the disk centered on the rightmost foot such that $D_+\cap D_-$ is a single point between the two feet. After sliding the positive endpoint of $D_+$ up over the band associated to the rightmost foot, we see that $D_+$ is a positive push-off of the co-core of the band. Let $B$ be the bypass with data $D_{\pm}$. To show that $B$ achieves the desired effect, there are two cases to analyze.

\begin{figure}[ht]
	\centering
	\begin{overpic}[scale=.38]{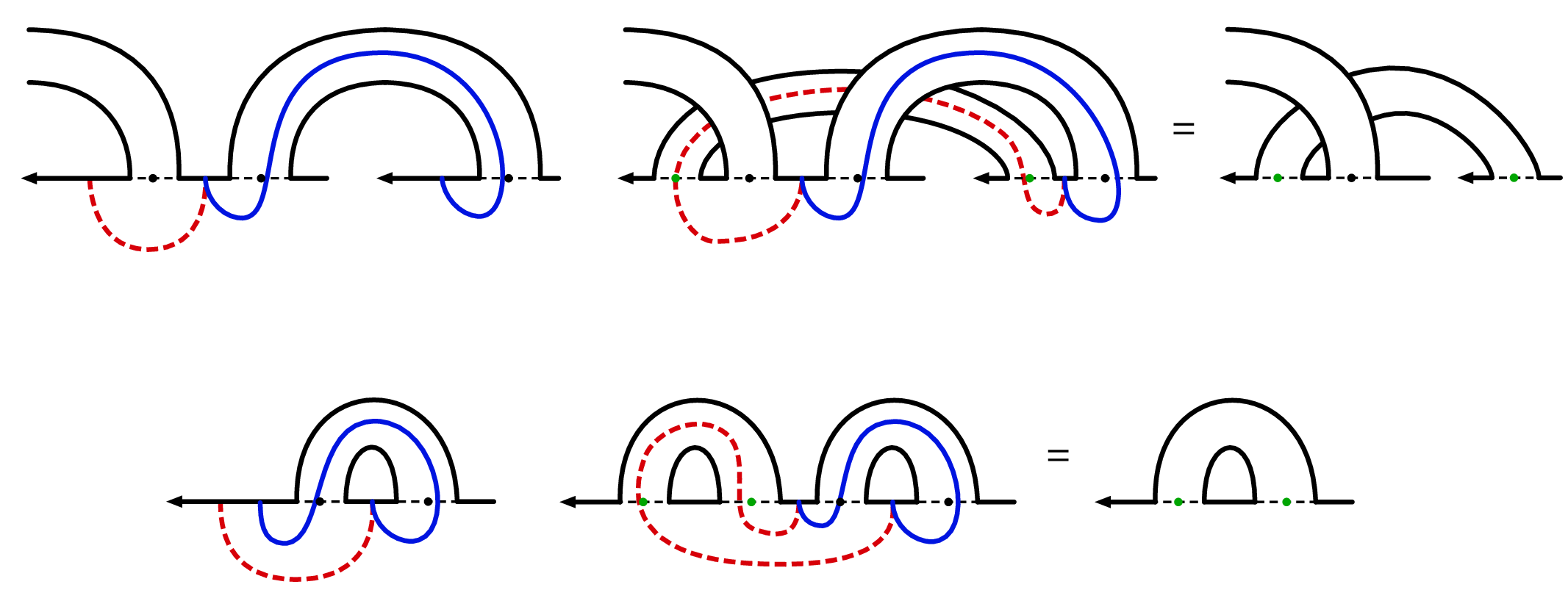}
 
	\end{overpic}
	\caption{The proof of $3$-dimensional clasp move. The top row is Case I, where the feet are associated to different bands, and the bottom row is Case II where the feet are part of the same band. Recall that we are working with abstract handlebodies and bypasses, so the embedding of the bands in the figure is meaningless.}
	\label{fig:3dclaspproof}
\end{figure}

\vspace{2mm}
\noindent \emph{Case I: The feet are associated to different bands.}
\vspace{2mm}

This is given by the top row of \cref{fig:3dclaspproof}. In the leftmost panel of the figure we have identified $D_+$ and $D_-$. The middle panel depicts the result of the bypass attachment, which has two elements: a contact $1$-handle attached along $\partial (D_+^{-\ve} \cup D_-^{\ve})$, which effectively attaches a band to the underlying surface and a contact $2$-handle attached along the circle which is the union of $D_+^{-\ve}$ and the result of sliding $D_-^{\ve}$ down across the contact $1$-handle. By \cref{lemma:trivial_bypass_lemma}, the contact $1$-handle associated to the original rightmost band then pairs with the contact $2$-handle as a trivial bypass. Canceling the handles then leaves a contact handlebody over $W_1$, as desired. 

\vspace{2mm}
\noindent \emph{Case II: The feet are associated to the same band.}
\vspace{2mm}

Observe that when the two feet on the left side of \cref{fig:3dclasp} are associated to the same band, switching their attaching regions does nothing to the abstract surface. Thus, it suffices to show that attaching $B$ in this case produces a contact handlebody over the same surface. The same argument as in Case I then follows through word for word; visual inspection of the bottom row of \cref{fig:3dclaspproof} reveals that the resulting surface is the same. 

We caution the reader that although the clasp move in this case does not change the underlying abstract surface type, $B$ is not a trivial bypass. For a clarifying example, see the discussion of slope modification in the following subsection.  
\end{proof}

\subsection{Modifying slopes of dividing curves}\label{subsec:slopes}

In many classifications (e.g., \cite{honda2000classification,honda2000classification2}) of tight contact structures on 3-manifolds, the fundamental building block is the \emph{basic slice}, which is one of two tight contact structures on the thickened torus $T^2\times I$.  The basic slice owes its status to the fact that it is the contact cobordism, as described in \cref{theorem:bypass_attachment}, resulting from particular bypass attachment data on a convex $T^2$ whose dividing set has two components.  Indeed, the basic slice corresponds to the unique nontrivial bypass attachment data on such a $T^2$ which produces a tight contact structure on $T^2\times I$.  Given a contact handlebody $S^1\times D^2$, the clasp move produces a contact cobordism $(T^2\times I,\xi)$, and in this subsection we verify that this cobordism is a basic slice.

\begin{proof}[Proof of \eqref{part:slope} of \cref{thm:main_dim3}.]

Let $W_0$ be an annulus and let $H_0 := W_0\times [-1,1]$ be the contact handlebody over $W_0$. Then $H_0$ is a solid torus and $\Sigma_0 := \partial H_0$ is a standard convex torus with two dividing curves. Viewing $W_0$ as a disk with a single band attached as on the left side of \cref{fig:slope}, we identify the bypass arc $D_+\cup D_- \subset \Sigma_0$ as indicated. Note that, by Case II of \cref{subsec:3dclasp}, this bypass performs a self-clasp of the band; compare the left side of \cref{fig:slope} and the lower left panel of \cref{fig:3dclaspproof}.

\begin{figure}[ht]
	\centering
	\begin{overpic}[scale=.41]{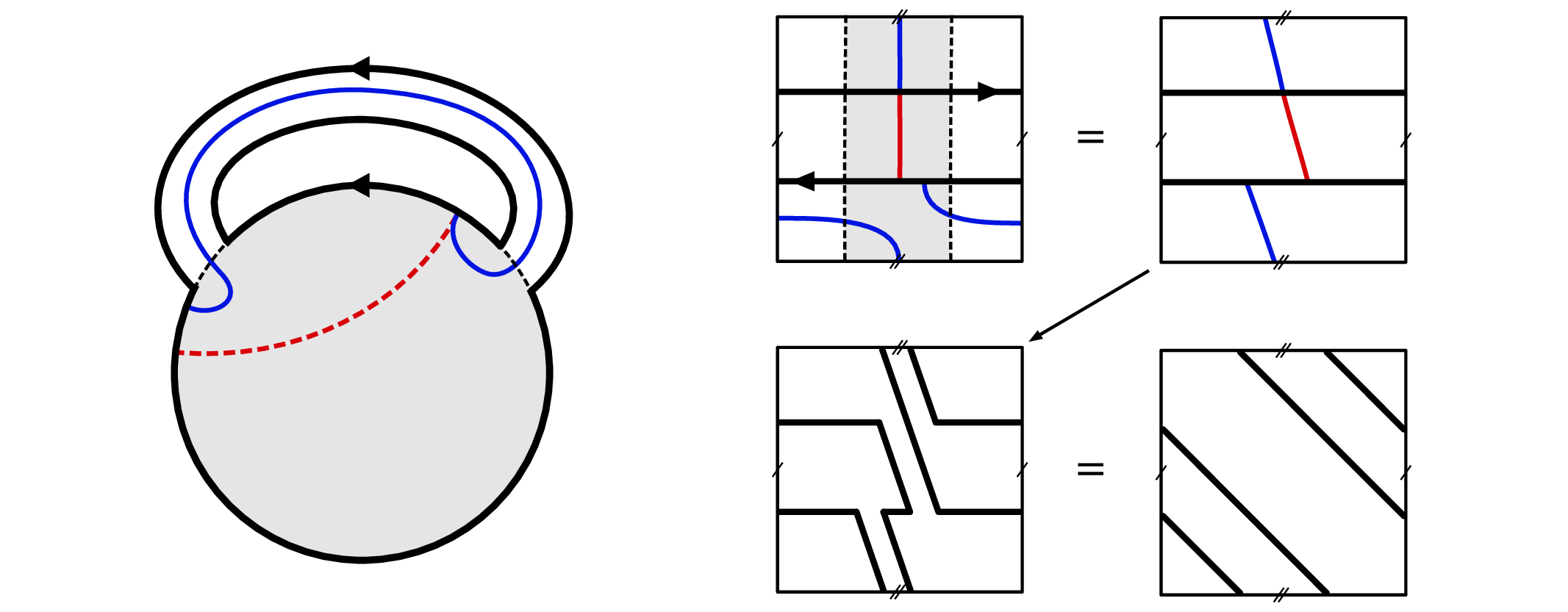}
    \put(22,0){\small $W_0$}
	\end{overpic}
	\caption{A (self) clasp bypass that Dehn twists the dividing curves on a standard convex torus.}
	\label{fig:slope}
\end{figure}

Next, we translate to bypass arcs and dividing curves in the style of \cite{honda2000classification}. First, we draw $\Sigma_0$ as a fundamental square with edges identified and dividing curves of slope $0$; see the top left square in \cref{fig:slope}. To aid in the translation, we have lightly shaded the doubled disk in the disk-band presentation. The bypass arc $D_+ \cup D_-$, after isotopy, is then given as the arc in the top right square. By \cite[Lemma 3.12]{honda2000classification}, the resulting dividing set after attaching the bypass is given by the lower row of squares and now has slope $-1$. 
\end{proof}

\subsection{Destabilizing Legendrian arcs}\label{subsec:destab_arc}

In addition to producing basic slices, the bypass disks of \cite{honda2000classification} allow for the destabilization of Legendrian knots, the boundary of a bypass disk being a union of two Legendrian arcs, one of which is obtained by stabilizing the other.  The result is a tight relationship between bypasses, basic slices, and stabilizations of Legendrian arcs, exploited, for instance, in classifications of Legendrian knots such as \cite{etnyre2001knots,tosun2012legendrian}.  Having interpreted the clasp move as a basic slice above, we show here how the clasp move witnesses the destabilization of a Legendrian arc.

\begin{proof}[Proof of \eqref{part:destab_arc} of \cref{thm:main_dim3}.]

Let $\Lambda_0\subset (Y, \xi)$ be an embedded stabilized Legendrian arc in a contact manifold. Let $W_0$ be its ribbon neighborhood bounded by the double transverse pushoff, so that the contact handlebody $H_0 = W_0 \times [-1,1]$ is the standard $J^1$-neighborhood of $\Lambda_0$. See the left panel of \cref{fig:3ddestab} for $\Lambda_0$ in the front projection, and the upper part of the middle panel for $\Lambda_0$ and $W_0$ in the Lagrangian projection. 

It is well-known that there is a bypass attachment in $(Y, \xi)$ which produces a standard neighborhood of the destabilization of $\Lambda_0$. The embedded bypass half-disk is drawn in the front projection on the left of \cref{fig:3ddestab}. In the Lagrangian projection in the middle panel, we take the green arc as the core of the contact $1$-handle associated to the bypass attachment and draw the core of the contact $2$-handle --- roughly the bypass half-disk --- accordingly.

\begin{figure}[ht]
	\centering
	\begin{overpic}[scale=.41]{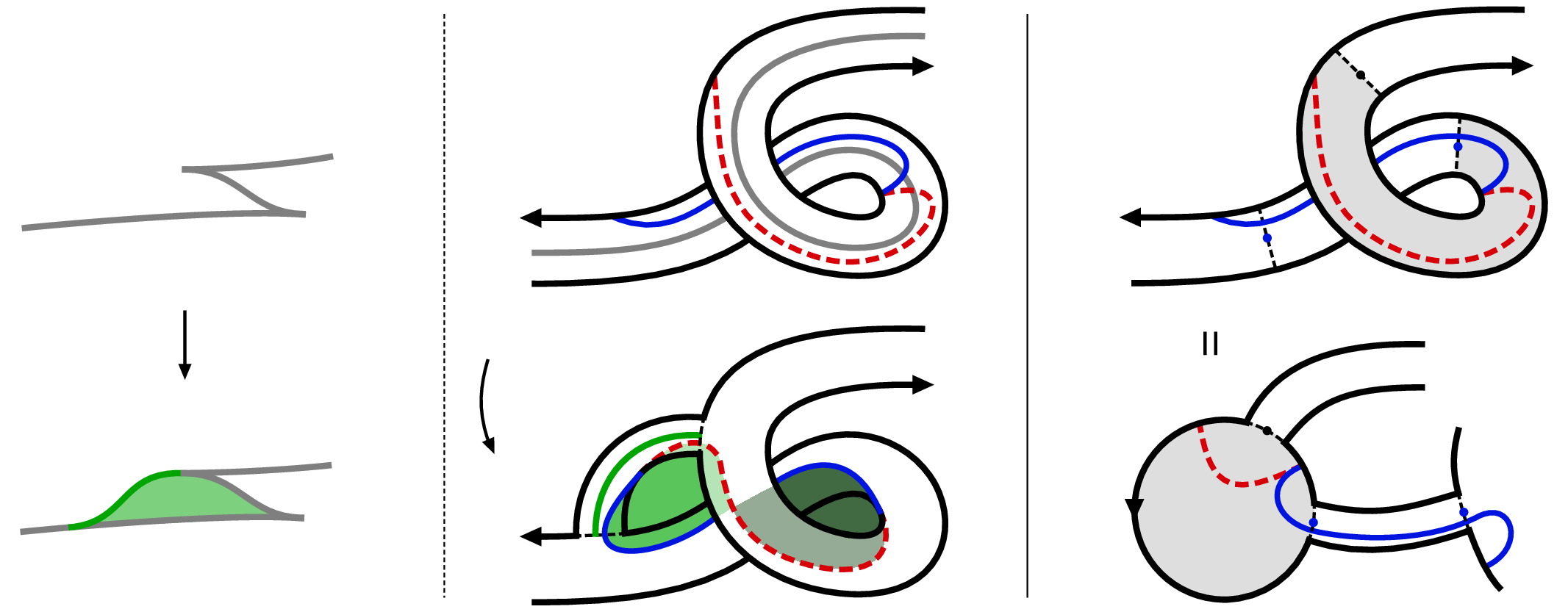}
    \put(10,30){\small \textcolor{gray}{$\Lambda_0$}}
	\end{overpic}
	\caption{Destablization of a neighborhood of a Legendrian arc as performed by a clasp bypass.}
	\label{fig:3ddestab}
\end{figure}

Finally, in the right panel we clean up the Lagrangian projection and identify a specific (local) disk-band presentation of $W_0$; a choice of $0$-handle is once again shaded for clarity. The top part of the right panel remains an embedded representation of the model in the Lagrangian projection, while the bottom part is an abstract version, unwound to more easily identify the nature of the bypass arc. In particular, the destabilizing bypass is a clasp move by \cref{subsec:3dclasp}.  
\end{proof}

\subsection{Legendrian graphs}\label{subsec:legendrian_graphs}

A \textit{Legendrian graph} in a contact $3$-manifold $(Y, \xi)$ is a spatial graph everywhere tangent to $\xi$. Giroux's proof \cite{Giroux2002GeometrieDC} of the correspondence between contact structures and open book decompositions in dimension $3$ used the language of Legendrian graphs; since then, Legendrian graphs have been studied both as interesting objects in their own right \cite{o2012legendrian,lambert2016planar} and because their ribbon neighborhoods are precisely the class of quasipositive surfaces \cite{baader2009graphs,hayden2022ribbons}.

The edges adjacent to a vertex $v$ of a Legendrian graph inherit a counterclockwise cyclic ordering determined by the oriented contact plane at $v$. The \textit{contact framing} of a Legendrian graph $G$, which generalizes the Thurston-Bennequin invariant of a knot, is the smooth isotopy type of the canonical Weinstein ribbon surface of $G$.

The clasping bypass allows us to transpose the cyclic order of two edges adjacent to a vertex in an arbitrarily small neighborhood. Together with arc destabilization, we can arbitrarily modify the contact framing of a Legendrian graph via the clasp move.

\begin{proposition}\label{prop:graph_prop}
Let $G\subset (Y, \xi)$ be a Legendrian graph and let $H$ be its standard neighborhood. Let $v \in G$ be a vertex. There is a Legendrian graph $G'\subset H$ with standard neighborhood $H' \subset H$ such that
\begin{enumerate}
    \item $G$ and $G'$ are isomorphic as abstract graphs,  
    \item $G$ and $G'$ agree outside of an arbitrarily small neighborhood $\mathcal{U}$ of $v$, 
    \item $G$ and $G'$ share $v$ as a vertex and the cyclic order of edges adjacent to $v$ in $G$ and $G'$ differs by an adjacent transposition, and 
    \item $H=H' \cup B$, where $B$ is a(n embedded) clasping bypass attachment. 
\end{enumerate}
Specifically, up to Legendrian graph isotopy and contactomorphism of neighborhoods, $G$ and $G'$ differ near $v$ according to \cref{fig:graph_prop}.
\end{proposition}

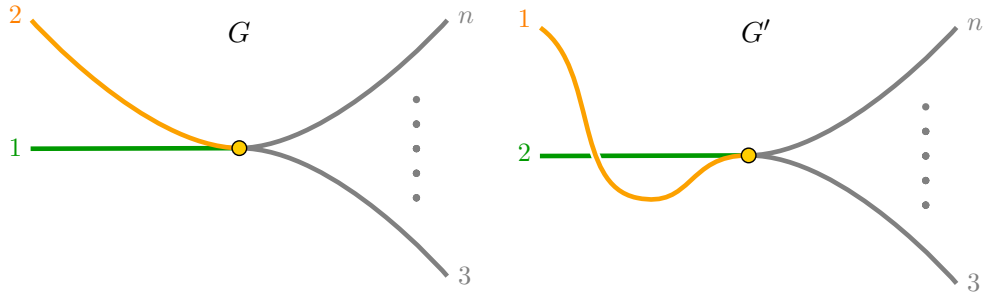
\begin{figure}[ht]
	\centering
	\begin{overpic}[scale=.35]{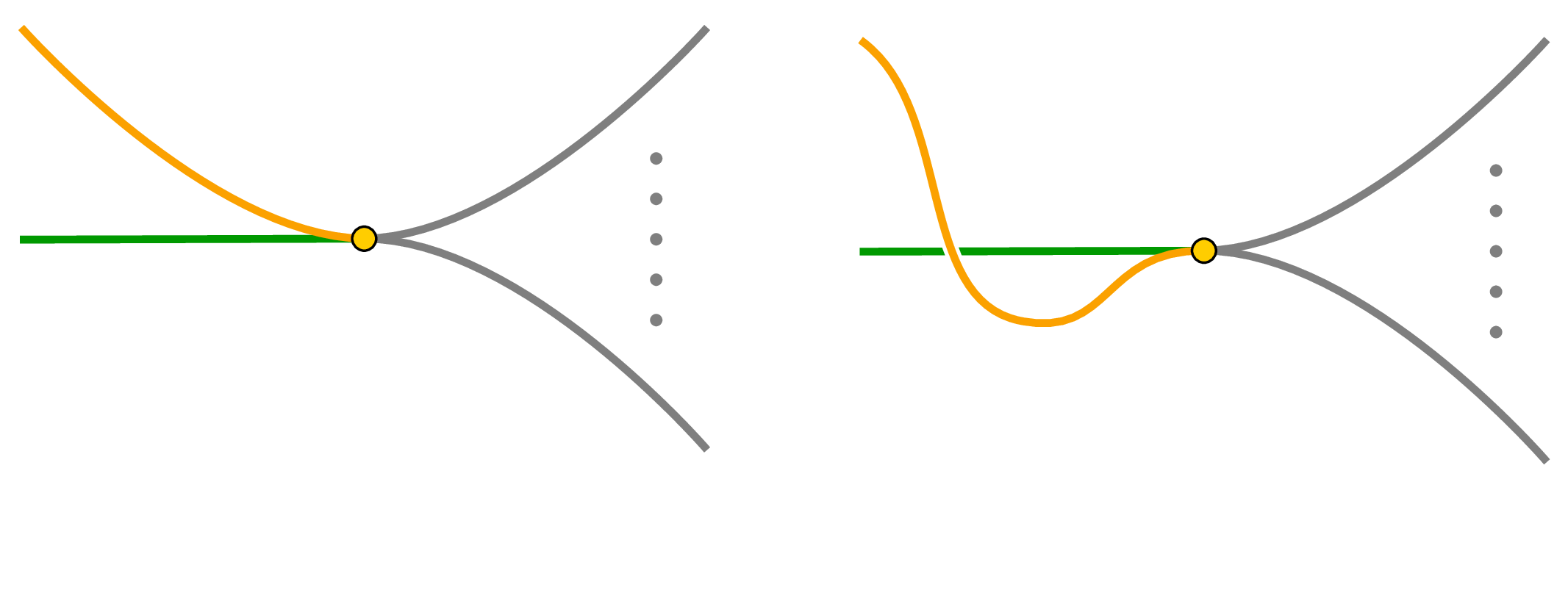}
    \put(22,35){$G$}
    \put(76,35){$G'$}

    \put(-1,37){\small \textcolor{orange}{$2$}}
    \put(52.5,36.5){\small \textcolor{orange}{$1$}}
    \put(-1,23){\small \textcolor{darkgreen}{$1$}}
    \put(52.5,22.5){\small \textcolor{darkgreen}{$2$}}
    \put(46.25,9.5){\small \textcolor{gray}{$3$}}
    \put(46.25,37){\small \textcolor{gray}{$n$}}
    \put(99.75,8.75){\small \textcolor{gray}{$3$}}
    \put(99.75,36.25){\small \textcolor{gray}{$n$}}
	\end{overpic}
    \vskip-1cm
	\caption{Changing the cyclic order of edges in a Legendrian graph via a bypass move in the front projection.}
	\label{fig:graph_prop}
\end{figure}

Repeated application of \cref{prop:graph_prop} then gives:  

\begin{corollary}
Let $G\subset (Y, \xi)$ be a Legendrian graph. Let $W'$ be an orientable surface homotopy equivalent to $G$. There is a Legendrian graph $G'\subset Y$ such that
\begin{enumerate}
    \item $G'$ has the same vertices as $G$, 
    \item $G'$ agrees with $G$ outside of an arbitrarily small neighborhood of the vertex set, 
    \item $G'$ has contact framing given by $W'$, and 
    \item a standard neighborhood of $G$ is obtained by attaching a sequence of clasping bypasses to a standard neighborhood of $G'$.
\end{enumerate}
\end{corollary}

\begin{proof}[Proof of \cref{prop:graph_prop}.]
Up to contactomorphism and repeated applications of Legendrian graph Reidemeister move VI in \cite{o2012legendrian}, we may assume that in a sufficiently small neighborhood of $v$, $G$ has a front projection given by the left side of \cref{fig:graph_prop}. Then we define $G'$ as on the right. Clearly $G$ and $G'$ are isomorphic as abstract graphs. By construction they share $v$ as a vertex, agree outside of the model neighborhood, and the cyclic order of edges labeled $1$ and $2$ has been transposed. 

Thus, to prove the proposition it suffices to identify a bypass taking a neighborhood of $G'$ to a neighborhood of $G$, and to show that the bypass corresponds to a clasp move. At this point the proof is very similar to that of arc destabilization. First, we draw $G'$ in a a neighborhood of $v$ in the Lagrangian projection, together with its standard ribbon neighborhood; this is given by the far left panel of \cref{fig:graph_proof}.

\begin{figure}[ht]
	\centering
	\begin{overpic}[scale=.41]{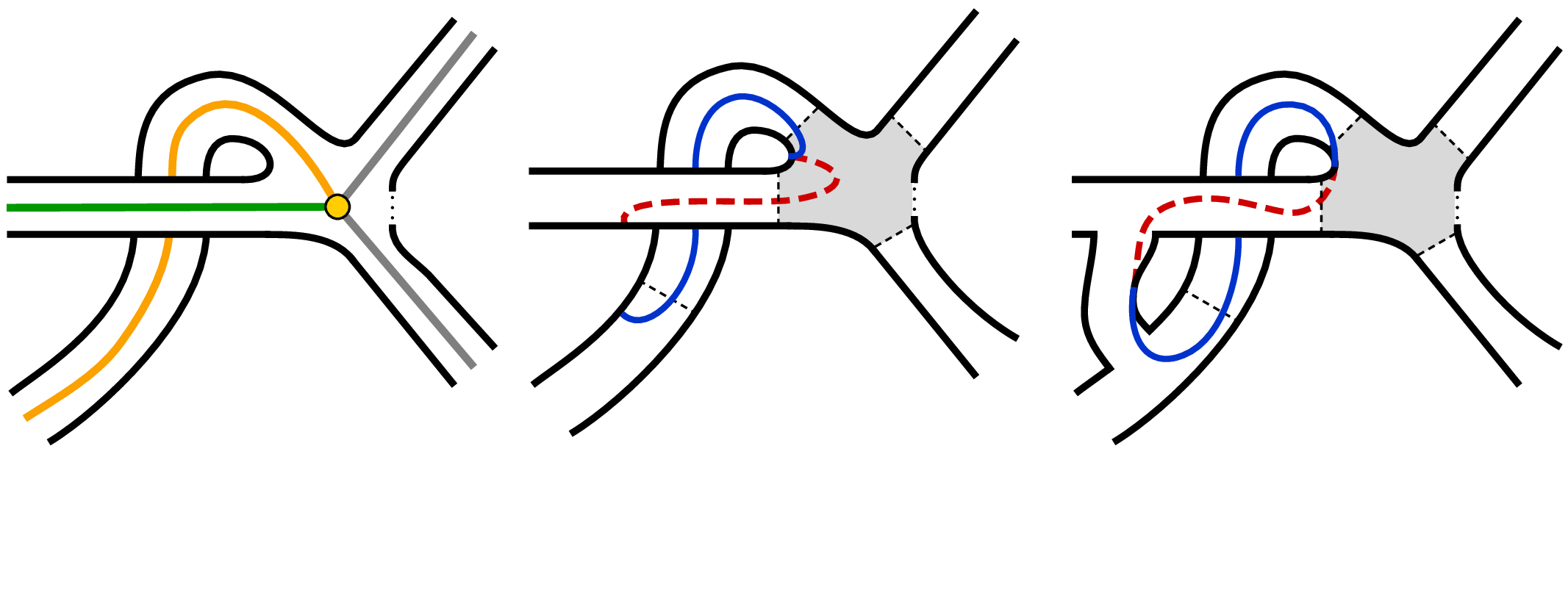}
    \put(20,35){$G'$}

    \put(-0.25,10.75){\small \textcolor{orange}{$1$}}
    \put(-1.5,25){\small \textcolor{darkgreen}{$2$}}
    \put(30.75,13.75){\small \textcolor{gray}{$3$}}
    \put(30.75,37.75){\small \textcolor{gray}{$n$}}
    
	\end{overpic}
    \vskip-1.5cm
	\caption{Identifying the clasping bypass near a vertex of a Legendrian graph in the Lagrangian projection.}
	\label{fig:graph_proof}
\end{figure}
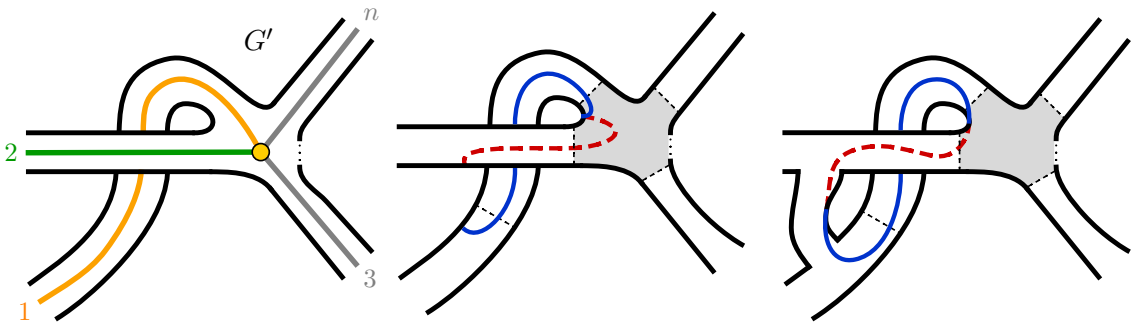

We can take a neighborhood of $v$ in the Weinstein ribbon to be a $0$-handle, and neighborhoods of each edge to be bands attached to the $0$-handle. As in the proof of arc destabilization, we have shaded the $0$-handle in the middle panel for clarity. Then we identify the bypass arc in the middle figure, and observe that (up to an isotopy of the dashed red arc) the data corresponds to a clasp move of the bands numbered $1$ and $2$. In the right panel, we have attached the $1$-handle associated to the bypass data and observe that, after sliding the bypass arc across the $1$-handle and gluing the ends, the resulting knot is a max-tb unknot and bounds a disk in the complement of the neighborhood of $G'$. Hence, the bypass attachment exists in the ambient contact manifold, and this completes the proof. 
\end{proof}

\subsection{Convex Heegaard splittings}\label{subsec:convex_heegaard}

A \textit{convex Heegaard splitting} of a closed contact manifold $(Y, \xi)$ is simply a smooth Heegaard splitting where the Heegaard surface is convex. (A \textit{contact Heegaard splitting} is a stronger notion where each handlebody is a contact handlebody, and is equivalent to a supporting open book decomposition.) Here we prove: 

\heegaard*

\begin{proof}
Let $H$ be one of the smooth handlebodies in the associated Heegaard splitting. Let $\Lambda_0\subset H$ be an embedded Legendrian graph such that $H$ deformation retracts onto $\Lambda_0$. One can identify such a graph by, for instance, picking a gradient-like vector field for a Morse function $\phi: Y \to \R$ such that $H = \phi^{-1}(-\infty, 0]$ and performing Legendrian approximation on the union of stable manifolds of critical points of index $1$ and $2$. Let $W_0$ be the standard ribbon neighborhood of $\Lambda_0$ and $H_0 = W_0 \times [-1,1]$ the standard contact handlebody neighborhood. Note that, after rounding corners, $\Sigma_0 := \partial H_0$ is a convex Heegaard surface smoothly isotopic to $\Sigma$. 

Let $W'$ be the orientable (abstract) surface of genus $0$ with the same Euler characteristic as $W_0$ and let $H' = W' \times [-1,1]$ be the (abstract) contact handlebody over $W'$. By \cref{thm:main_dim3}, there is a sequence of clasp move bypass attachments $B_1, \dots, B_n$ such that $H' \cup B_1 \cup \cdots \cup B_n$ is contactomorphic rel boundary to $H_0$. This contactomorphism gives a contact embedding $H' \hookrightarrow H_0 \subset Y$ such that $H_0 \setminus H' \cong \Sigma' \times [0,1]$. In particular, $\Sigma' = \partial H'$ is a convex Heegaard surface smoothly isotopic to $\Sigma$ such that $R_+(\Sigma')$ and $R_-(\Sigma')$ are diffeomorphic to $W'$. This gives \eqref{part:heegaard_planar}.

If $g(\Sigma) \neq 1$, then we may repeat the same argument above but instead beginning with a(n abstract) surface $W''$ with connected boundary and Euler characteristic that coincides with $W_0$. This gives \eqref{part:heegaard_connected}.
\end{proof}

\section{Appearance of bypass moves in higher dimensions}\label{sec:appearance_high_dim}
In this section, we consider two manifestations of the bypass moves of \cref{thm:main_moves} in higher dimensions.  First, we follow Ding-Geiges-van Koert \cite{Ding2012Diagrams} in considering the moves not as bypasses, but as modifications which can be made to the page of an abstract open book, and study the effect of these modifications on the resulting contact manifold.  Next, we use this understanding of the bypass moves as Weinstein modifications to construct convex hypersurfaces in contact manifolds with interesting positive and negative regions.

\subsection{Subcritically fillable contact structures}\label{subsec:subcritically_fillable}
Here, we adopt the perspective of Ding-Geiges-van Koert \cite{Ding2012Diagrams} and prove:

\movesForObds*

As explained in the introduction, \cref{thm:moves_for_obds} admits a proof which omits any mention of bypasses.  Indeed, consider the following $h$-principle of Cieliebak-Eliashberg:

\begin{theorem}[{\cite[Theorem 14.3]{cieliebak2012stein}}]\label{thm:ce-h-principle}
Let $(W_0,\lambda_0, \phi_0)$ and $(W_1,\lambda_1, \phi_1)$ be subcritical Weinstein domains of dimension at least 6.  If there is a diffeomorphism $f\colon W_0\to W_1$ such that $f^*d\lambda_1$ is homotopic to $d\lambda_0$ as a nondegenerate 2-form, then there is a diffeotopy $f_t\colon W_0\to W_1$, $t\in[0,1]$, such that $f_0=f$ and $f_1$ is an exact symplectic deformation equivalence.
\end{theorem}

Note that none of the local moves described in \cref{thm:main_moves} modify the rotation class of the Legendrian to which they are applied.  Moreover, if $\Lambda\subset\partial W$ is a Legendrian in the boundary of a Weinstein domain, then the rotation class of $\Lambda$ determines the formal isotropic isotopy class of the corresponding isotropic $\Lambda\times\{0\}\subset\partial(W\times D^2)$ in the stabilized Weinstein domain.  Now if $(W',\lambda', \phi')$ is obtained from $(W,\lambda, \phi)$ by any of the moves of \cref{thm:main_moves}, then $(W\times D^2,\lambda+\lambda_{\mathrm{std}}, \phi+ \phi_{\mathrm{std}})$ and $(W'\times D^2,\lambda'+\lambda_{\mathrm{std}}, \phi'+ \phi_{\mathrm{std}})$ admit Weinstein handle decompositions whose attaching spheres are subcritical and pairwise formally isotropically isotopic.  By the $h$-principle for subcritical isotropics, these attaching spheres are genuinely isotropically isotopic, and thus the stabilized domains satisfy the hypotheses of \cref{thm:ce-h-principle}.  We conclude that $(W\times D^2,\lambda+\lambda_{\mathrm{std}}, \phi+ \phi_{\mathrm{std}})$ and $(W'\times D^2,\lambda'+\lambda_{\mathrm{std}}, \phi'+ \phi_{\mathrm{std}})$ have contactomorphic boundaries.  Nevertheless, we present an alternate proof that carefully examines the bypasses corresponding to the moves of \cref{thm:main_moves}, with an eye towards using the moves in abstract open books with nontrivial monodromy.

\begin{proof}[Proof of \cref{thm:moves_for_obds}]
Suppose that $(W',\lambda', \phi')$ is obtained from $(W,\lambda, \phi)$ by applying the vertical clasp move on two distinct attaching spheres $\Lambda_0$ and $\Lambda_1$; the case of a self-clasp is nearly identical. Let $(\Lambda_{\pm};D_{\pm})$ be the bypass attachment data identified in the proof of \cref{prop:clasping_bypass};  see \cref{fig:5proof}.  We will show that $(W'',\tau_L)$ is a common stabilization of the abstract open books $(W,\mathrm{id})$ and $(W',\mathrm{id})$, where $(W'',\lambda'',\phi'')$ is obtained from $(W,\lambda,\phi)$ by attaching a Weinstein $n$-handle to $W$ along the Legendrian sphere $\Lambda_- \uplus \Lambda_+$ and $L\subset W''$ is a particular Lagrangian sphere that we construct explicitly.

\begin{figure}[ht]
	\centering
	\begin{overpic}[scale=.35]{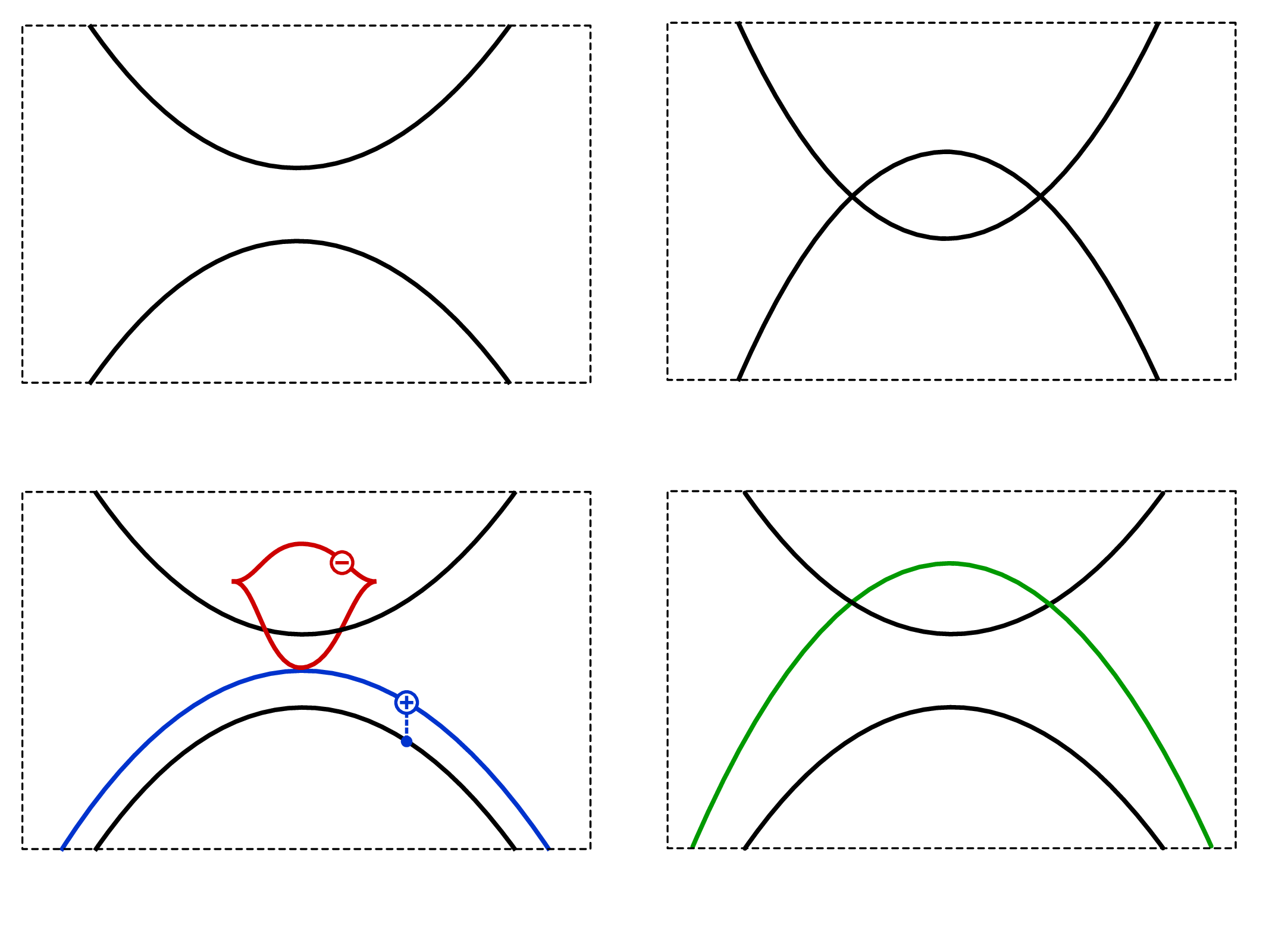}
	\put(16.5,32){\color{darkred}\small $\Lambda_-$}
	\put(53.5,26){\small \textcolor{darkgreen}{$\Lambda_-\uplus \Lambda_+$}}
	\put(9.5,20){\color{darkblue}\small $\Lambda_+$}
	
	\put(92.5,18){\tiny $(-1)$}
 \put(92.5,26){\tiny $(-1)$}
	\put(92.5,22){\tiny \textcolor{darkgreen}{$(-1)$}}
	
	\put(41.5,54){\tiny $(-1)$}
 \put(41.5,62){\tiny $(-1)$}
 \put(41.5,18){\tiny $(-1)$}
 \put(41.5,26){\tiny $(-1)$}
 \put(92.5,54){\tiny $(-1)$}
 \put(92.5,62){\tiny $(-1)$}

    \put(22.5,42){\small $W$}
    \put(22.5,5.25){\small $W$}
    \put(73.5,5.25){\small $W''$}
    \put(73.75,42){\small $W'$}

    \put(3,46){\small $\Lambda_0$}
    \put(3,70){\small $\Lambda_1$}

    \put(53.5,46.5){\small $\Lambda_0'$}
    \put(53.5,70){\small $\Lambda_1'$}
    \put(62,9.25){\small $\Lambda_0$}
    \put(62,33.5){\small $\Lambda_1$}
	
	\end{overpic}
	\vskip-.3in
	\caption{The proof of \cref{thm:moves_for_obds}.}
	\label{fig:5proof}
\end{figure}

Our construction of $L$ begins with the \emph{Legendrian boundary sum} $D_-\uplus_b D_+ \subset M$, where $(M\simeq W\times[-1,1],\xi=\ker(dz+\lambda))$ is a contact handlebody over $W$.  To obtain $D_-\uplus_b D_+$ we first note that, following a slight perturbation into the interior of $M$, $D_-$ and $D_+$ are Legendrian disks satisfying  the following conditions:
\begin{enumerate}[label=(\Alph*)]
    \item the interiors of $D_-$ and $D_+$ are disjoint;\label{cond:leg-bdry-sum-disjoint-interior}
    \item the boundaries $\partial D_-=\Lambda_-$ and $\partial D_+=\Lambda_+$ are Legendrian submanifolds of $(\Gamma,\xi|_\Gamma)$ which intersect $\xi|_{\Gamma}$-transversely at a single point $p$;\label{cond:leg-bdry-sum-transverse}
    \item the dividing set $\Gamma$ of $\partial M$ admits a tubular neighborhood $N(\Gamma)\subset (M,\xi)$ contactomorphic to $(\partial W\times D,\ker(\lambda|_{\partial W} + r^2\,d\theta))$, where $D=\{(r,\theta)\,|\, 0\leq r\leq \delta, 0\leq \theta\leq \pi\}$ is a closed half-disk of radius $\delta>0$ with polar coordinates $(r,\theta)$, such that
    \begin{itemize}
        \item $D_- \cap N(\Gamma) = \Lambda_- \times \{\theta=\theta_-\}$;
        \item $D_+ \cap N(\Gamma) = \Lambda_+ \times \{\theta=\theta_+\}$;
    \end{itemize}
    for some $0 < \theta_- < \theta_+ < \pi$.\label{cond:leg-bdry-sum-cylindrical}
\end{enumerate}
We observe that these conditions are satisfied by any bypass attachment data $(\Lambda_\pm;D_\pm)$.  Together, Conditions~\ref{cond:leg-bdry-sum-disjoint-interior}-\ref{cond:leg-bdry-sum-cylindrical} ensure that the Legendrian boundary sum $D_-\uplus_b D_+$, first constructed in~\cite[Section 4.3]{HH18}, is well-defined.  The result is a Legendrian disk in $M$ with boundary $\partial(D_-\uplus_b D_+) = \Lambda_- \uplus \Lambda_+ \subset \Gamma$ which is smoothly equivalent to the boundary connected sum of $D_-$ and $D_+$ at the point $p$.  We now recall the definition given in~\cite[Section 4.3]{HH18}.

Away from a neighborhood of $p$, $D_-\uplus_b D_+$ agrees with the union $D_-\cup D_+$.  In a neighborhood of $p$ we consider a chart $U=\mathbb{R}_z\times Q_{x_1,y_1}\times\mathbb{R}^{2n-2}_{x,y}$, where
\[
Q = \{(x_1,y_1)\,|\, x_1\geq 0 \text{~or~} y_1\geq 0\}
\]
and $x=(x_2,\ldots,x_n)$ and $y=(y_2,\ldots,y_n)$ are coordinates on $\mathbb{R}^{2n-2}$.  We choose these coordinates so that 
\[
\lambda|_U = dz + \frac{1}{2}\sum_{i=2}^n x_i \, dy_i - y_i\, dx_i
\quad\text{and}\quad
r^2\,d\theta|_U = \frac{1}{2}(x_1 \, dy_1 - y_1\, dx_1),
\]
where $(r,\theta)$ are the polar coordinates on $N(\Gamma)$ identified above.  Moreover, we have
\[
\Pi(D_- \cap U) = \{x_1\geq 0, y_1=0\} \times \{y=0\}
\quad\text{and}\quad
\Pi(D_+ \cap U) = \{x_1=0, y_1\geq 0\} \times \{x=0\},
\]
where $\Pi\colon U \to Q\times\mathbb{R}^{2n-2}$ is the Lagrangian projection.  Then $(D_-\uplus_b D_+)\cap U$ is defined to be the Legendrian submanifold
\[
(D_-\uplus_b D_+)\cap U := \left\{ z=t, x_i=e^t\,a_i, y_i=e^{-t}\,a_i \,\Big|\, \sum_{i=1}^n a_i^2 = 1, a_1\geq 0, t\in\mathbb{R} \right\}.
\]
We leave the global smoothing of $D_-\uplus_b D_+$ as $t\to\pm\infty$ to the reader. 

Consider the image $D=\pi_W(D_-\uplus_b D_+)$, where $\pi_W\colon M\to W$ is the natural projection map.  This image is Lagrangian, but \emph{a priori} singular; we now verify that $D$ is in fact smoothly embedded in $W$.  Where $D_-\uplus_b D_+$ agrees with the union $D_-\cup D_+$, our choice of $(\Lambda_{\pm};D_{\pm})$ ensures that $\pi_W(D_-)$ and $\pi_W(D_+)$ have disjoint interiors.  Indeed, $\pi_W(D_+)$ is a small forward-time Reeb isotopy of the co-core of an $n$-handle in $(W,\lambda,\phi)$, while $\pi_W(D_-)$ does not cross over this handle.  It follows that $D$ is nonsingular in the region where $D_-\uplus_b D_+$ agrees with $D_-\cup D_+$.  In the neighborhood $U$ of $p$ described above, the map $\pi_W$ agrees with the natural projection map
\[
\pi_W\colon \mathbb{R}_z\times Q_{r,\theta}\times\mathbb{R}^{2n-2}_{x,y} \to \mathbb{R}_z\times [0,\infty)_r \times\mathbb{R}^{2n-2}_{x,y},
\]
where $(r,\theta)$ are again the polar coordinates used above.  So $D$ is described in this neighborhood by
\begin{multline*}
\pi_W((D_-\uplus_b D_+)\cap U) =\\ \left\{ z=t, r=a_1\sqrt{e^{2t}+e^{-2t}}, x_i=e^t\,a_i, y_i=e^{-t}\,a_i \,\Big|\, \sum_{i=1}^n a_i^2 = 1, a_1\geq 0, t\in\mathbb{R} \right\},    
\end{multline*}
thus is also nonsingular.  Finally, because $D_-$ sits above $D_+$ near $p$ with respect to the Reeb flow of $\Gamma$, the smoothing of $D_-\uplus_b D_+$ introduces no self-intersections in $D$.  Specifically, the definition of $D_-\uplus_b D_+$ forces the portion of $D_-\uplus_b D_+$ corresponding to $D_-$ to approach $p$ from above while the portion corresponding to $D_+$ approaches from below.  Because this matches the configuration of $D_-$ and $D_+$, the smoothing may be performed to ensure that the $z$-coordinate of $D_-\uplus_b D_+$ is positive in its portion corresponding to $D_-$ and negative in its portion corresponding to $D_+$.  Note that this is not the case for arbitrary bypass attachment data $(\Lambda_\pm;D_\pm)$.  Altogether, $D$ is a smooth Lagrangian filling of $\Lambda_-\uplus\Lambda_+$ in $W$.

Finally, since $(W'',\lambda'',\phi'')$ is obtained from $(W,\lambda,\phi)$ by attaching a Weinstein $n$-handle to $W$ along $\Lambda_- \uplus \Lambda_+$, we let $L\subset W''$ be the Lagrangian sphere which results from attaching the core of this $n$-handle to $D$.  By construction, the open book $(W'',\tau_L)$ is a positive stabilization of $(W,\mathrm{id})$. At the same time, $(W'',\lambda'',\phi'')$ includes a Weinstein $n$-handle attached along $\Lambda_0$ whose removal yields $(W',\lambda',\phi')$ and $L$ intersects the co-core of this handle exactly once, making $(W'',\tau_L)$ a positive stabilization of $(W',\mathrm{id})$; see \cref{fig:5proof}. (This latter positive stabilization corresponds to the trivial bypass observed in \cref{fig:clasp_move}.)  It follows that each of $\mathcal{OB}(W,\mathrm{id})$ and $\mathcal{OB}(W',\mathrm{id})$ is contactomorphic to $\mathcal{OB}(W'',\mathrm{id})$.

As explained in \cref{sec:bypass_moves}, the moves of horizontal (un)\nobreak\hspace{0pt}clasping, (de)\nobreak\hspace{0pt}stabilization, and crossing change may each be obtained as applications of vertical (un)\nobreak\hspace{0pt}clasping, and thus any Weinstein domains $(W,\lambda,\phi')$, $(W',\lambda',\phi')$ related by these other moves provide the pages for abstract open books $(W,\mathrm{id})$ and $(W',\mathrm{id})$ whose resulting contact manifolds are contactomorphic.
\end{proof}

\begin{remark}
We recall a subtlety in the difference between stabilizations of partial and full open book decompositions; see \cref{subsec:OBD}. The open book $(W,\mathrm{id})$ is obtained by "closing up" the partial open book $(W,\mathrm{id}\colon S\to S)$, where $S$ is a collar neighborhood of $\partial W\subset W$.  While the proof above constructs a common stabilization of $(W,\mathrm{id})$ and $(W',\mathrm{id})$, it does not claim the existence of a common stabilization of $(W,\mathrm{id}\colon S\to S)$ and $(W',\mathrm{id}\colon S'\to S')$.  If we trace through the proof in the relative setting, we find stabilizations of each of $(W,\mathrm{id}\colon S\to S)$ and $(W',\mathrm{id}\colon S'\to S')$ that have the same page and have monodromy given by a positive Dehn twist about the same Lagrangian sphere. However, the domains of these monodromies could differ, resulting in distinct contact manifolds-with-boundary.
\end{remark}

A full converse to \cref{thm:moves_for_obds} seems unlikely, as this would imply that all subcritically fillable contact manifolds admit unique subcritical fillings, up to deformation equivalence.  However, the following theorem provides a characterization of Weinstein domains which are related by the moves described in \cref{thm:main_moves}.  We thank an anonymous referee for substantially simplifying its proof.

\movesSufficient*

\begin{proof}
In \cite{Cie02subscritical}, Cieliebak proves that all subcritical Weinstein domains are split according to the following strategy: suppose that $V$ is a subcritical Weinstein domain of dimension $2n+2$, obtained from $(W_0\times D^2,\lambda_0+\lambda_{\mathrm{std}},\phi_0+\phi_{\mathrm{std}})$, where $(W_0,\lambda_0,\phi_0)$ is a Weinstein domain, by attaching a Weinstein $k$-handle, for some $k\leq n$.  Letting $\tilde{f}\colon S^{k-1}\hookrightarrow\partial(W_0\times D^2)$ denote the isotropic attaching sphere of this handle, Cieliebak shows that $\tilde{f}$ is formally isotropically isotopic (and thus genuinely isotropically isotopic, being subcritical) to an isotropic sphere of the form $f\times\{0\}\subset \partial W_0\times D^2$, for some isotropic embedding $f\colon S^{k-1}\hookrightarrow\partial W_0$. It follows that $V$ is split with respect to the same $D^2$-factor which appears in $W_0\times D^2$; because any subcritical Weinstein domain may be obtained from $D^{2n}\times D^2$ by successively attaching subcritical Weinstein handles, the result follows inductively.

An important feature of Cieliebak's proof is that the isotropic isotopy type of $f$ is not uniquely determined by $\tilde{f}$.  However, we make the following claim:
\begin{quote}
\textbf{Claim.}  Up to possibly applying clasp moves as described in \cref{thm:main_moves}, $f$ is uniquely determined by $\tilde{f}$.
\end{quote}
The implication \eqref{thm-part:stably-almost-symplectomorphic} $\implies$ \eqref{thm-part:related-by-moves} will follow from this claim, as we may let $(V,\beta,\psi)$ denote a fixed Weinstein domain to which $(W\times D^2,d\lambda+\omega_{\mathrm{std}})$ and $(W'\times D^2,d\lambda'+\omega_{\mathrm{std}})$ are symplectic deformation equivalent --- the existence of which follows from the $h$-principle for subcritical Weinstein domains --- and any Weinstein handle decomposition of $(V,\beta,\psi)$ may be used to produce Weinstein handle decompositions of $(W,\lambda,\phi)$ and $(W',\lambda',\phi')$ which are related by clasp moves. The implication \eqref{thm-part:related-by-moves} $\implies$ \eqref{thm-part:stably-almost-symplectomorphic} was established in \cref{sec:introduction}, following the statement of \cref{thm:moves_sufficient_for_stable_equivalence}.

We now work towards our claim above by investigating the formal data associated to $f$ and $\tilde{f}$.  First, note that attaching a Weinstein handle to $W_0$ along $f$ requires a conformally symplectic trivialization of the conformal symplectic normal bundle of $f$ if $k<n$.  We will take this trivialization to be a complex bundle monomorphism $\theta\colon S^{k-1}\times\mathbb{C}^{n-k}\to f^*\xi_{W_0}$, where we have fixed some almost complex structure $J_0$ on $W_0$ compatible with the Weinstein structure.  The handle attachment is then carried out using a normal framing $\beta\colon S^{k-1}\times\mathbb{R}^{2n-k}\to\nu_f$ of $f$ in $\partial W_0$ and a complex bundle isomorphism $\gamma\colon S^{k-1}\times\mathbb{C}^n\to f^*TW_0$ defined by
\[
\beta = Jdf \oplus R_{\lambda_0} \oplus\theta\colon TS^{k-1}\oplus\underline{\mathbb{R}}\oplus\underline{\mathbb{C}}^{n-k}\to\nu_f
\]
and
\[
\gamma = df_{\mathbb{C}}\oplus R_{\lambda_0}\oplus X_{\lambda_0}\oplus\theta\colon (TS^{k-1}\otimes\mathbb{C})\oplus\underline{\mathbb{R}}^2\oplus\underline{\mathbb{C}}^{n-k}\to f^*TW_0,
\]
respectively, where
\begin{itemize}
    \item $R_{\lambda_0}$ and $X_{\lambda_0}$ denote the Reeb and Liouville vector fields associated to $\lambda_0$, respectively;
    \item we identify $TS^{k-1}\oplus\underline{\mathbb{R}}$ with the trivial bundle $S^{k-1}\times\mathbb{R}^k$ via the natural embedding of $S^{k-1}$ into $\mathbb{R}^k$;
    \item $df_\mathbb{C}\colon TS^{k-1}\otimes\mathbb{C}\to f^*\xi_{W_0}$ denotes the complexification of $df$, a complex bundle monomorphism.
\end{itemize}
There are analogous framings $\tilde{\beta}, \tilde{\gamma}$ associated to $\tilde{f}$, and Cieliebak's work constructs $f$ so that the isotopy connecting $\tilde{f}$ to $f\times\{0\}$ is covered by isotopies connecting $\tilde{\beta}$ and $\tilde{\gamma}$ to $\beta\times\mathrm{id}_\mathbb{C}$ and $\gamma\times\mathrm{id}_\mathbb{C}$, respectively.  Now suppose that $\gamma'\colon S^{k-1}\times\mathbb{C}^n\to f^*TW_0$ is some other complex bundle isomorphism with the property that $\gamma'\times\mathrm{id}_\mathbb{C}$ and $\gamma\times\mathrm{id}_\mathbb{C}$ are isotopic.  By defining $g(p):=\gamma_p^{-1}\circ\gamma'_p$ for all $p\in S^{k-1}$, we obtain $[g]\in\pi_{k-1}U(n)$ relating $\gamma$ to $\gamma'$.  Moreover, because $\gamma'\times\mathrm{id}_\mathbb{C}$ and $\gamma\times\mathrm{id}_\mathbb{C}$ are isotopic, $[g]$ is contained in the kernel of the map
\[
\pi_{k-1}U(n) \xrightarrow{\times\mathrm{id}_\mathbb{C}} \pi_{k-1}U(n+1).
\]
But since $k\leq 2n$, we know from \cite[Lemma 23.4]{milnor1963morse} that this map is an isomorphism. So $[g]$ is trivial and we see that the homotopy class of $\gamma$ in the space of complex bundle isomorphisms is uniquely determined by $\tilde{\gamma}$. It follows that $\tilde{\gamma}$ determines the homotopy class of
\[
df_{\mathbb{C}}\colon TS^{k-1}\otimes\mathbb{C}\to f^*\xi_{W_0}
\]
in the space of complex bundle monomorphisms, where $\xi_{W_0}$ is the contact structure on $\partial W_0$ induced by $\lambda_0$.

With the homotopy class of $df_{\mathbb{C}}$ now determined, there are two cases.  If $k<n$, then the $h$-principle for subcritical isotropic embeddings (c.f. \cite[Theorem 12.4.1]{eliashberg2002h}) tells us that $f$ is determined up to isotopy through isotropic embeddings, so we are done.  If $k=n$, then $f$ is a Legendrian embedding and the homotopy class of $df_{\mathbb{C}}$ is its rotation class.  The $h$-principle for Legendrian immersions (c.f. \cite{gromov1971topological,duchamp1984classification,gromov1986partial} and \cite[Theorem 16.1.3]{eliashberg2002h}) then tells us that $f$ is determined up to regular homotopy through Legendrian immersions\footnote{By applying the reasoning used for $\gamma$ to $\beta$, we find that $\tilde{\beta}$ determines $\beta$ up to choosing an element of $\ker(\pi_{k-1}SO(2n-k)\to\pi_{k-1}SO(2n+2-k))$.  This kernel can be nontrivial when $k=n$, reflecting the fact that formal isotopy types can vary under regular homotopy through Legendrian immersions.}.  Concretely, if $f_0,f_1\colon S^{n-1}\hookrightarrow \partial W_0$ are two Legendrian embeddings with equivalent rotation classes, then we may construct a regular homotopy $f_t$, $0\leq t\leq 1$, of Legendrian immersions such that $f_t\colon S^{n-1}\hookrightarrow \partial V$ fails to be a Legendrian embedding for finitely many values of $t$, at which times $f_t$ includes a $\xi$-transverse self-intersection.  But this is precisely the failure of Legendrian embedding modeled by the clasp move, so front projections of $f_0$ and $f_1$ are related by a sequence of clasp moves.
\end{proof}

We conclude this subsection with a brief discussion of how the local moves of \cref{thm:main_moves} affect the formal data associated to a Legendrian submanifold.  The key point in our bypass-free proof of \cref{thm:moves_for_obds} and in proving the implication \eqref{thm-part:related-by-moves} $\implies$ \eqref{thm-part:stably-almost-symplectomorphic} of \cref{thm:moves_sufficient_for_stable_equivalence} was that neither the (de)stabilization move nor the (un)clasping move will modify the rotation class of a Legendrian attaching sphere.  Indeed, the (de)stabilization move preserves the formal Legendrian isotopy class altogether in ambient contact dimension at least 5.  However, the (un)clasping move will typically have a nontrivial effect on the formal Legendrian isotopy class.  If the two strands depicted in the top left of \cref{fig:thm_moves} are oppositely-oriented portions of the same Legendrian $n$-sphere $\Lambda$, then applying the clasp move amounts to performing a \emph{$B^n$-stabilization}, in the language of \cite[Section 7.4]{cieliebak2012stein} or \cite[Section 2.2.1]{casals2020non} (also see \cite[Section 4.3]{ekholm2005non}).  Because $\chi(B^n)=1$, the discussion at the end of \cite[Appendix B]{cieliebak2012stein} tells us that this move preserves the formal Legendrian isotopy class of $\Lambda$ only when $n=2$.  In the particular case where $n$ is odd and $\Lambda$ is a homologically trivial knot, applying the clasp move will decrease the Thurston-Bennequin number by 2.

\subsection{Convex and Weinstein hypersurfaces}\label{sec:subcritically-examples}
Returning to the setting of contact manifolds, one consequence of \cref{thm:moves_sufficient_for_stable_equivalence} is a construction of interesting convex hypersurfaces in contact manifolds.

\highDimCor*

\begin{proof}
By \cref{thm:moves_sufficient_for_stable_equivalence}, these $W$ and $W'$ are related by the moves of \cref{thm:main_moves}.  It follows that the contact handlebody $N(W')$ embeds into $N(W)$ in such a way that the $N(W)-N(W')$ is a bypass cobordism; we then let $\Sigma'=\partial N(W')$.
\end{proof}

While our primary interest in this subsection will be convex hypersurfaces, we observe here that \cref{thm:main_moves} also allows us to recover the following existence $h$-principle for Weinstein hypersurfaces.

\weinsteinHPrinciple*

\begin{proof}
Because $2n\geq 6$, we may apply the stabilizing move of \cref{thm:main_moves} to $W$ to produce a flexible Weinstein hypersurface $W_{\mathrm{flex}}\subset M$ smoothly isotopic to $W$, and the contact handlebody $N(W_{\mathrm{flex}})$ will be contained in $N(W)$.  Similarly, $(W',\lambda',\phi')$ is related to its flexibilization via (de)stabilizing moves, leading to an embedding of the (abstract) contact handlebody $N(W')$ into $N(W'_{\mathrm{flex}})$.  Because $W$ and $W'$ are almost symplectomorphic, $W_{\mathrm{flex}}$ and $W'_{\mathrm{flex}}$ are as well, and thus by the $h$-principle for flexible Weinstein domains \cite[Theorem 14.5]{cieliebak2012stein} we may identify $N(W'_{\mathrm{flex}})$ with $N(W_{\mathrm{flex}})$.  Thus we have an embedding of $N(W')$ into $N(W_{\mathrm{flex}})$, which is in turn embedded into $N(W)$, with $W'$, $W_{\mathrm{flex}}$, and $W$ pairwise smoothly isotopic to one another by \cref{thm:main_moves}.
\end{proof}

\begin{remark}
Notice that the hypotheses of \cref{cor:weinstein-h-principle} cannot be weakened to those of \cref{cor:subcritically-fillable}, since the condition that $W$ and $W'$ become almost symplectomorphic after one stabilization fails to ensure that $W$ and $W'$ are even diffeomorphic. One may then ask whether the conclusion of \cref{cor:weinstein-h-principle} holds under the assumption that $W$ and $W'$ are diffeomorphic and have almost symplectomorphic stabilizations; the authors are unaware of an answer, and thank an anonymous referee for this question.
\end{remark}

The literature contains many pairs of Weinstein domains satisfying the hypotheses of \cref{cor:subcritically-fillable}.  In the remainder of this subsection we highlight a few examples which yield noteworthy convex hypersurfaces. 

\subsubsection{Weinstein Mazur manifolds and convex decompositions of $S^4$}\label{subsec:mazur}
In \cite{mazur1961note}, Mazur described a technique for producing a contractible, compact 4-manifold $M$ which is not diffeomorphic to $D^4$, despite having the property that $M\times I$ is diffeomorphic to $D^5$.  Namely, the \emph{Mazur manifold} $M_K$ is constructed by attaching a 2-handle to $S^1\times D^3$ along a knot $K\subset\partial(S^1\times D^3)\simeq S^1\times S^2$ which represents a homology generator of $H_1(S^1\times S^2;\mathbb{Z})$.  Mazur observed that the 5-manifold $M_K\times I$ is constructed by attaching to $D^5$ a 1-/2-handle pair in algebraically canceling position; because these are 5-dimensional handles, the attaching sphere of the 2-handle can be untied, meaning that the 1-/2-handle pair is genuinely canceling.  So $M_K\times I$ is in fact diffeomorphic to $D^5$.  This construction works just as well in the symplectic category, where we may attach an algebraically cancelling pair of Weinstein handles; an example of such a Weinstein Mazur manifold is seen in \cref{fig:mazur}.  (In fact, while we will not reproduce details here, Hayden-Mark-Piccirillo \cite{hayden2021exotic} and Akbulut-Yildiz \cite{akbulut2019knot} have exhibited exotic pairs of Weinstein Mazur manifolds --- that is, pairs of contractible Weinstein domains which are homeomorphic, but not diffeomorphic.)

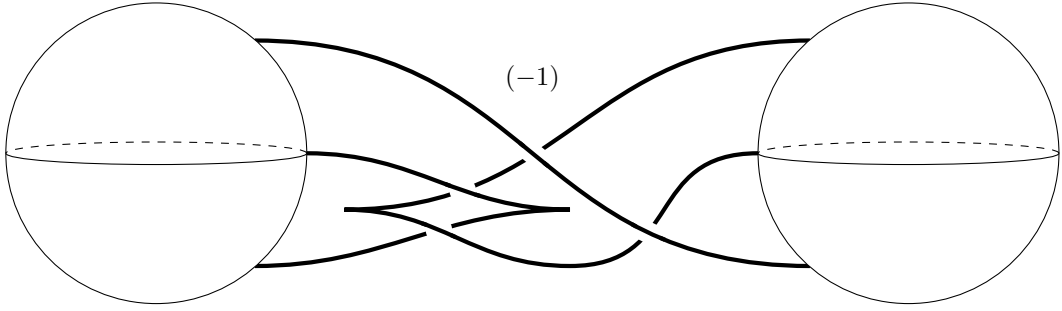
\begin{figure}
    \centering
    \begin{tikzpicture}[xscale=1,yscale=0.5]

\begin{knot}[
    style=thick,
	clip width=5,
	clip radius=5pt,
	ignore endpoint intersections=false,
 	flip crossing/.list={4,6,8,10,12}]
\strand[ultra thick] (-3.69,3) to[out=0,in=180]
	    (3.69,-3);
	    
\strand[ultra thick] (-3,0) to[out=0,in=180]
	    (0.5,-1.5) to[out=180,in=0]
	    (-3.69,-3);
	
\strand[ultra thick] (3.69,3) to[out=180,in=0]
	    (-2.5,-1.5) to[out=0,in=180]
	    (0.5,-3) to[out=0,in=180]
	    (3,0);
\end{knot}

\draw[fill=white] (-5,0) ellipse (2 and 4);
\draw (-7,0) arc (180:360:2 and 0.3);
\draw[dashed] (-3,0) arc (0:180:2 and 0.3);

\draw[fill=white] (5,0) ellipse (2 and 4);
\draw (3,0) arc (180:360:2 and 0.3);
\draw[dashed] (7,0) arc (0:180:2 and 0.3);
\draw (0,2) node {\small $(-1)$};

\end{tikzpicture}
    \caption{A handle diagram for a Weinstein Mazur manifold.}
    \label{fig:mazur}
\end{figure}

Using $(W_L,\lambda_L,\phi_L)$ to denote the Weinstein domain resulting from our construction above, for some Legendrian $L\subset(S^3,\xi_{\mathrm{std}})$, the crossing change move described in \cref{subsec:moves_dim5} identifies a sequence of bypasses which may be attached to $(D^5,\xi_{\mathrm{std}})$ to produce $(W_L\times I,\ker(dt+\lambda_L))$, as well as a complementary sequence of bypasses whose attachment to $(W_L\times I,\ker(dt+\lambda_L))$ produces $(D^5,\xi_{\mathrm{std}})$.  That is, the contact structure $\ker(dt+\lambda_L)$ on $W_L\times I\simeq D^5$ is certainly nonstandard --- as witnessed by the fact the failure of the positive and negative regions of its boundary to be balls --- but \cref{subsec:moves_dim5} gives an explicit sequence of bypasses relating $\ker(dt+\lambda_L)$ to $\xi_{\mathrm{std}}$.

Having produced a nonstandard contact structure $\ker(dt+\lambda_L)$ on $W_L\times I\simeq D^5$, one may hope to produce an exotic contact structure on $S^5$ via doubling --- that is, by writing $S^5$ as $(W_L\times I)\cup -(W_L\times I)$.  However, \cref{thm:moves_for_obds} tells us that the resulting contact structure is standard.  Indeed, let $\xi$ denote the contact structure on $S^5$ obtained in this manner; equivalently, $\xi$ is the contact structure on $S^5$ supported by the open book decomposition with page $W_L$ and identity monodromy.  The crossing change move of \cref{subsec:moves_dim5} may be used to untie the Legendrian knot $L\subset(S^1\times S^2,\xi_{\mathrm{std}})$, and \cref{thm:moves_for_obds} tells us that $\xi$ is supported by an open book decomposition whose page admits a handle decomposition consisting of a single 1-handle and a 2-handle attached along a knot which passes over the 1-handle once; the monodromy of this open book decomposition is the identity map.  Independent of the Legendrian isotopy type of the Legendrian knot, the 1- and 2-handles used to construct the page form a canceling pair, and thus $\xi$ is supported by the abstract open book decomposition $(D^4,\mathrm{id})$ --- that is, $\xi$ is standard.  This construction produces an infinite number of distinct OBDs $(W^4,\mathrm{id})$ for $(S^5,\xi_{\mathrm{std}})$ whose pages are contractible.

\begin{remark}
It may be tempting to generalize this discussion to all dimensions, but \cite[Corollary 6.6]{haefliger1966differentiable} tells us that Legendrian spheres in dimension $2n-1$ cannot be interestingly knotted for $n\geq 3$, and thus an algebraically canceling $(n-1)/n$-handle pair is genuinely canceling.  Therefore, the only Weinstein domain constructed according to the above recipe is $(D^{2n},\lambda_{\mathrm{std}},\phi_{\mathrm{std}})$. 
\end{remark}

The fact that $\mathcal{OB}(W,\mathrm{id}) = (S^5,\xi_{\mathrm{std}})$ when $(W^4,\lambda,\phi)$ is a Weinstein Mazur manifold is seen just as easily using the diagram moves described in \cite[Section 4]{Ding2012Diagrams}.  The novelty of the moves described in \cref{thm:main_moves} is that, according to \cref{cor:subcritically-fillable}, each Weinstein Mazur manifold gives rise to a distinct convex decomposition of $S^4\subset(S^5,\xi_{\mathrm{std}})$, and the moves in \cref{thm:main_moves} produce explicit smooth isotopies between each of these convex spheres.  All of the phenomena observed for Weinstein Mazur manifolds --- e.g., the exotic examples of \cite{hayden2021exotic,akbulut2019knot} --- can thus be witnessed in the convex decompositions of 4-spheres in $(S^5,\xi_{\mathrm{std}})$.

\subsubsection{Plane bundles over $S^2$ and convex decompositions of $S^2\times S^2$}
Upon demonstrating in \cite{Cie02subscritical} that subcritical Weinstein domains are split, Cieliebak gave an infinite family $W^k_m$, $k\in\mathbb{Z}$, $m\in\mathbb{N}$, of Weinstein domains with the property that $W^k_m\times D^2$ is symplectic deformation equivalent to $W^k_{m'}\times D^2$, for any $m,m'\in\mathbb{N}$.  Namely, $(W^k_m,\lambda^k_m,\phi^k_m)$ is the 2-plane bundle over $S^2$ with Euler class $e=-2(k+1+m)\in H^2(S^2;\mathbb{Z})$ and carries a compatible almost complex structure $J^k_m$ with first Chern class $c_1(J^k_m)=2k$; this Weinstein domain may be realized by attaching a Weinstein 2-handle to $(D^4,\lambda_{\mathrm{std}},\phi_{\mathrm{std}})$ along a Legendrian unknot in $(S^3,\xi_{\mathrm{std}})$ with Thurston-Bennequin number $1-2(k+1+m)$ and rotation number $2k$.  See \cref{fig:cieliebak}.  The handle diagrams of $W^k_m$ and $W^k_{m'}$ are then related by $|m-m'|$ applications of the (de)stabilization move described in \cref{prop:stabilizing_bypass}.

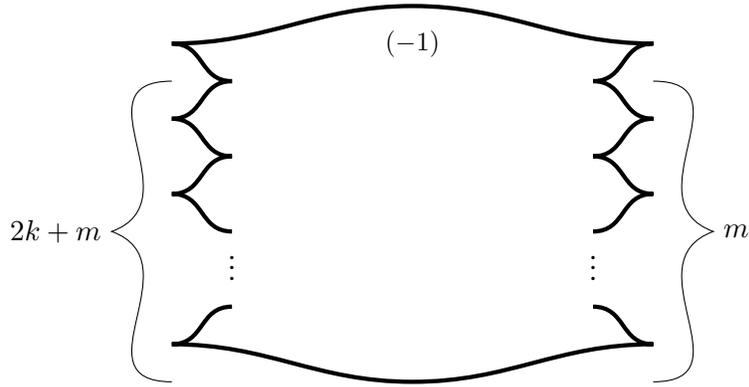
\begin{figure}
    \centering
    \begin{tikzpicture}[xscale=0.8,yscale=0.5]

\begin{knot}
\strand[ultra thick] (-3,4) to[out=180,in=0]
	    (-4,5) to[out=0,in=180]
	    (-3,6) to[out=180,in=0]
	    (-4,7) to[out=0,in=180]
	    (-3,8) to[out=180,in=0]
	    (-4,9) to[out=0,in=180]
	    (0,10) to[out=0,in=180]
	    (4,9) to[out=180,in=0]
	    (3,8) to[out=0,in=180]
	    (4,7) to[out=180,in=0]
	    (3,6) to[out=0,in=180]
	    (4,5) to[out=180,in=0]
	    (3,4);
	    
\strand[ultra thick] (-3,2) to[out=180,in=0]
	    (-4,1) to[out=0,in=180]
	    (0,0) to[out=0,in=180]
	    (4,1) to[out=180,in=0]
	    (3,2);
	    
\strand (4,0) to[out=0,in=210]
	    (5,4) to[out=150,in=0]
	    (4,8);
	    
\strand (-4,0) to[out=180,in=330]
	    (-5,4) to[out=30,in=180]
	    (-4,8);

\end{knot}

\node[right] at (5,4) {$m$};
\node[left] at (-5,4) {$2k+m$};

\node at (-3,3.25) {$\vdots$};
\node at (3,3.25) {$\vdots$};

\node at (0,9) {\small $(-1)$};

\end{tikzpicture}
    \caption{A handle diagram for $W^k_m$.}
    \label{fig:cieliebak}
\end{figure}

From each $W^k_m$ we obtain a contact handlebody $M_k^m = W_k^m \times [-1,1]$ whose rounded boundary is diffeomorphic to $S^2\times S^2$, being the double of a 2-plane bundle over $S^2$ with even Euler class.  However, each $W^k_m$ gives a distinct convex decomposition of $S^2\times S^2$, and for $m\neq m'$ the moves of \cref{thm:main_moves} provide a smooth isotopy between $W^k_m\cup(-W^k_m)$ and $W^k_{m'}\cup(-W^k_{m'})$ inside $M_k^m$.

\subsubsection{Exotic convex decompositions}
In \cite{lazarev2020contact}, Lazarev constructed the first infinite families of almost symplectomorphic Weinstein domains whose contact boundaries are not contactomorphic.  In particular, for any almost contact manifold $(Y^{2n-1},J)$, $n\geq 4$, with $c_1(Y,J)=0$, with almost Weinstein filling $W$, there is an infinite family of distinct contact structures on $Y$ which are pairwise almost contactomorphic, each of which admits a Weinstein filling which is almost symplectomorphic to $W\natural P$, where $P$ is a certain plumbing of $T^*S^n$ depending only on $n$.  See \cite[Theorem 1.14]{lazarev2020contact}.  Because the contact structures are distinct, the Weinstein fillings necessarily are as well.

Now any pair $W_0,W_1$ of these Weinstein fillings are almost symplectomorphic, and thus the same is true of $W_0\times D^2$ and $W_1\times D^2$.  It then follows from \cref{thm:moves_sufficient_for_stable_equivalence} that $W_0$ and $W_1$ are related by the moves described in \cref{thm:main_moves}, and that the contact handlebodies over $W_0$ and $W_1$, respectively, may be obtained from one another by attaching a sequence of bypasses.  In the contact manifold $\partial(W_0\times D^2)\cong\partial(W_1\times D^2)$ we may therefore find smoothly isotopic surfaces $\Sigma_0,\Sigma_1$ with $R_+(\Sigma_j)$ symplectic deformation equivalent to $W_j$, for $j=0,1$.  In this case, unlike the previous examples, $R_+(\Sigma_0)$ and $R_+(\Sigma_1)$ are almost symplectomorphic, and thus form a different sort of exotic pair of convex decompositions.  Note that the dividing sets of $\Sigma_0$ are $\Sigma_1$ are similarly almost contactomorphic --- but not genuinely contactomorphic --- to one another.

These examples generalize a phenomenon we can witness with McLean's exotic Weinstein structures on $D^{2n}$ for $n\geq 4$ \cite{mclean2009lefschetz}.  McLean constructed a Weinstein structure on $D^{2n}$ for every $k\in\mathbb{N}$, and we denote the Weinstein domain by $D^{2n}_k$.  The Weinstein domains $D^{2n}_k$ and $D^{2n}_{k'}$ are not symplectic deformation equivalent for $k\neq k'$, but $D^{2n}_k\times D^2$ and $D^{2n}_{k'}\times D^2$ must be almost symplectomorphic to $D^{2n+2}_{\mathrm{std}}$, since $D^{2n+2}$ has a unique almost symplectic structure.  A combination of \cref{thm:moves_sufficient_for_stable_equivalence} and \cref{thm:main_moves} then tells us that the contact handlebodies over $D^{2n}_k$ and $D^{2n}_{k'}$, respectively, each of which sit inside of $(S^{2n+1},\xi_{\mathrm{std}})$, are related by a sequence of bypass attachments.  As a result, we obtain exotic convex decompositions of $S^{2n}\subset (S^{2n+1},\xi_{\mathrm{std}})$.  That is, we have for each $k\in\mathbb{N}$ a realization of $S^{2n}$ as a convex hypersurface in $(S^{2n+1},\xi_{\mathrm{std}})$ whose positive and negative regions are disks, but distinct $k$ give convex decompositions whose positive regions are not symplectic deformation equivalent to one another. Note that the contactomorphism type of the resulting dividing set is also unknown.

\bibliography{references}
\bibliographystyle{amsalpha}

\end{document}